\documentclass[a4paper,reqno,11pt]{amsart}
\usepackage[margin = 0.9in]{geometry}
\usepackage{amsmath,amsfonts,amssymb,amsthm}
\usepackage[british]{babel}
\usepackage{xcolor}
\usepackage{color}
\usepackage{hyperref}
\usepackage{dsfont}
\usepackage{stmaryrd}
\SetSymbolFont{stmry}{bold}{U}{stmry}{m}{n}
\usepackage{mathtools}
\usepackage{mathrsfs}
\usepackage{comment}
\usepackage[inline]{enumitem}
\usepackage{multicol}
\usepackage{bm}
\usepackage{color}

\renewcommand{\hat}[1]{\widehat{#1}}
\renewcommand{\tilde}[1]{\widetilde{#1}}

\usepackage[colorinlistoftodos,prependcaption,textsize=tiny]{todonotes}

\newcommand{\ext}{{\mathrm{ext}}}

\newcommand{\R}{\mathbb{R}}

\newcommand{\eps}{\varepsilon}
\newcommand{\bbR}{\mathbb{R}}

\renewcommand{\subset}{\subseteq}
\renewcommand{\supset}{\supseteq}
\newcommand{\br}{\color{black}}

\DeclareFontFamily{U}{mathx}{}
\DeclareFontShape{U}{mathx}{m}{n}{<-> mathx10}{}
\DeclareSymbolFont{mathx}{U}{mathx}{m}{n}
\DeclareMathAccent{\widehat}{0}{mathx}{"70}
\DeclareMathAccent{\widecheck}{0}{mathx}{"71}

\newcommand{\vc}[1]{\mathbf{#1}}

\newcommand{\grad}{\nabla}
\newcommand{\del}{\Delta}
\newcommand{\Om}{\Omega}
\newcommand{\dOm}{\partial\Omega}

\newcommand{\gam}{\gamma}
\newcommand{\om}{\omega}

\newcommand{\dr}{\partial_r}
\newcommand{\dph}{\partial_{\ph}}
\newcommand{\dn}{\partial_{\vc{n}}}

\newcommand{\RR}{\mathbb{R}}

\newcommand{\CC}{\mathbb{C}}
\newcommand{\NN}{\mathbb{N}}
\newcommand{\ZZ}{\mathbb{Z}}

\renewcommand{\P}{\mathbb{P}}
\newcommand{\PP}{\mathcal{P}}

\newcommand{\PPu}{\bm{\mathcal{P}}_{\vu}}
\newcommand{\PPbu}{\bm{\mathcal{P}}_{\bar{\vu}}}

\newcommand{\QQu}{\bm{\mathcal{Q}}_{\vu}}
\newcommand{\QQp}{\mathcal{Q}_p}
\newcommand{\QQbu}{\bm{\mathcal{Q}}_{\bar{\vu}}}
\newcommand{\QQbp}{\mathcal{Q}_{\bar{p}}}
\newcommand{\QQf}{\bm{\mathcal{Q}}_{\vf}}
\newcommand{\QQg}{\mathcal{Q}_g}

\newcommand{\PPp}{\mathcal{P}_p}
\newcommand{\PPbp}{\mathcal{P}_{\bar{p}}}
\newcommand{\PPf}{\bm{\mathcal{P}}_{\vf}}
\newcommand{\PPpsi}{\mathcal{P}_{\psi}}
\newcommand{\PPbpsi}{\mathcal{P}_{\bar{\psi}}}

\newcommand{\MM}{\mathcal{M}}
\newcommand{\TT}{C_{\mathrm{c},\sigma}^{\infty}(\overline{\Omega}\setminus\{0\})}%Space of test functions

\newcommand{\ph}{\varphi}
\newcommand{\pht}{\Tilde{\varphi}}
\newcommand{\dd}{\:\mathrm{d}}
\newcommand{\ddrr}{\:\frac{\mathrm{d}r}{r}}
\renewcommand{\Re}{\mathrm{Re}}
\renewcommand{\Im}{\mathrm{Im}}
\renewcommand{\div}{\operatorname{div}}

\newcommand{\curl}{\operatorname{curl}}

\newcommand{\esin}{\vc{\Tilde{e}}}
\newcommand{\ecos}{\vc{e}}

\newcommand{\ur}{u_r}
\newcommand{\uph}{u_{\ph}}
\newcommand{\vr}{v_r}
\newcommand{\vph}{v_{\ph}}
\renewcommand{\wr}{w_r}
\newcommand{\wph}{w_{\ph}}
\newcommand{\fr}{f_r}
\newcommand{\fph}{f_{\ph}}
\newcommand{\vn}{\mathbf{n}}
\newcommand{\vu}{\mathbf{u}}
\newcommand{\vv}{\mathbf{v}}
\newcommand{\ww}{\mathbf{w}}

\newcommand{\vf}{\mathbf{f}}
\newcommand{\vtau}{\bm{\tau}}

\newcommand{\ver}{\mathbf{e}_{\mathrm{r}}}
\newcommand{\veph}{\mathbf{e}_{\ph}}
\newcommand{\half}{\frac{1}{2}}

\newcommand{\nn}{\vc{n}}

\newcommand{\widehatvr}{\widehat{v}_r}
\newcommand{\widehatvph}{\widehat{v}_{\ph}}
\newcommand{\widehatwr}{\widehat{w}_r}
\newcommand{\widehatwph}{\widehat{w}_{\ph}}

\newcommand{\widehatvrk}{\widehat{v}_{rk}}

\numberwithin{equation}{section}
\newcounter{counter}
\numberwithin{counter}{section}
\newtheorem{thm}[counter]{Theorem}
\newtheorem{prop}[counter]{Proposition}
\newtheorem*{prop*}{Proposition}
\newtheorem{definition}[counter]{Definition}
\newtheorem{cor}[counter]{Corollary}
\newtheorem*{cor*}{Corollary}
\newtheorem*{thm*}{Theorem}
\newtheorem*{ex*}{Example}
\newtheorem{lem}[counter]{Lemma}

\theoremstyle{definition}
\newtheorem{rem}[counter]{Remark}

\allowdisplaybreaks

\begin{document}
\title[Well-posedness of the Stokes equations on a wedge with Navier slip]{Well-posedness of the Stokes equations on a wedge with Navier-slip boundary conditions}
\makeatletter
\@namedef{subjclassname@2020}{%
  \textup{2020} Mathematics Subject Classification}
\makeatother

%35J25: Boundary value problems for second-order elliptic equations
%35Q30: Navier-Stokes equations
%76D05: Navier-Stokes equations for incompressible viscous fluids
\thanks{This work is partially based on the master's thesis in applied mathematics of FBR prepared under the advice of MVG at Delft University of Technology. %MVG, HK and JS are grateful to the New York University in Abu Dhabi for its hospitality. 
MB is partially supported by the grant OCENW.M20.194 of the Dutch Research Council (NWO) awarded to MVG. MVG is supported by Vidi grant VI.Vidi.223.019 of the NWO. HK gratefully acknowledges support by Germany’s Excellence Strategy EXC-2181/1 – 390900948. NM is supported by NSF grant DMS-1716466 and by Tamkeen under the NYU Abu Dhabi Research Institute grant of the center SITE. FBR is partially supported by the Vici grant VI.C.212.027 of the NWO. The authors thank the Lorentz Center in Leiden for co-organizing the workshop ``Analysis and numerics of nonlinear PDEs: degeneracies \& free boundaries'', where many useful discussions regarding this project took place. The authors thank respectively their institutions for hosting co-authors at different stages of this project.}
\date{\today}
\author[M. Bravin]{Marco Bravin}
\address[Marco~Bravin]{Departmento de Matem\'atica Aplicada y Ciencias de la Computac\'ion, E.T.S.I. Industriales y de Telecomunicac\'ion, Universidad
de Cantabria, 39005 Santander, Spain. }
\email{marco.bravin@unican.es}
\author[M.V. Gnann]{Manuel V. Gnann}
\address[Manuel~Gnann]{Delft Institute of Applied Mathematics, Faculty of Electrical Engineering, Mathematics and Computer Science, Delft University of Technology, Mekelweg 4, 2628CD Delft, Netherlands}
\email{M.V.Gnann@tudelft.nl}
\author[H. Kn\"upfer]{Hans Kn\"upfer}
\address[Hans~Kn\"upfer]{Institute for Applied Mathematics and IWR, Heidelberg University, INF 205, 69120 Heidelberg, Germany}
\email{knuepfer@uni-heidelberg.de}
\author[N. Masmoudi]{Nader Masmoudi}
\address[Nader~Masmoudi]{NYUAD Research Institute, New York University Abu Dhabi, Saadiyat Island, Abu Dhabi, P.O. Box 129188, United Arab Emirates.\ 
Courant Institute of Mathematical Sciences, New York University, 251 Mercer Street New York, NY 10012 USA}
\email{nm30@nyu.edu}
\author[F.B. Roodenburg]{Floris B. Roodenburg}
\address[Floris~Roodenburg]{Delft Institute of Applied Mathematics, Faculty of Electrical Engineering, Mathematics and Computer Science, Delft University of Technology, Mekelweg 4, 2628CD Delft, Netherlands}
\email{F.B.Roodenburg@tudelft.nl}
\author[J. Sauer]{Jonas Sauer}
\address[Jonas~Sauer]{Institute for Mathematics, Faculty of Mathematics and Computer Science, University of Jena, Ernst-Abbe-Platz 2, 07737 Jena, Germany}
\email{jonas.sauer@uni-jena.de}
\begin{abstract}
We consider the incompressible and stationary Stokes equations on an infinite two-dimensional wedge with non-scaling invariant Navier-slip
boundary conditions. %The problem is decomposed into a singular expansion near the tip of the wedge (polynomial problem) and a regular problem. 
We prove well-posedness and higher
 regularity of the Stokes problem in a certain class of weighted Sobolev spaces. % under a smallness condition on the product of the opening angle of the wedge and the parameter $ \alpha $ associated with the weight.

The novelty of this work is the occurrence of two different scalings in the boundary condition, which is not treated so far for the Stokes system in unbounded wedge-type domains. These difficulties are overcome by first constructing a variational solution in a second order weighted Sobolev space and subsequently proving higher regularity up to the tip of the wedge by employing an iterative scheme. We believe that this method can be used for other problems with variational structure and multiple scales.
\end{abstract}
\maketitle
\setcounter{tocdepth}{1}% To remove all subsections from the TOC
\tableofcontents
\section{Introduction}
In this paper we study well-posedness and regularity of the stationary and incompressible Stokes equations
\begin{equation} \label{sto:sys:int}
  \begin{aligned}
-\nu \Delta\vu + \nabla p = \, & \vf \quad && \text{ in } \Omega,   \\
\div \vu = \, & 0 \quad && \text{ in } \Omega,\\
\vu\cdot \vn = \, & 0 \quad && \text{ on } \partial \Omega',  \\
\vu \cdot \vtau + \beta \partial_{\vn}(\vu \cdot \vtau) = \, & 0 \quad && \text{ on } \partial \Omega',
\end{aligned}
\end{equation}
where $\Om$ is the two-dimensional wedge-shaped domain
$$ \Omega := \{(x,y) \in \mathbb{R}^2: x, y > 0  \text{ and } y < \tan(\theta) x  \}$$
for some opening angle $ \theta\in(0, \pi) $ and we denote $\partial \Omega' :=  \partial \Omega  \setminus \{(0,0)\} $. The body force density $\vf:\Om\to\RR^2$ is given. The unknown functions are the velocity field $ \vu : \Omega \to \bbR^2 $ and the pressure $ p : \Omega \to \bbR $. On the boundary $\dOm'$, we denote by  $\vn $ the outward pointing normal vector and by $\vtau$ the counterclockwise tangent vector. The system \eqref{sto:sys:int} is subject to the no-penetration boundary condition and the Navier-slip boundary condition. Finally, the two parameters $ \nu >0$ and $ \beta>0 $ describe the viscosity of the fluid and the slippage of the fluid on the boundary, respectively.

\subsection{Motivation}
The novelty of this work is the combination of an unbounded wedge-type domain and the Navier-slip boundary condition which is \emph{not scaling invariant}. Well-posedness of \eqref{sto:sys:int} is a first step towards studying free boundary value problems arising from moving droplets. In the case of moving domains with no-slip boundary conditions (corresponding to $\beta=0$ in \eqref{sto:sys:int}), there is infinite energy dissipation at the moving contact line: the so-called no-slip paradox \cite{Dussan_1974, Huh1971}. As is motivated in \cite{Jager2001, Maxwell1878} one could instead consider Navier-slip boundary conditions and this is in fact what we will consider for a static domain.\\

The domain $\Om$ has a conical point at the tip $(0,0)$ and in general regularity results for smooth domains do not hold for non-smooth domains. Nonetheless, there is a vast literature on solving scaling-invariant problems on domains with conical points, see e.g. \cite{KN14, KN16} and the monographs \cite{Kozlov1997, Kozlov2001, Mazya2010, Nazarov1994}. In these works weighted Sobolev spaces are considered which allow for a certain blow up of the solution near the conical point. For gaining higher regularity the solution is decomposed into a polynomial Taylor expansion which captures the singular behavior near the conical point, and a regular remainder. For a wedge, the polynomial problem can be reduced to a system of uncoupled ordinary differential equations (ODEs) in the angle for the coefficients in the expansion using a matching procedure. All these ODEs can be solved explicitly. The far-field contribution of the polynomials are then removed by a cut off. This leads to additional terms on the right hand side of the regular problem, which can then be solved by reducing it to an ODE using the Mellin transform in the radius using polar coordinates.\\

In principle we pursue the same method for our model \eqref{sto:sys:int} with Navier-slip boundary conditions. However, this boundary condition is not scaling invariant (since the normal derivative scales as $r^{-1}$ in polar coordinates) leading to additional difficulties. Namely, the system of ODEs for the coefficients in the expansion is not uncoupled anymore and we do not derive explicit solution representations. Instead, we iteratively solve the coupled system of ODEs and obtain additional contributions of the polynomial problem in the regular problem.

Moreover, the Navier-slip boundary condition complicates the analysis of the regular problem. To construct a solution to the problem, we will use ($L^2$-type) Sobolev spaces with power weights $|\vc{x}|^{-2\alpha}$ for some $\alpha\in \RR\setminus\ZZ$. We cannot directly apply the Lax-Milgram theorem to the bilinear form in the variational formulation of the problem, since the bilinear form is not coercive due to different scalings in the boundary condition. To circumvent this issue, we derive additional bilinear forms involving second order derivatives. We show that this higher-order variational problem has a unique Lax-Milgram solution under the condition that the product of the opening angle $\theta$ and the exponent of the weight $ \alpha $ is sufficiently small. Moreover, we prove that this solution is twice weakly differentiable in the weighted space and satisfies the original partial differential equations. 

Finally, to improve the regularity of the solution we reduce problem \eqref{sto:sys:int} with Navier slip to a problem with an inhomogeneous free-slip boundary condition and use the strong solution as data in the boundary condition. The problem with inhomogeneous free slip is then easier to solve than the original problem with Navier slip, since this boundary condition is scaling invariant.  \\

Regarding the (Navier-)Stokes equations on non-smooth domains, many results are already known.
Note that most of the previous results are concerned with the case of scaling invariant boundary conditions. 
The Navier-Stokes equations on a wedge with no-slip and free-slip (corresponding to $\beta=\infty$ in \eqref{sto:sys:int}) boundary conditions have been studied in \cite{Dauge1989, Kozlov2016, Kozlov2018, Kozlov2020, Mazya2010, Rossmann2018} and \cite{Kohne2021, Maier2014}, respectively. A treatment of the (Navier-)Stokes equations with scaling-invariant boundary conditions in domains with corners can be found in \cite{Kozlov1997, Kozlov2001, Mazya2010}. Furthermore, the (Navier-)Stokes equations have been studied on more regular domains, see e.g. \cite{Galdi2011, Hieber2018, Sohr2001} and references therein.

Results for the free-boundary problem to the Navier-Stokes equations with Laplace’s law at the liquid-gas interface and moving contact line are limited. For $ \frac{\pi}{2} $ wall angle, a reflection technique significantly simplifies the problem. Well-posedness for the time-dependent Navier-Stokes equations in two dimensions with $ \frac{\pi}{2} $ contact angle were treated in \cite{schweizer2001well} and many works on other free-boundary problems followed 
\cite{abels2021well, garcke2022stability, rauchecker2020strong, rauchecker2020well}. The stationary Navier-Stokes equations in three dimensions with $ \pi $ angle and non-moving free boundary were analyzed in \cite{jin2005free}. The stationary Navier-Stokes equations with $ \pi $ angle and assuming a non-moving free boundary were treated in \cite{solonnikov1995some}.

The (Navier-)Stokes equations in two dimensions with dynamic contact angle, see \cite{ren2007boundary, ren2011contact}, and a moving free boundary were treated
in \cite{ Guo2018, guo2023stability, tice2021dynamics, zheng2017local} establishing well-posedness of solutions and stability of the equilibrium state. The methods there are based on nonlinear energy estimates using (weighted) Sobolev spaces.
However, the employed function spaces have corresponding norms that are too weak to control the
singularities expected in pressure, velocity, and shape of the profile close to the contact line, see e.g. \cite{davis1974motion, fricke2019kinematic, huh1971hydrodynamic, shikhmurzaev1997moving}. In this work we are able to consider strong enough spaces to study such singularities in future works.
   
The Stokes and Navier-Stokes equations with free surface and moving contact line share many similarities with the fourth-order degenerate-parabolic thin-film equation
\begin{equation}\label{tfe}
\partial_t h + \nabla \cdot (M(h) \nabla \Delta h) = 0 \qquad \text{in } \{h > 0\},
\end{equation}
describing the film height $h = h(t,x)$ as a function of time $t \ge 0$ and $x \in \R^d$ with $d \in \{1,2\}$. In fact, within a lubrication approximation, \eqref{tfe} with $M(h) = h^3 + \beta h^2$ with slip parameter $\beta > 0$, can be derived from the (Navier-)Stokes equations with free boundary and contact line within a formal lubrication approximation in the regime of small contact angles \cite{beimr.2009,d.1985,odb.1997} using a formal asymptotic expansion and without contact line in  \cite{gunther2008justification}, while the lubrication approximation in case of Darcy dynamics in the Hele-Shaw cell was rigorously carried out in \cite{matioc2012hele} without contact line and in  \cite{go.2003,Knupfer2013,Knupfer2015} with contact line. For \eqref{tfe} a well-posedness and regularity theory for zero and nonzero contact angles has been developed in \cite{bgk.2016,ggko.2014,ggo.2013,GiacomelliKnuepfer2010,gko.2008,g.2015,g.2016,GnannIbrahimMasmoudi2019,GnannWisse2023} and \cite{Knupfer2011,Knupfer2013,Knupfer2015,Knuepfer2015} in one spatial dimension, respectively, while the higher-dimensional case is so far limited to the works \cite{Degtyarev2017,gp.2018,John2015,Seis2018}. There it was found that, except for linear mobilities, solutions are in general not smooth in the distance to the free boundary, even if the leading order is factored off. Our goal is to develop a corresponding theory for the (Navier-)Stokes free boundary problem with contact line and we expect that the methods developed in this note serve as a natural first step towards this goal. Furthermore, a thin-film linearization with two length scales was treated in \cite{GnannIbrahimMasmoudi2019}, where, as in our case it was found that coercivity in the weighted setting requires using higher-order Sobolev norms, see Section \ref{sec:3}.

\subsection{Weighted Sobolev spaces and main results}\label{subsec:mainresults}
In this section we introduce appropriate weighted Sobolev spaces to study well-posedness and regularity for the Stokes equations and we present the main result. Without loss of generality, by rescaling we assume $\nu=\beta=1$ in \eqref{sto:sys:int}. Furthermore, because of the wedge-shaped domain it is natural to consider polar coordinates (see also Appendix \ref{app:polar}). Let $\vu=\ur \ver + \uph \veph$  and $\vf=\fr \ver + \fph \veph $, then \eqref{sto:sys:int} in polar coordinates is given by
\begin{subequations}\label{sys:1}
  \begin{align}
-r^{-2}\big[((r \partial_r)^2 + \partial_{\varphi}^2)u_r - 2 \partial_{\varphi} u_{\varphi} - u_r\big] + \partial_r p = \, & f_r \quad && \text{ for } r > 0, \varphi \in (0,\theta),  \label{sys:1a} \\
-r^{-2}\big[((r\partial_r)^2 + \partial_{\varphi}^2)u_{\varphi} + 2 \partial_{\varphi} u_r - u_{\varphi}\big] + r^{-1}\partial_{\varphi} p = & \, f_{\varphi} \quad && \text{ for } r > 0, \varphi \in (0,\theta),  \label{sys:1b}\\
(r \partial_r +1) u_r + \partial_{\varphi} u_{\varphi} = & \, 0 \quad && \text{ for } r > 0, \varphi \in (0,\theta), \label{sys:1c}  \\
u_{\varphi} = & \, 0  \quad && \text{ for } r > 0, \varphi \in \{0,\theta\},  \label{sys:1d} \\
u_r \pm r^{-1} \partial_{\varphi} u_{r} = & \, 0  \quad && \text{ for } r > 0, \varphi \in \{0,\theta\},  \label{sys:1e}
\end{align}
\end{subequations}
where
$$ \Omega = \{(r\cos(\ph),r\sin(\ph)): r > 0  \text{ and } \ph\in (0,\theta)\}. $$
Moreover, recall that in \eqref{sto:sys:int} $\vn$ is the outward pointing normal vector, so that $\dn$ becomes $\pm r^{-1}\dph$ in polar coordinates where the notation $ \pm $ in \eqref{sys:1e} means $ + $ for $ \varphi = \theta $ and $ - $ for $ \varphi = 0 $.\\

For a vector field $ \vu: \Omega \to \bbR^2 $ such that $ \vu \in C^{\infty}_{\mathrm{c}}(\overline{\Omega}\setminus \{0\}) $ and for a function $ u: \partial \Omega \to \bbR $ such that $ u \in C^{\infty}_{\mathrm{c}}(\partial \Omega') $, we define for any $ \alpha \in \bbR $ and $ k \in \mathbb{N} $ such that $ \alpha + k -1 \neq 0 $ the norms
\begin{align*}
\| \vu \|_{\alpha}^2 &:= \int_{0}^{\theta}\int_0^{\infty} r^{-2\alpha}|\vu|^2 r \dd r \dd \ph,\\
\llbracket  \vu  \rrbracket_{k,\alpha}^2 &:= \sum_{j+\ell = k }  \int_{0}^{\theta}\int_0^{\infty} r^{-2\alpha-2k}\big| (r\partial_r)^{j} \partial_{\varphi}^{\ell} \vu \big|^2 r \dd r \dd \ph \\
| u |_{\alpha}^2 &:= \sum_{\vartheta \in \{0, \theta\}} \int_0^{\infty}r^{-2\alpha} |u(r, \vartheta) |^2 \dd r, \\
[u ]_{k-\frac{1}{2}, \alpha}^2 &:= \inf_{\bar{u} \in C^{\infty}_{\mathrm{c}}(\overline{\Omega}\setminus \{0\}), \bar{u}|_{\partial \Omega'} = u } \llbracket  \bar{u}  \rrbracket_{k,\alpha}^2.
\end{align*}
Moreover, the norm $\|\cdot\|_{\alpha}$ is induced by the weighted inner product $(\cdot,\cdot)_{\alpha}$ on $L^2(\Om, r^{-2\alpha}\mathrm{d} x)$. Denote by  $\prescript{k}{}{\mathcal{H}}^{k}_{\alpha} $ the closure of $ C^{\infty}_{\mathrm{c}}(\overline{\Omega}\setminus \{0\}) $ with respect to the norm $  \llbracket  \cdot  \rrbracket_{k,\alpha} $. The space $\prescript{j}{}{\mathcal{H}}^{k}_{\alpha} $ denotes the closure of $ C^{\infty}_{\mathrm{c}}(\overline{\Omega}\setminus \{0\}) $ with respect to the norm
\begin{equation*}
\| \vu \|_{\prescript{j}{}{\mathcal{H}}^{k}_{\alpha}}^2 := \sum_{\ell=j}^k \llbracket  \vu  \rrbracket_{\ell,\alpha}^2.
\end{equation*}
In the case $ j = 0 $ we write $ \mathcal{H}^{k}_{\alpha} : = \prescript{0}{}{\mathcal{H}}^{k}_{\alpha}$. In a similar fashion, the space $\prescript{j}{}{\mathcal{B}}^{k}_{\alpha} $ denotes the closure of $ C^{\infty}_{\mathrm{c}}(\partial \Omega') $ with respect to the norm
\begin{equation}
\label{norm:on:the:boundary}
| u |_{\prescript{j}{}{\mathcal{B}}^{k}_{\alpha}}^2 := \sum_{\ell=j}^k  [ u ]_{\ell-\frac{1}{2},\alpha}^2.
\end{equation}
Again, for $ j = 0 $ we write $ \mathcal{B}^{k}_{\alpha} : = \prescript{0}{}{\mathcal{B}}^{k}_{\alpha}$.
\begin{rem}
\label{rem:1}
Note that $ \llbracket \cdot \rrbracket_{k,\alpha}  $ with $\alpha + k -1 \neq 0 $ is a norm on $ C^{\infty}_{\mathrm{c}}(\overline{\Omega} \setminus \{0\})$ due to Hardy's inequality (see Lemma \ref{lem:Hardy_on_wedge}), while in general it is only a semi-norm on the space of all locally integrable $  \vu $ such that  $ \llbracket \vu \rrbracket_{k,\alpha}  < \infty $. Moreover, the inclusion
\begin{equation*}
\prescript{k}{}{\mathcal{H}}^{k}_{\alpha} \longrightarrow \left\{ \vu: \llbracket \vu \rrbracket_{k-\ell,\alpha+\ell}  <  \infty   \text{ for all } 0 \leq \ell \leq k            \right\},
\end{equation*}
where the latter space is endowed with  the norm $  \sqrt{\llbracket \vu \rrbracket_{0,\alpha+k}^2 + \dots + \llbracket \vu \rrbracket_{k,\alpha}^2}   $, is a bijective and continuous map. For a proof we refer to Lemma \ref{lemma:norm}.
\end{rem}

 Let $\TT$ be the space of divergence-free test functions with vanishing normal component at the boundary, i.e., we define
\begin{equation}\label{eq:test_function}
        \TT := \big\{\vc{v}\in C_{\mathrm{c}}^{\infty}(\overline{\Om}\setminus\{0\}):  \div\vv=0\text{ in }\Omega,\; v_{\ph}=0\text{ on }\dOm'\big\}.
\end{equation}
For $ k \geq 1 $ and $\alpha\in \RR\setminus\ZZ$, the space $ \mathscr{H}^{k}_{\alpha} $ is the closure of all  $\vu\in  \TT$  with respect to the norm
\begin{equation*}\label{eq:def_scriptH}
\| \vu \|_{\mathscr{H}^k_{\alpha}}^2 := \| \vu \|_{\prescript{1}{}{\mathcal{H}}^{k}_{\alpha}}^2 +  | u_r |_{\alpha}^2 = \sum_{\ell=1}^{k} \llbracket  \vu  \rrbracket_{\ell,\alpha}^2 + |  u_r |_{\alpha}^2. 
\end{equation*}
This is an appropriate space for velocity fields that satisfy \eqref{sto:sys:int}. Namely, the norm contains the terms $ \| \vu \|_{\prescript{1}{}{\mathcal{H}}^{1}_{\alpha}}^2 +  | u_r |_{\alpha}^2 $, which is a weighted energy dissipation, and $ \| \vu \|_{\prescript{2}{}{\mathcal{H}}^{k}_{\alpha}}^2  $ that gives control on higher-order derivatives.

Then, for $ k \geq 0 $ and $\alpha\in\RR\setminus\ZZ$, we introduce  the space $ \mathscr{Z}^{k}_{\alpha} $ as the closure of all  $\vf \in  C^{\infty}_{\mathrm{c}}(\overline{\Omega}\setminus \{0\}) $ with respect to the norm
\begin{equation*}
\|\vf\|_{\mathscr{Z}^{k}_{\alpha}}: = \| \vf\|_{\alpha-1}+ \| \vf\|_{\mathcal{H}^k_{\alpha}}.
\end{equation*}
The term $ \| \vf \|_{\alpha-1} $ allows us to control the weighted energy dissipation of the solution, while $ \| \vf\|_{\mathcal{H}^k_{\alpha}} $ gives control on higher-order derivatives of the solution.

Throughout the rest of the paper we will use a fixed smooth cut-off function $\zeta=\zeta(r) $ satisfying
\begin{equation}\label{eq:cutoff}
  \zeta \in C^{\infty}_{\mathrm{c}}([0,\infty)), \text{ such that } \mathds{1}_{[0,1]} \leq \zeta \leq \mathds{1}_{[0,2]}.
\end{equation}
\begin{rem}
\label{rem:pol:decay}
The spaces $ \mathcal{H}^{k}_{\alpha} $,  $ \mathscr{H}^{k}_{\alpha} $ and $ \mathscr{Z}^{k}_{\alpha} $ with $\alpha\in \RR\setminus\ZZ$ imply a prescribed decay at the tip in the radial direction. For example, for $n \in \mathbb{N}$ we have
\begin{equation}
\label{reg:pol:boundary}
\zeta(r) r^{n} \in \mathcal{H}^{k}_{\alpha} \quad \text{ if and only if } \quad n > k + \alpha  - 1.
\end{equation}
\end{rem}
For a given function $\vf \in C^{\infty}_{\mathrm{c}}(\overline{\Omega}) $, we denote by $ \PPf^n $ the Taylor polynomial of order $ n $ at the tip. For $M\in\NN$ we have $ \vf - \zeta\PPf^n  \in  \mathcal{H}^{M}_{\alpha}  $ if $ n \geq  \lfloor M+ \alpha  -1  \rfloor $, where for any $s\in\RR$ we use the notation
\begin{equation*}
\lfloor s \rfloor :=  \max\{ \ell \in \mathbb{Z} \text{ such that } \ell \leq s\}.
\end{equation*}
We write $ \vf =  \zeta\PPf^n + (\vf - \zeta\PPf^n) $, where we call $ \zeta\PPf^n $ the polynomial part and $ \vf - \zeta\PPf^n \in \mathcal{H}^{M}_{\alpha} $ the regular part. This motivates the definition of the following spaces below. We will denote for $a<b$ by $ H^M(a,b) $ the classical Sobolev spaces of order $M$ on $(a,b)$. For $ k, M \in \mathbb{N} $ we define
\begin{equation*}
\PP_{k,M} := \Big\{ p(r,\varphi) = \sum_{j=0}^{k} a^{(j)}(\varphi)r^j :  a^{(j)} \in H^M(0,\theta) \Big\},
\end{equation*}
endowed with the norm
\begin{equation*}
\| p \|^2_{\PP_{k,M} } := \sum_{j =0}^k\| a^{(j)}\|_{H^M(0,\theta)}^2.
\end{equation*}
Moreover, if $ k < 0 $, we define $ \PP_{k,M} $ to consist of only the zero polynomial.\\

Any solution $ \vu $ to \eqref{sys:1} satisfies the condition $ \div \vu  = 0 $. If we rewrite $ \vu = \zeta \PPu^n + (\vu -  \zeta \PPu^n) $ with $ \div\PPu^n = 0 $, then $ \vu -  \zeta \PPu^n $ is in general not divergence free.  To avoid this issue, we note that $ \PPu^n = \nabla^{\perp} \PP _{\psi}^{n+1}  $ for some polynomial $ \PP _{\psi}^{n+1} $,  where $ \nabla^{\perp} = -r^{-1}\dph \ver + \dr \veph $. We then decompose  $ \vu = \nabla^{\perp}(\zeta \PP_{\psi}^{n+1}) + \vu - \nabla^{\perp}(\zeta \PP_{\psi}^{n+1})  $.  We formalize the above discussion in the following lemma.

\begin{lem}
\label{451}
Let $ k, M \in \mathbb{N} $ and let $ \PPu \in \PP_{k,M} $ such that $ \div\PPu = 0 $. Then there exists a $ \PP_{\psi} \in  \PP_{k+1,M+1} $ with $ \nabla^{\perp} \PP_{\psi} = \PPu  $,  given by the formula
\begin{equation}
\label{45}
 \PPpsi (r, \varphi) = \sum_{j = 0}^{k} \Big( -\int_0^{\varphi} u^{(j)}_{r}(\tilde{\varphi})\dd\tilde{\varphi} \Big) r^{j+1}, \quad \text{ where } \; \PPu = \sum_{j = 0}^{k}  \vu^{(j)}(\varphi) r^{j}.
\end{equation}
Moreover, if $ \PPu \cdot \vn = 0 $ on $ \partial \Omega'$, then  $ \nabla^{\perp}(\zeta \PP_{\psi}) \cdot \vn = 0 $  on $ \partial \Omega'$, where $\zeta$ is as defined in \eqref{eq:cutoff}.
\end{lem}
\begin{proof}
This follows by a direct computation.
\end{proof}

 Given $ \PPu \in \PP_{k,M} $ such that $ \div\PPu = 0 $ with $ k, M \in \mathbb{N} $, we introduce the localized polynomial velocity by
 \begin{equation}
 \label{LPV}
 \QQu (r, \ph)  =  \nabla^{\perp}\left( \zeta(r) \PPpsi(r,\varphi)\right),  \quad \text{ where }    \PPpsi \text{ is associated with } \PPu \text{ by }  \eqref{45}.
 \end{equation}
 In particular, $ \QQu = \PPu $ for $ (r,\varphi) \in (0,1)  \times (0,\theta)$.\\
 
Let $\alpha\in \RR\setminus\ZZ$. For $ M \geq 2 $, we define the space for the velocity $\vu$ by
\begin{align*}
X^M_{\alpha} := \big\{ & \vu: \Omega \to \mathbb{R}^2: \vu = \QQu + \vu_1 \text{ with } \PPu \in  \PP_{n,M},  n= \lfloor  M +  \alpha -1 \rfloor   \text{ and } \vu_1 \in \mathscr{H}^{M}_{\alpha}  \big\},
\end{align*}
endowed with the norm
\begin{equation*}
\| \vu \|_{X^M_{\alpha}}: = \| \PPu \|_{\PP_{n,M} } + \| \vu_1\|_{\mathscr{H}^{M}_{\alpha}}.
\end{equation*}
For $ M \geq 1 $, we define the space for the pressure $p$ by
\begin{align*}
Y^M_{\alpha}: = \big\{  p: \Omega \to \mathbb{R}: p = \zeta \PPp + p_1 \text{ with } \PPp \in  \PP_{n,M},  n = \lfloor  M +  \alpha  -1 \rfloor,&  \PPp(0) = 0  \text{ if } \alpha >  0  \\ & \, \text{ and } p_1 \in \prescript{1}{}{\mathcal{H}}^{M}_{\alpha}   \big\},
\end{align*}
endowed with the norm
\begin{equation*}
\| p \|_{Y^M_{\alpha}} := \| \PPp \|_{\PP_{n,M} } + \| p_1\|_{\prescript{1}{}{\mathcal{H}}^{M}_{\alpha}}.
\end{equation*}
Finally, for $ M \geq 0 $, we define the space for the source term $\vf$ by
\begin{align*}
Z^M_{\alpha} := \big\{ & \vf: \Omega \to \mathbb{R}^2: \vf = \zeta  \PPf + \vf_1 \text{ with }    \PPf \in  \PP_{n,M},  n = \lfloor   M + \alpha - 1 \rfloor  \text{ and } \vf_1 \in \mathscr{Z}^{M}_{\alpha}  \big\},
\end{align*}
endowed with the norm
\begin{equation*}
\| \vf \|_{Z^M_{\alpha}} := \|  \PPf \|_{\PP_{n,M} } + \| \vf_1\|_{\mathscr{Z}^{M}_{\alpha}}.
\end{equation*}

Loosely speaking, to prove regularity of the Stokes equations \eqref{sys:1} it is required to solve the system twice: once with a singular source term at the tip and once with a regular source term. We will refer to those different cases as the \textit{polynomial problem} and the \textit{regular problem}, respectively. A combination of the polynomial and regular problem will lead to our main result on well-posedness and regularity of the Stokes equations on a wedge with Navier slip.

Finally, we expect that the Stokes operator has resonances at $ -\tfrac{\pi}{\theta} +1 $ and  $ \tfrac{\pi}{\theta} - 1 $. We therefore consider the exponent of the weight $ \alpha $ in the interval 
\begin{equation}
\label{int:Ieps}
 I_{\eps} := \left[-(1-\eps) \tfrac{\pi}{\theta} +1, (1-\eps) \tfrac{\pi}{\theta} -1 \right],
\end{equation}
where the small parameter $ \eps > 0  $ is a measure for the distance between $\alpha$ and resonances of the Stokes operator. To ensure that $ I_{\eps} $ is not empty we assume that  $\theta\in (0,(1-\eps)\pi)$. 

We now state the main result.
\begin{thm}[Well-posedness \& regularity]
\label{theo:with:poly} Let $ \eps \in (0,1) $. There is a constant $ c > 0 $ such that for any $ \theta \in (0, (1-\eps)\pi) $ and $ \alpha \in I_{\eps} \setminus \mathbb{Z} $ satisfying $ |\alpha \theta | < c $, for any $ \vf \in Z^M_{\alpha} $ with 
$ M \in \NN$  such that $ M+ \alpha +1 \in I_{\eps} \setminus \mathbb{Z} $, there exists a unique solution  $ (\vu, p) \in X^{M+2}_{\alpha}\times Y_{\alpha}^{M+1} $ to \eqref{sys:1} which satisfies
\begin{equation}
\label{reg:est_mainthm}
\| \vu  \|_{X^{M+2}_{\alpha}} + \| p\|_{Y^{M+1}_{\alpha}}  \leq C_{\alpha, \eps, M} \| \vf \|_{Z^{M}_{\alpha}}.
\end{equation}

\end{thm}

%The novelty of this theorem is the combination of an unbounded wedge-type domain and the Navier-slip boundary condition which is \emph{not scaling invariant}.

\begin{rem} \hspace{2em}
\begin{itemize}
  \item The constant $C$ in \eqref{reg:est_mainthm} depends on $ \alpha $, but not on $\theta$. Throughout the paper, we prove estimates independent of $ \theta $ and any constant is independent of $ \theta $ unless explicitly stated otherwise.  
  \item The pressure in Theorem \ref{theo:with:poly} is uniquely determined as an element of $ Y^{M+1}_{\alpha}$. In fact, the only constant in the space $ Y_{\alpha}^{M+1} $ is the zero constant.
\end{itemize}
\end{rem}

\subsection{Outline}
In Section \ref{sec:prelim} we collect some preliminaries required throughout the paper. In Section \ref{chap:proof_main_result} we gather the results of the regular and polynomial problem to prove the main result Theorem \ref{theo:with:poly}. The construction of a strong solution to the regular problem is explained and carried out in Section \ref{sec:3} and is a consequence of three steps. These three steps are worked out in detail in Section \ref{chap:5}, \ref{chap:proof_coer} and \ref{section:SSSP}, while in Section \ref{sec:Estimates_Helmholtz} we prove the necessary estimates on the Helmholtz projection. Section \ref{sec:higher_reg} deals with the higher regularity of the regular problem and in Section \ref{sec:proof_210} the proof of Proposition \ref{pro:exi:pol} concerning the polynomial problem is given. Appendices \ref{app:polar_vec} and \ref{app:Hardy:traces} contain known results on vector identities in polar coordinates and weighted Sobolev spaces, respectively. Appendix \ref{app:Aux_est} contains some auxiliary estimates required in Section \ref{sec:higher_reg}.

\section{Preliminaries}\label{sec:prelim}
In this section we recall some required tools for proving the main result. This includes Hardy's inequality, the Mellin transform and the Helmholtz projection.
\subsection{Hardy's inequality}
We recall the classical Hardy inequality, see \cite[Theorem 327]{Hardy1952}.
A proof of this inequality on a wedge type domain can for instance be found in \cite[Appendix C]{Guo2018}.
\begin{lem}[Hardy's inequality]\label{lem:Hardy_on_wedge}
 For all $\alpha\neq 0$ and  $u\in C_{\mathrm{c}}^1((0,\infty))$ it holds
\begin{equation*}
  \int_0^{\infty}r^{2\alpha}|u(r)|^2\ddrr \leq \frac{1}{\alpha^2} \int_0^{\infty} r^{2\alpha}|r \dr u(r)|^2\ddrr .
\end{equation*}
\end{lem}
We continue with an improved Hardy type inequality for $\vu\in\TT$, i.e., for divergence-free vector fields on the wedge that are tangent to the boundary, see \eqref{eq:test_function}. Recall that $\Omega$ denotes the wedge with opening angle $\theta>0$.

\begin{lem}[Improved Hardy's inequality for $\vu\in \TT$]\label{lem:new_estimates_uph_ur}
Let $\theta\in(0,\pi)$ and $\alpha \neq 0  $. Then
\begin{equation*}
         \left\| \tfrac{1 }{r} \vu\right\|_{\alpha}^2 \leq C_0(\theta) \theta^2 \|\grad\vu\|^2_{\alpha} \quad \text{ for all }\vu\in\TT,
\end{equation*}
where $C_0:(0,\pi)\to(0,\infty) $ is an increasing function that does not depend on $\alpha$.
\end{lem}

To show the above result we take advantage on the following Poincar\'e estimates with optimal constants.

\begin{lem}[Optimal Poincar\'e's constant {\cite[Section II.5]{Galdi2011}}, {\cite{PayneWeinberger}}]   \label{lem:poincare}
Let $ f \in C^{\infty}_{{\rm c}}((0,\theta)) $ or $ f \in C^{\infty}((0,\theta)) $ such that $ \int_0^{\theta} f \, \dd \varphi= 0 $. Then 
$$ \int_0^{\theta} |f|^2 \, \dd \ph \leq \frac{\theta^2}{\pi^2} \int_0^{\theta} |\dph f|^2 \, \dd \ph. $$
\end{lem}

\begin{proof}[Proof of Lemma \ref{lem:new_estimates_uph_ur}]
Note that $\uph(r,\ph)=0$ for $\ph\in\{0,\theta\}$, so that by Lemma \ref{lem:poincare}
\begin{align*}
    \int_0^{\theta} \uph^2\dd \ph&  \leq \frac{\theta^2}{\pi^2} \int_0^{\theta}\big(\dph\uph)^2\dd \ph
    \leq C_{\eps} \theta^2  \int_0^{\theta}\big(\dph\uph+\ur\big)^2\dd \ph +\left(\tfrac{1}{\pi^2}+\eps \right)\theta^2\int_0^{\theta}\ur^2\dd\ph,
\end{align*}
for some $\eps>0$. Therefore
\begin{equation}
\label{est:uph:lalalland}
 \left\| \frac{\uph}{r} \right\|_{\alpha}^2 \leq C_{\eps} \theta^2 \left\| \nabla \vu \right\|_{\alpha}^2 +  \left(\tfrac{1}{\pi^2}+\eps\right) \theta^2 \left\| \frac{\ur}{r} \right\|_{\alpha}^2.
\end{equation}
From \eqref{sys:1c} we have $ \partial_r(r u_r) = - \partial_{\varphi} u_{\varphi} $ and integrating over $ (0,\theta) $ gives
\begin{equation*}
\partial_r \int _0^{\theta} r u_r \dd \ph = - \int_0^{\theta} \partial_{\varphi} u_{\varphi}  \dd \varphi = 0.
\end{equation*}
The function  $ \int _0^{\theta} r u_r \dd \ph $ is constant in $ r $ and thus $ \int _0^{\theta} r u_r \dd \ph = 0 $ since $ \vu $ has compact support. This implies that $ \int _0^{\theta} r u_r \dd \ph = 0 $ for any $ r > 0$. Therefore, by Lemma \ref{lem:poincare} we obtain
\begin{equation}
\label{est:ur:lalalland:eeeee}
\begin{aligned}
\int_0^{\theta}|u_r|^2 \dd \ph&
\leq \frac{\theta^2}{\pi^2}\int_0^{\theta}|\partial_{\varphi} u_r|^2 \dd \ph \\
&\leq C_{\eps} \theta^2\int_0^{\theta}|\partial_{\varphi} u_r-u_{\varphi}|^2\dd\ph + \left(\tfrac{ 1}{\pi^2}+\eps\right)\theta^2 \int_0^{\theta}| u_{\varphi}|^2\dd\ph.
\end{aligned}
\end{equation}
Adding \eqref{est:uph:lalalland} and $ r^{-2}$\eqref{est:ur:lalalland:eeeee} and  integrating over $ (0,\infty) $ yield
\begin{equation*}
\left\| \tfrac{1}{r}\vu \right\|_{\alpha}^2 \leq C_{\eps} \theta^2 \|\nabla \vu \|_{\alpha}^2 + \left(\tfrac{1}{\pi^2}+\eps\right)\theta^2\left\| \tfrac{1}{r} \vu\right\|_{\alpha}^2.
\end{equation*}
For $ \theta \in(0,\pi) $, there exists $ \eps > 0 $ such that we can absorb $ \left(\frac{1}{\pi^2}+\eps\right)\theta^2\left\| \frac{1}{r}\vu \right\|_{\alpha}^2 $ on the left-hand side. We deduce the result.
\end{proof}

\subsection{The Mellin transform}
\label{Mellin:sec:1}
We collect some properties of the Mellin transform. For more details on the Mellin transform see e.g. \cite{Kozlov2001}.
\begin{definition}
For $f\in C_{\mathrm{c}}^{\infty}((0,\infty))$ the Mellin transform is defined as
\begin{equation*}
    (\MM f)(\lambda)=\widehat{f}(\lambda):=\frac{1}{\sqrt{2\pi}}\int_0^{\infty}r^{-\lambda}f(r)\ddrr,\qquad \lambda \in \mathbb{C}.
\end{equation*}
For $\gam\in\RR$ we define
\begin{equation*}
    f(r)=\frac{1}{\sqrt{2\pi}}\int_{\Re\lambda=\gamma}r^{\lambda} \widehat{f}(\lambda)\dd \Im\lambda, \qquad r\in(0,\infty),
\end{equation*}
which is called the inverse Mellin transform. 
\end{definition}
The definition continues to make sense for $f\in L^1_{\text{loc}}((0,\infty)) $, in which case the integral might fail to converge. If, however, it converges for some $ \lambda_1 $, $ \lambda_2 \in \mathbb{C} $, then it converges on the strip $ \mathcal{S}: = \{ \lambda \in \mathbb{C} $ : $ \Re \lambda_1<\Re\lambda<\Re \lambda_2 \} $ and is analytic on $ \mathcal{S}$.   
Therefore, the inverse transform does not depend on the choice of $\Re \lambda \in( \Re \lambda_1,\Re \lambda_2)$ by Cauchy's integral theorem.

\begin{lem}[{\cite[Lemma 6.1.3]{Kozlov2001}}]\label{lem:Mellin_prop}
For $f,g\in C_{\mathrm{c}}^{\infty}((0,\infty))$ and for any $ \lambda \in \mathbb{C} $, $ n \in \mathbb{N} $ and $ \alpha \in \mathbb{R}$, we have:
\begin{enumerate}[label=(\roman*)]
    \item
    $
        \widehat{r^{-\alpha}f}(\lambda)=\widehat{f}(\lambda+\alpha)
    $,
    \item
    $
        \widehat{\dr^n f}(\lambda)=(\lambda+1)\cdots(\lambda+n)\widehat{f}(\lambda+n),
    $
    \item
    $
        \widehat{(r\dr)^n f}(\lambda)=\lambda^n\widehat{f}(\lambda),
    $ 
    \item  \begin{equation*}
    \begin{aligned}
        \int_0^{\infty}r^{-2\alpha}\overline{f(r)}g(r)\ddrr&=\int_{\Re\lambda=\alpha}\overline{
        \widehat{f}(\lambda)}\widehat{g}(\lambda)\dd \Im \lambda,
    \end{aligned}
\end{equation*}
\item 
\begin{equation*}
    \int_0^{\infty}r^{-2\alpha}|f(r)|^2\ddrr=\int_{\Re\lambda=\alpha}|\widehat{f}(\lambda)|^2\dd \Im \lambda.
\end{equation*}
\end{enumerate}
\end{lem}

Let $\widehat{H}^k_{\alpha}  $ be the closure of $  C^{\infty}_{\mathrm{c}}(\{\lambda \in \mathbb{C} $ : $ \Re \lambda =  \alpha+k -1 \} \times [0,\theta]) $ with respect to the norm
\begin{equation*}
 \| u \|_{\widehat{H}^k_{\alpha}}^2 := \sum_{j+\ell = k} \int_0^{\theta}  \int_{\Re \lambda = \alpha +k-1 }
| \lambda |^{2j} |\partial_{\varphi}^{\ell} u|^2 \, \dd \Im \lambda \dd \varphi.
\end{equation*}
\begin{lem}[{\cite[Lemma 6.1.3]{Kozlov2001}}]
\label{isom:space}
There exists a natural isomorphism $\prescript{j}{}{\mathcal{H}}^{k}_{\alpha}\longrightarrow \widehat{H}^k_{\alpha}  $ with $ u \mapsto \mathcal{M} u $
with inverse that associates to $ \widehat{u} \in \widehat{H}^k_{\alpha} $ the function
\begin{equation*}
 u(r,\varphi) = \frac{1}{\sqrt{2\pi}}\int_{\Re\lambda=\alpha+k -1 }r^{\lambda} \widehat{u}(\lambda, \varphi) \dd \Im\lambda.
\end{equation*}
\end{lem}

\subsection{Helmholtz projection}
\label{chap:4} In order to deal with the pressure $p$ in the Stokes equations, we consider the Helmholtz projection. For $\ww\in C_{\mathrm{c}}^{\infty}(\overline{\Om}\setminus \{0\})$ we study the elliptic problem
\begin{equation}\label{eq:problemPhi}
\begin{aligned}
    \del \Phi &= \div\ww \qquad&&\text{ in }\Om,\\
    \dn \Phi&=\vc{n}\cdot \ww\qquad && \text{ on } \dOm',
\end{aligned}
\end{equation}
%where $\Omega = \{(x,y) \in \mathbb{R}^2: x, y > 0  \text{ and } y < \tan(\theta) x  \} $ and $\dOm'=\dOm\setminus\{(0,0)\}$ on which we can define the outward pointing normal vector $\nn$. 
Note that \eqref{eq:problemPhi} does not have a unique solution in general since $ \Omega $ is unbounded.  We will use the Mellin transform and ODE techniques to uniquely define the Helmholtz projection in our class of weighted spaces. \\

In polar coordinates \eqref{eq:problemPhi} reads
\begin{equation*}
\begin{aligned}
    \big((r\dr)^2+\dph^2\big)\Phi &= r\big((r\dr+1)\wr+\dph\wph\big) \qquad&&\text{ in }\Om,\\
    \dph\Phi&=r\wph\qquad && \text{ on } \dOm'.\\
\end{aligned}
\end{equation*}
Applying the Mellin transform (see Lemma \ref{lem:Mellin_prop}) gives
\begin{equation}\label{eq:problem_Phi_hat}
\begin{aligned}
    \big(\lambda^2+\dph^2\big)\widehat{\Phi}(\lambda,\ph)&=\lambda\widehatwr(\lambda-1,\ph)+\dph\widehatwph(\lambda-1,\ph)=:\widehat{g}(\lambda,\ph)\qquad &&\text{ in }\Om,\\
    \dph\widehat{\Phi}(\lambda,\ph)&=\widehatwph(\lambda-1,\ph)\qquad &&\text{ on }\dOm'.
\end{aligned}
\end{equation}
\begin{prop}\label{prop:Helm_rep_Green}
Let $\Re\lambda\cdot\theta\notin\pi\ZZ$. Then the unique solution to problem \eqref{eq:problem_Phi_hat} is given by
\begin{equation}\label{def:phi:mamo}
\begin{aligned}
    \widehat{\Phi}( \lambda, \ph)    &=\frac{\widehatwph(\lambda-1,0)}{\lambda\sin(\lambda\theta)}\cos(\lambda(\theta-\ph))-\frac{\widehatwph(\lambda-1,\theta)}{\lambda\sin(\lambda\theta)}\cos(\lambda\ph) + \int_0^{\theta}G(\ph, \Tilde{\ph}, \lambda)\widehat{g}(\lambda, \Tilde{\ph})\dd \Tilde{\ph},
    \end{aligned}
\end{equation}
for $\ph\in(0,\theta)$ and where the Green's function is given by
\begin{equation*}
    G(\ph, \Tilde{\ph}, \lambda)=\begin{cases}
   \frac{\cos(\lambda\Tilde{\ph})\cos(\lambda(\theta-\ph))}{\lambda\sin(\lambda\theta)}&\quad\text{ for } 0\leq\Tilde{\ph}<\ph\leq\theta,\\
    \frac{\cos(\lambda(\theta-\Tilde{\ph}))\cos(\lambda\ph)}{\lambda\sin(\lambda\theta)}&\quad\text{ for } 0\leq\ph<\Tilde{\ph}\leq\theta.
    \end{cases}
\end{equation*}
\end{prop}
\begin{proof}
Uniqueness follows by standard ODE results. Thus the formula can be verified a posteriori.
\end{proof}
By properties of the Mellin transform we obtain
\begin{equation*}\label{eq:mellingradPhi}
    \widehat{\grad\Phi}(\lambda,\ph)=\begin{pmatrix}(\lambda+1)\widehat{\Phi}(\lambda+1,\ph)\\\dph\widehat{\Phi}(\lambda+1,\ph)\end{pmatrix}
\end{equation*}
and from the representation for $\widehat{\Phi}$ in Proposition \ref{prop:Helm_rep_Green} we find with integration by parts
\begin{align*}
    \widehat{\Phi}(\lambda, \ph)= \int_0^{\theta}\big(G(\ph,\Tilde{\ph},\lambda)\lambda\widehatwr(\lambda-1,\Tilde{\ph})-\partial_{\Tilde{\ph}}G(\ph,\Tilde{\ph},\lambda)\widehatwph(\lambda-1,\Tilde{\ph})\big)\dd\Tilde{\ph}.
\end{align*}
From the Green's function $G(\ph,\Tilde{\ph},\lambda)$ in Proposition \ref{prop:Helm_rep_Green} it follows that $\lambda\widehat{\Phi}(\lambda,\ph)$ and $\dph\widehat{\Phi}(\lambda,\ph)$
only have singularities at $\lambda=k\pi / \theta $ for $k\in\ZZ\setminus\{0\}$. Hence, for any $ \lambda \in \CC$ such that $\Re \lambda\neq k\pi / \theta  $ for $k\in\ZZ\setminus\{0\}$, we can uniquely define $\grad\Phi$ as the inverse Mellin transform of $\widehat{\grad\Phi}(\lambda,\ph)$ if one integrates over any vertical line such that $\Re\lambda+1$ lies within the interval $\big(-\frac{\pi}{\theta}, \frac{\pi}{\theta}\big)$, i.e., we can integrate over vertical lines $\Re \lambda\in(-\frac{\pi+\theta}{\theta},\frac{\pi-\theta}{\theta})$. Note that as $\theta\downarrow 0$ these singularities move to $ \pm \infty $.
\begin{definition}[Helmholtz projection]\label{def:Helmholtz} Let $\ww\in C_{\mathrm{c}}^{\infty}(\overline{\Om}\setminus \{0\})$ and $ \lambda \in \mathbb{C} $ such that $\Re \lambda\in(-\frac{\pi+\theta}{\theta},\frac{\pi-\theta}{\theta})$. The Helmholtz projection $\P$ is defined by
\begin{equation*}
    \P\ww:= \ww-\nabla \Phi.
\end{equation*}
Here, $\grad\Phi$ is the inverse Mellin transform of $\widehat{\grad\Phi}(\lambda,\ph)$ where the integral is taken over any vertical line $\Re \lambda\in(-\frac{\pi+\theta}{\theta},\frac{\pi-\theta}{\theta})$, i.e.,
\begin{equation*}
    \grad\Phi(r,\ph)=\frac{1}{\sqrt{2\pi}}\int_{\Re\lambda=\gamma}r^{\lambda}\widehat{\grad \Phi}(\lambda,\ph)\dd\Im\lambda\qquad \text{ with }\gamma\in\big(-\tfrac{\pi+\theta}{\theta},\tfrac{\pi-\theta}{\theta}\big).
\end{equation*}
Furthermore, $\widehat{\Phi}(\lambda,\ph)$ is defined by \eqref{def:phi:mamo} and $ \Phi $ solves \eqref{eq:problemPhi}.
\end{definition}

\begin{lem}\label{lem:commute_rdr_P} The Helmholtz projection satisfies the following properties:
\begin{enumerate}[label=(\roman*)]
  \item \label{it:propHelm1} $\P^2=\P$,
  \item \label{it:propHelm2} $\P$ is symmetric (on $ C_{\mathrm{c}}(\overline{\Omega}\setminus \{0\}) $) with respect to $(\cdot,\cdot)_{L^2(\Om)}$,
  \item \label{it:propHelm3} $\P r\dr \ww = r\dr \P\ww$ for $\ww\in C_{\mathrm{c}}^{\infty}(\overline{\Om}\setminus\{0\})$.
\end{enumerate}
\end{lem}
\begin{proof}
It is straightforward to check that \ref{it:propHelm1} and \ref{it:propHelm2} hold. For \ref{it:propHelm3} note that by \eqref{eqapp:com_rel} we have
\begin{equation*}
    r\dr\P\ww = r\dr\ww-\grad(r\dr-1)\Psi\quad\text{ and }\quad \P(r\dr\ww)=r\dr\ww-\grad\Phi,
\end{equation*}
where
\begin{alignat*}{4}
    \begin{cases}
            \del\Psi=\div \ww\qquad&\text{ in }\Om,\\
    \partial_{\nn}\Psi=\nn\cdot\ww\qquad & \text{ on } \dOm',
    \end{cases}
\end{alignat*}
and
\begin{alignat*}{4}
    \begin{cases}
            \del\Phi=\div(r\dr\ww)=(r\dr+1)\div\ww\qquad&\text{ in }\Om,\\
    \partial_{\nn}\Phi=\nn\cdot r\dr\ww\qquad & \text{ on } \dOm'.
    \end{cases}
\end{alignat*}
Both problems have a unique solution up to an additive constant by Definition \ref{def:Helmholtz}. By using \eqref{eqapp:com_rel} again, we learn that $(r\dr-1)\Psi$ satisfies the problem for $\Phi$ and therefore $(r\dr-1)\Psi$ and $\Phi$ are equal up to an additive constant which proves the lemma.
\end{proof}

\section{Proof of the main result}\label{chap:proof_main_result}
We state the results for the regular and polynomial problem in Sections \ref{sec:reg_problem} and \ref{sec:pol_problem}, respectively. In Section \ref{sec:proof_main_result} these results will be combined to prove Theorem \ref{theo:with:poly}.

\subsection{Well-posedness of the Stokes system with non-singular right-hand side}\label{sec:reg_problem}

We state the well-posedness result for the Stokes equations in \eqref{sys:1} with a regular source term and with inhomogeneous Navier-slip boundary conditions. We need to consider inhomogeneous Navier slip to be able to deal with remainder terms coming from the localization of the polynomial problem in Section \ref{sec:pol_problem}. We start by introducing appropriate spaces and norms on $ \partial \Om'$. 

Let $k\geq 0$ and $\alpha\in\RR\setminus\ZZ$. We define the space $ \mathscr{X}^{k}_{\alpha} $ as the closure of $C^{\infty}_{\mathrm{c}}(\partial \Omega') $ with respect to the norm
\begin{equation*}
| g |_{\mathscr{X}^{k}_{\alpha}}^2 := |g|_{\alpha}^2 + | r \partial_r g|_{\alpha}^2 + |g |_{\prescript{1}{}{\mathcal{B}}^{k}_{\alpha}}^2 \stackrel{\eqref{norm:on:the:boundary}}{=} |g|_{\alpha}^2 + | r \partial_r g|_{\alpha}^2 + \sum_{\ell = 1}^{k} [g ]_{\ell - \frac{1}{2}, \alpha}^2.
\end{equation*}

The Stokes equations in \eqref{sys:1} with right-hand side $ \vf \in \mathscr{Z}^{M}_{\alpha} $
and 
$ g \in \mathscr{X}^{M+1}_{\alpha}  $ are
\begin{equation}\label{sto:sys:inhom}
  \begin{aligned}
- \Delta\vu + \nabla p = \, & \vf \quad && \text{ in } \Omega,  \\
\div \vu = \, & 0 \quad && \text{ in } \Omega, \\
\vu\cdot \vn = \, & 0 \quad && \text{ on } \partial \Omega',  \\
\vu \cdot \vtau +  \partial_{\vn}(\vu \cdot \vtau) = \, & g \quad && \text{ on } \partial \Omega'.
\end{aligned}
\end{equation}
Therefore, we introduce for $k\geq 1$ the space
\begin{equation}\label{eq:def_scriptY}
  \mathscr{Y}^k_{\alpha} :=\begin{cases}
    \prescript{1}{}{\mathcal{H}}^{k}_{\alpha} & \mbox{if } \alpha > 0 \text{ or }  k + \alpha - 1 < 0,
    \\
    \big\{  p: \Omega \to \mathbb{R}: p = \zeta p_0 + p_1 \text{ with } p_0 \in \mathbb{R} \text{ and } p_1 \in \prescript{1}{}{\mathcal{H}}^{k}_{\alpha} \big\}  & \mbox{if }   \alpha < 0 \leq k + \alpha - 1,
  \end{cases}
\end{equation}
where the latter space is endowed with the norm $\| p \|_{\mathscr{Y}^k_{\alpha} } := |p_0 | + \| p_1 \|_{\prescript{1}{}{\mathcal{H}}^{k}_{\alpha}}$. The following result holds for the regular problem.
{\br 
\begin{thm}[Elliptic problem]
\label{main:theo}
Let $ \eps \in (0,1) $. There exists a constant $ c > 0 $ such that for any $ \theta \in (0, (1-\eps)\pi) $ and $ \alpha \in I_{\eps} \setminus \mathbb{Z} $ satisfying $ |\alpha \theta | < c $, for any $ \vf \in \mathscr{Z}^M_{\alpha} $ and $  g \in \mathscr{X}^{M+1}_{\alpha} $ with
$ M \in \NN $  such that $ M + \alpha+1 \in I_{\eps}\setminus \mathbb{Z} $,
there exists a unique solution  $ (\vu,p)\in \mathscr{H}^{M+2}_{\alpha}\times \mathscr{Y}^{M+1}_{\alpha} $ to \eqref{sto:sys:inhom} which satisfies
\begin{equation}
\label{reg:est}
\| \vu  \|_{\mathscr{H}^{M+2}_{\alpha}} + \| p  \|_{\mathscr{Y}^{M+1}_{\alpha}} \leq C_{\alpha ,\eps,M} \big(\| \vf \|_{\mathscr{Z}^M_{\alpha}} + | g |_{\mathscr{X}^{M+1}_{\alpha}}\big).
\end{equation}
\end{thm}
}
This theorem is a consequence of two results: existence and uniqueness of the solution in the base case $M=0$ (Theorem \ref{Floris:theo}) and improvement of the regularity to $M\geq 1$ (Proposition \ref{my:step}).

Existence and uniqueness of the solution to the stationary Stokes problem \eqref{sto:sys:inhom} for $M=0$ is not shown in $\mathscr{H}^{2}_{\alpha}\times \mathscr{Y}^{1}_{\alpha} $, but in a space with a norm that is suitably scaled in {\br $|\alpha| $ and $ \theta$}. We use this space to get explicitly stronger information on the solution in the base regularity setting.  

Let $\alpha\in\RR\setminus\ZZ$ and $\theta\in (0,\pi)$. We define the space $ \mathfrak{X}^2_{\alpha,\theta} $ as the closure of all  $\vu\in \TT $ (see \eqref{eq:test_function}) with respect to the norm
\begin{equation}
\label{eq:def_frakX}
\| \vu \|_{\mathfrak{X}^2_{\alpha,\theta}}^2 := |  u_r |_{\alpha}^2 + {\br |\alpha| \theta^3}|r\partial_r u_r|_{\alpha}^2+\llbracket  \vu  \rrbracket_{1,\alpha}^2 +  {\br |\alpha|\theta^3} \llbracket  \vu  \rrbracket_{2,\alpha-1}^2.
\end{equation}
Moreover, $ \mathfrak{Y}^1_{\alpha,\theta} $ is the closure of all  $ p \in  C^{\infty}_{\mathrm{c}}(\overline{\Omega}\setminus \{0\}) $ with respect to the norm
\begin{equation*}
\| p \|_{\mathfrak{Y}^1_{\alpha,\theta}}^2 := {\br |\alpha| \theta^3} \llbracket  p  \rrbracket_{1,\alpha-1}^2.
\end{equation*}
Sections \ref{sec:3}-\ref{section:SSSP} will be concerned with the proof of the following theorem.
{\br
\begin{thm}%Floris Theorem
\label{Floris:theo}
Let $ \eps \in (0,1) $. There is a constant $ c > 0 $ such that for any $ \theta \in (0, (1-\eps)\pi) $ and $ \alpha \in I_{\eps} \setminus \mathbb{Z} $ satisfying $ |\alpha \theta | < c $, for any $ \vf \in \mathcal{H}^0_{\alpha-1} $ and $  g \in \mathscr{X}^{0}_{\alpha} $, there exists a unique solution $ (\vu, p) \in \mathfrak{X}^2_{\alpha,\theta} \times \mathfrak{Y}^1_{\alpha,\theta} $ to \eqref{sto:sys:inhom} which satisfies
\begin{equation*}
\| \vu \|_{\mathfrak{X}^2_{\alpha,\theta}} + \| p \|_{\mathfrak{Y}^1_{\alpha,\theta}}  \leq C_{\alpha}\big(\| \vf\|_{\mathcal{H}^0_{\alpha-1}} +|g|_{\mathscr{X}^0_{\alpha}}\big).
\end{equation*}
In particular, we have
\begin{equation}
\label{first:step:of:ind}
\| \vu \|_{\mathcal{H}^1_{\alpha}} \leq C_{\alpha} \big(\| \vf\|_{\mathcal{H}^0_{\alpha-1}} +|g|_{\mathscr{X}^0_{\alpha}}\big).
\end{equation}
\end{thm}
}
The second result that is required to prove Theorem \ref{main:theo} deals with improving the regularity of the solution using a-priori estimates for the Stokes equations in a wedge with non-homogeneous free-slip boundary conditions. Because we have a strong solution from Theorem \ref{Floris:theo}, we can consider a system with a (scaling-invariant) free-slip boundary condition  instead of the original problem with the Navier-slip condition. Consider the equations
\begin{align}
-r^{-2}[((r \partial_r)^2 + \partial_{\varphi}^2)u_r - \partial_{\varphi} u_{\varphi} - u_r] + \partial_r p = \, & f_r \quad && \text{ for } r > 0, \varphi \in (0,\theta), \nonumber \\
-r^{-2}[((r\partial_r)^2 + \partial_{\varphi}^2)u_{\varphi} + 2 \partial_{\varphi} u_r - u_{\varphi}] + r^{-1}\partial_{\varphi} p = & \, f_{\varphi} \quad && \text{ for } r > 0, \varphi \in (0,\theta), \nonumber \\
(r \partial_r +1) u_r + \partial_{\varphi} u_{\varphi} = & \, 0 \quad && \text{ for } r > 0, \varphi \in (0,\theta),  \label{sys:2} \\
u_{\varphi} = & \, 0  \quad && \text{ for } r > 0, \varphi \in \{0,\theta\}, \nonumber \\
\partial_{\varphi} u_{r} = & \, \mathfrak{g}  \quad && \text{ for } r > 0, \varphi \in \{0,\theta\}, \nonumber
\end{align}
for which the following result holds. The proof of this proposition is given in Section \ref{sec:higher_reg}.
\begin{prop}
\label{my:step}
Let $ \eps \in (0,1) $. There is a constant $ c > 0 $ such that for any $ \theta \in (0, (1-\eps)\pi) $ and $ \alpha \in I_{\eps} \setminus \mathbb{Z} $ satisfying $ |\alpha \theta | < c $ we have:

\begin{enumerate}

\item  If $ \vf \in \prescript{M}{}{\mathcal{H}}^{M}_{\alpha} $ and $ \mathfrak{g} \in  \prescript{M+1}{}{\mathcal{B}}^{M+1}_{\alpha+1} $ with $ M \in\NN$ such that $ M+ \alpha+1 \in I_{\eps}\setminus \mathbb{Z} $, then there exists a unique solution $ (\vu,p) \in \prescript{M+2}{}{\mathcal{H}}^{M+2}_{\alpha} \times  \prescript{M+1}{}{\mathcal{H}}^{M+1}_{\alpha}$ to \eqref{sys:2} which satisfies
\begin{equation}
\label{est:skopel}
\llbracket \vu \rrbracket_{M+2, \alpha} + \llbracket p \rrbracket_{M+1, \alpha}  \leq C_{\alpha,\eps} \left( \llbracket \vf \rrbracket_{M,\alpha} + [ \mathfrak{g} ]_{M+\frac{1}{2}, \alpha +1}\right).
\end{equation}

\item If $ \vu $ is the solution from Theorem \ref{Floris:theo} and if $ \mathfrak{g} = \pm (g - r u_{r})  $. Then the solution of \eqref{sys:2} that satisfies \eqref{est:skopel} coincides with the solution of Theorem \ref{Floris:theo} when $ M = 0 $.\\

\item If $\zeta$ is defined in \eqref{eq:cutoff}, $ \tilde{M} \geq 0 $ is a natural number and 
\begin{equation*}
  \widetilde{M} +\alpha +1 \in I_{\eps} \setminus \mathbb{Z},\quad \vf \in {\mathcal{H}}^{\tilde{M}}_{\alpha} \quad\text{and}\quad \mathfrak{g} \in  \prescript{1}{}{\mathcal{B}}^{\tilde{M}+1}_{\alpha+1}.
\end{equation*}
\begin{enumerate}[label=(\roman*)]
\item If, in addition, $ \widetilde{M} + \alpha  < 0 $ or $ \alpha > 0 $, then there exists a unique solution $ (\vu,p) \in \prescript{1}{}{\mathcal{H}}^{\tilde{M}+2}_{\alpha} \times  \prescript{1}{}{\mathcal{H}}^{\tilde{M}+1}_{\alpha} $
to \eqref{sys:2} that satisfies \eqref{est:skopel} for any $ 0 \leq M \leq \tilde{M }$.

\item  If, in addition, $ \alpha < 0 < \widetilde{M} + \alpha  $, then there exists a unique triple $ (\vu,p_1, p_0) \in \prescript{1}{}{\mathcal{H}}^{\tilde{M}+2}_{\alpha} \times  \prescript{1}{}{\mathcal{H}}^{\tilde{M}+1}_{\alpha} \times \mathbb{R} $ such that $ (\vu, p) = (\vu, \zeta p_0 + p_1 ) $  satisfies  \eqref{sys:2} and  $ (\vu,p_1) $ satisfies \eqref{est:skopel} (with $p$ replaced by $p_1$) for any $ 0 \leq M \leq \widetilde{M }$. Moreover,
 \begin{equation*}
|p_0|  \leq C_{\alpha, \eps,M} \big( \llbracket \vf \rrbracket_{\tilde{M},\alpha} + [ \mathfrak{g} ]_{\tilde{M}+\frac{1}{2}, \alpha+1}\big). 
\end{equation*}
\end{enumerate}
\end{enumerate}
\end{prop}

With the aid of Theorem \ref{Floris:theo} and Proposition \ref{my:step} we can prove Theorem \ref{main:theo}.
\begin{proof}[Proof of Theorem \ref{main:theo}]
For $ M= 0 $, Theorem \ref{Floris:theo} implies the existence of a unique solution $(\vu,p) $ in $ \mathfrak{X}^2_{\alpha,\theta} \times \mathfrak{Y}^1_{\alpha,\theta} $ of \eqref{sto:sys:inhom} which satisfies \eqref{first:step:of:ind}. After noticing that \eqref{sto:sys:inhom} equals \eqref{sys:2} with $ \mathfrak{g} = \pm (rg - r u_r) $, we apply Proposition \ref{my:step} with these values and deduce Theorem \ref{main:theo}. For $ M \geq 1 $ the proof of the regularity estimate \eqref{reg:est}  goes via finite induction. Suppose that \eqref{reg:est} holds for $ M = m$ and suppose that $ m+ 1 + \alpha +1 \in I_{\eps} \setminus \mathbb{Z}$. It then remains to show that  \eqref{reg:est} holds for $M= m +1$.
Again \eqref{sto:sys:inhom} equals \eqref{sys:2} with $ \mathfrak{g} = \pm (rg - r u_r) $, which reads in Mellin variables $ \hat{\mathfrak{g}}(\lambda) =\pm( \hat{g}(\lambda-1) - \hat{u}_r(\lambda-1)) $. Then by definition of the norm $[\,\cdot\,]_{m+3/2, \alpha}^2$ (see Section \ref{subsec:mainresults}), we obtain
\begin{align}
\label{trivial:est}
[ \mathfrak{g} ]_{m+3/2, \alpha+1}^2
\leq  & \, \frac{4}{\min\{1,|m+\alpha+1|\}^{2m+2}} \big([ u_{r} ]_{m+3/2, \alpha}^2 +   [ g ]_{m+3/2, \alpha}^2\big)\leq  C_{\alpha} \big( \llbracket \vu \rrbracket^2_{m+2, \alpha} + [ g ]_{m+3/2, \alpha}^2 \big).
\end{align}
By Proposition \ref{my:step} with $ M = m+1 $ and the above estimate we find
\begin{align*}
\llbracket \vu \rrbracket_{m+3, \alpha} +  \llbracket p_1 \rrbracket_{m+2, \alpha}\leq & \,  C_{\alpha, \eps} \big(  \llbracket \vf \rrbracket_{m+1, \alpha} + [ \mathfrak{g} ]_{m+3/2, \alpha+1}\big) \\
\stackrel{\mathclap{\eqref{trivial:est}}}{\leq} & \,  C_{\alpha, \eps} \big(  \llbracket \vf \rrbracket_{m+1, \alpha} + \llbracket \vu \rrbracket_{m+2, \alpha}+ [ g ]_{m+3/2, \alpha}\big) \\
\leq &  \, C_{\alpha, \eps} \big( \llbracket \vf \rrbracket_{m+1, \alpha}  + \| \vf \|_{\mathcal{H}^{m}_{\alpha}}+  |g|_{\prescript{1}{}{\mathcal{B}}^{m+1}_{\alpha}} + [ g ]_{m+3/2, \alpha} \big) \\
\leq & \,C_{\alpha, \eps} \big(\| \vf \|_{\mathcal{H}^{m+1}_{\alpha}} +  |g|_{\prescript{1}{}{\mathcal{B}}^{m+2}_{\alpha}}\big),
\end{align*}
where in the last step we have used the induction hypothesis
\begin{equation*}
\llbracket \vu \rrbracket_{m+2, \alpha} +  \llbracket p_1 \rrbracket_{m+1, \alpha} \leq C_{\alpha, \eps}\big(\| \vf \|_{\mathcal{H}^{m}_{\alpha}} + |g|_{\prescript{1}{}{\mathcal{B}}^{m+1}_{\alpha}} \big).\qedhere
\end{equation*}
\end{proof}

\subsection{The polynomial problem}\label{sec:pol_problem}
We continue with the polynomial problem, i.e., we consider the Stokes equations \eqref{sys:1} with a polynomial source term of the form
\begin{equation*}
   \PPf^n(r,\ph) = \sum_{j= 0}^{n-2} \vc{f}^{(j)}(\ph) \; r^j,
\end{equation*}
where $ n -2 $ is associated with the degree of the polynomial. Using the ansatz
\begin{equation*}
\PPu(r,\ph) = \sum_{j\geq0} \vu^{(j)}(\ph) \; r^j \quad \text{ and } \quad  \PPp(r,\ph) = \sum_{j\geq 0} p^{(j)}(\ph) \; r^j
\end{equation*}
in \eqref{sys:1} with source term $  \PPf^n(r,\ph) $, it is easy to see that the solution $ \PPu $ is a generalized polynomial and has infinitely many nonzero coefficients $ \vu^{(j)}$ due to the Navier-slip boundary condition. In fact, as we can see from \eqref{sys:1e}, the Navier-slip boundary condition implies a shift in the coefficients of the form $ \partial_{\varphi} u_{r}^{(j)}(\varphi) = \mp u_{r}^{(j-1)}(\varphi) $ for $ \varphi \in \{ 0, \theta\}$. This shift, which is due to the non-scaling invariance of the Navier-slip condition, causes that the theory developed in \cite{Kozlov1997, Kozlov2001, Mazya2010} is not directly applicable.

We truncate $ \PPu $ and $ \PPp $ at order $ n $ and $ n-1 $ respectively and define
 \begin{equation*}
\PPu^n(r,\ph) = \sum_{j=0}^{n} \vu^{(j)}(\ph) \; r^j \quad \text{ and } \quad  \PPp^n(r,\ph) = \sum_{j= 0}^{n-1} p^{(j)}(\ph) \; r^j.
\end{equation*}
 The couple $ (\PPu^n, \PPp^n) $ satisfies the following system with non-homogeneous boundary conditions
 \begin{subequations}\label{sto:sys:pol}
    \begin{alignat}{2}
- \Delta \PPu^n + \nabla \PPp^n = \, & \PPf^n \quad && \text{ in } \Omega, \label{sto:sys:pol1}  \\
\div \PPu^n = \, & 0 \quad && \text{ in } \Omega, \label{sto:sys:pol2} \\
\PPu^n \cdot \vn = \, &  0 \quad && \text{ on } \partial \Omega', \label{sto:sys:pol3} \\
\PPu^n \cdot \vtau +  \partial_{\vn}(\PPu^n \cdot \vtau) = \, & r^{n} \vu^{(n)} \cdot \vtau  \quad && \text{ on } \partial \Omega'.\label{sto:sys:pol4}
\end{alignat}
 \end{subequations}

 \begin{prop}
 \label{pro:exi:pol}

 Let $ \eps \in (0,1)$, $ \theta \in (0, (1-\eps)\pi) $, and $ M, n \in \mathbb{N} $ with $ 2 \leq n \leq (1-\eps)\frac{\pi}{\theta}-1 $. Let
 \begin{equation*}
   \PPf^n(r,\ph) = \sum_{j= 0}^{n-2} \vc{f}^{(j)}(\ph) \; r^j,
\end{equation*}
be such that $  \PPf^n(r,\ph)  \in \PP_{n-2,M} $. 
 Then there exists a unique solution $ (\PPu^{n}, \PPp^n) $ of \eqref{sto:sys:pol}  such that $ \PPu^{n}  \in \PP_{n,M+2} $, $ \PPp^{n} \in   \PP_{n-1,M+1} $ and $ \PPp^{n}(0) = 0 $ . Moreover, %$  \PPp^{n} \in   \PP_{n-1,M+1} $, i.e., $ p^{(-1)}= 0 $ and 
 it holds that
 \begin{equation*}
 \| \PPu^{n} \|_{\PP_{n,M+2}} + \|  \PPp^{n}\|_{\PP_{n-1,M+1}} \leq C_{\eps} \|  \PPf^n(r,\ph) \|_{\PP_{n-2,M}}.
 \end{equation*}
 \end{prop}
The proof of this proposition is given in Section \ref{sec:proof_210}.

\subsubsection*{Localization of the polynomial problem}
 \label{sec:loc:pros} Recall that we decomposed the source term as $ \zeta\PPf^n $ and $ \vf-\zeta\PPf^n $, in such a way that we can apply Proposition \ref{pro:exi:pol} for the polynomial part $ \PPf^n $ and Theorem \ref{main:theo} for the regular part $ \vf-\PPf^n $. The cut-off function $\zeta$ as defined in \eqref{eq:cutoff} was introduced to ensure that $ \vf-\zeta\PPf^n $ has the right behavior as $ r \longrightarrow \infty $.
Subsequently, we introduced the localized polynomial velocity $ \QQu^n $ of $ \PPu^n $ in \eqref{LPV}.
Moreover, the localized polynomial pressure is defined as
\begin{equation*}%\label{eq:defQn}
 \QQp^n(r, \varphi) = \zeta(r)\widetilde{\PPp^n}(r, \varphi) .
\end{equation*}
From \eqref{sto:sys:pol} and Lemma \ref{451}, it is straightforward to verify that
\begin{equation}\label{sto:sys:pol:2}
 \begin{aligned}
- \Delta \QQu^n + \nabla \QQp^n = \, & \zeta \PPf^{n}  + \QQf^n \quad && \text{ in } \Omega, \\
\div\QQu^n = \, & 0 \quad && \text{ in } \Omega, \\
\QQu^n \cdot \vn = \, &  0 \quad && \text{ on } \partial \Omega',  \\
\QQu^n \cdot \vtau +  \partial_{\vn}(\QQu^n \cdot \vtau) = \, &\QQg^n  \quad && \text{ on } \partial \Omega',
\end{aligned}
\end{equation}
where
\begin{equation}\label{def:QQF:QQg}
\begin{aligned}
\QQf^n &= - 2 \nabla \zeta \cdot \nabla \PPu^n  -  \PPu^n \Delta \zeta- \Delta(\PPpsi^n \nabla^{\perp}\zeta  )+ 
\PPp^n \nabla \zeta, \\
\QQg^n &= \zeta  r^{n} \vu^{(n)} \cdot \vtau   + \underbrace{\PPpsi^n}_{=0}\nabla^{\perp}\zeta \cdot \vtau+ \partial_{\vn}(\PPpsi^n \nabla^{\perp} \zeta \cdot \vtau).
\end{aligned}
\end{equation}

\begin{prop}
\label{pro:reg:boundary}
Let $ \eps \in (0,1) $,  $ \theta \in (0, (1-\eps)\pi) $, $M,n \in \mathbb{N}$ with
$ 2 \leq   n \leq (1-\eps)\frac{\pi}{\theta}-1 $. Let $ \PPf^n \in \PP_{n-1,M} $ %and $ \vf^{(-2)} = 0 $. In addition, assume the compatibility condition \eqref{comp:cond:est}.
and let $ \QQf^n $ and $ \QQg^n $ be as in \eqref{def:QQF:QQg}. Then for $ \alpha \in \mathbb{R}\setminus\ZZ $, we have
\begin{equation*}
\|\QQf^n \|_{\mathscr{Z}^{M}_{\alpha}} \leq C_{ \alpha,\eps} \| \PPf^n \|_{\PP_{n-2,M}}
\end{equation*}
and if $ n  > M + \alpha $
, we have
\begin{equation*}
\| \QQg^n\|_{\mathscr{X}^{M+1}_{\alpha}} \leq  C_{ \alpha,\eps} \|  \PPf^n \|_{\PP_{n-2,M}}.
\end{equation*}
\end{prop}

\begin{proof}
From the fact that $ \PPu^n $, $  \PPpsi^n $ are polynomials in $ r $ with coefficients in $ H^{M+2}(0,\theta) $, $  \PPp^n $  
is a polynomial in $ r $ with coefficients in $ H^{M+1}(0,\theta) $ and $ \nabla \zeta \in C^{\infty}_{{\mathrm{c}}}([1,2]) $, it is straightforward to deduce that  $ \|\QQf^n \|_{\mathscr{Z}^{M}_{\alpha}} \leq C \| \PPf^n \|_{\PP_{n-2,M}} $. Regarding $  \QQg^n $, we can argue similarly for $ \partial_{\vn}(\PPpsi^n \nabla^{\perp}\zeta  )\cdot\vtau $. 
To estimate $ \zeta  r^{n} \vu^{(n)} \cdot \vtau $ we use \eqref{reg:pol:boundary} and the restriction $ n > M + \alpha $.
\end{proof}

\subsection{Proof of the main result}\label{sec:proof_main_result}

We conclude with the proof of Theorem \ref{theo:with:poly}.

\begin{proof}[Proof of Theorem \ref{theo:with:poly}]

Let $ \vf \in Z^M_{\alpha} $. If $  \lfloor M+\alpha -1 \rfloor < 0 $, then  $ \vf \in \mathscr{Z}^{M}_{\alpha} $ and Theorem \ref{main:theo} already implies the result. Assume now that $  \lfloor M+\alpha - 1 \rfloor \geq 0 $. Then there exists $ \PPf^n $ with $ n = \lfloor M+\alpha + 1 \rfloor$ and  $ \PPf^n \in \PP_{n-2,M} $ such that  $ \vf = \zeta  \PPf^n  + \vf_1 $ with $ \vf_1 \in \mathscr{Z}^{M}_{\alpha} $. Proposition \ref{pro:exi:pol} ensures that there exists a solution $ (\PPu^n, \PPp^n) $ to \eqref{sto:sys:pol}. Following the localization procedure as in Section \ref{sec:loc:pros}, we obtain a localized polynomial solution $ (  \QQu^n, \QQp^n) $ that satisfies \eqref{sto:sys:pol:2}. A solution $ (\vu,p) $ to the Stokes equations \eqref{sys:1} is decomposed into a localized polynomial part and a regular part:
\begin{equation*}
\vu = \QQu^n +  \vu  - \QQu^n  = \QQu^n + \vu_{\text{reg}} \quad\text{ and }\quad p = \QQp^n+ p - \QQp^n =\QQp^n +p_{\text{reg}},
\end{equation*}
where $ (\vu_{\text{reg}},p_{\text{reg}}) $ solves
\begin{equation}\label{sto:sys:u_0}
  \begin{aligned}
-\Delta\vu_{\text{reg}} + \nabla p_{\text{reg}} = \, & \vf_1 -\QQf^n \quad && \text{ in } \Omega,  \\
\div \vu_{\text{reg}} = \, & 0 \quad && \text{ in } \Omega, \\
\vu_{\text{reg}}\cdot \vn = \, & 0 \quad && \text{ on } \partial \Omega',  \\
\vu_{\text{reg}} \cdot \vtau + \partial_{\vn}(\vu_{\text{reg}} \cdot \vtau) = \, & -\QQg^n \quad && \text{ on } \partial \Omega',
\end{aligned}
\end{equation}
and $\QQf^n, \QQg^n$ are as given in \eqref{def:QQF:QQg}. Proposition \ref{pro:reg:boundary} implies $ \QQf^n, \, \vf_1  \in \mathscr{Z}^{M}_{\alpha} $. Recall that $ n = \lfloor   M + \alpha + 1\rfloor > M + \alpha $. Therefore, by Proposition \ref{pro:reg:boundary} we get  $ \QQg^n \in \mathscr{X}^{M+1}_{\alpha}$. Theorem \ref{main:theo} implies existence of a unique solution $ (\vu_{\text{reg}}, p_{\text{reg}}) \in \mathscr{H}^{M+2}_{\alpha}\times \mathscr{Y}^{M+1}_{\alpha} $  to \eqref{sto:sys:u_0}.

We verify that for the pressure it holds $ p = \QQp^n+p_{\text{reg}} = \zeta \PPp^n  + p_{\text{reg}}   \in Y^{M+1}_{\alpha} $. Recall that by the definition of $ \mathscr{Y}^{M+1}_{\alpha} $ in \eqref{eq:def_scriptY} we have $ p_{\text{reg}} = \zeta p_0 + p_1 $ with $ p_1 \in \prescript{1}{}{\mathcal{H}}^{M+1}_{\alpha} $. Therefore, $ p =\zeta( \PPp^n + p_0) + p_1 \in Y^{M+1}_{\alpha} $.

It remains to show uniqueness. Let $ (\vu, p) = (\QQu +  \vu_{\text{reg}}, \QQp + p_{\text{reg}}) $ as before and let $ (\bar{\vu}, \bar{p}) =(\QQbu +  \bar{\vu}_{\text{reg}}, \QQbp + \bar{p}_{\text{reg}}) $ another solution in $ X^{M+2}_{\alpha} \times Y^{M+1}_{\alpha} $  of \eqref{sys:1} with the same source term $ \vf $. Then 
\begin{align}
- \Delta (\PPu^n -\PPbu^n)+ \nabla (\PPp^n-\PPbp^n)  = \, & \bm{\mathcal{R}_s} \quad && \text{ in } \Omega, \label{o1o} \\
\div (\PPu^n -\PPbu^n) = \, & 0 \quad && \text{ in } \Omega, \nonumber  \\
(\PPu^n -\PPbu^n) \cdot \vn = \, &  0 \quad && \text{ on } \partial \Omega', \nonumber  \\
(\PPu^n -\PPbu^n) \cdot \vtau +  \partial_{\vn}((\PPu^n -\PPbu^n) \cdot \vtau)- r^{n} (\vu^{(n)}- \bar{\vu}^{(n)})\cdot \vtau = \, & \mathcal{R}_b \quad && \text{ on } \partial \Omega'.\label{o2o}
\end{align} 
where 
\begin{align}
\bm{\mathcal{R}_s} = & \, - \vf_1 - \Delta \bar{\vu}_{\text{reg}} + \nabla \bar{p}_{\text{reg}} - (1-\zeta)(- \Delta \PPbu -  \nabla \PPbp + \PPf) - \div(\nabla^{\perp} \zeta \otimes \PPbu) + \Delta(\nabla^{\perp}\zeta \PPbpsi) \label{oo1}, \\
\mathcal{R}_b = & \, \bar{\vu}_{\text{reg}} \cdot \vtau + \partial_{\vn}(\bar{\vu}_{\text{reg}} \cdot \vtau) - (1-\zeta)\PPbu\cdot \vtau -\PPbpsi\cdot \nabla^{\perp}\zeta \cdot \vtau - \partial_{\vn} ((1-\zeta)\PPbu\cdot \vtau -\PPbpsi\cdot \nabla^{\perp}\zeta \cdot \vtau ). \label{oo2}
\end{align}
Notice that $ - \Delta (\PPu^n -\PPbu^n)+ \nabla (\PPp^n-\PPbp^n) \in \PP_{n-2,M} $, we deduce from \eqref{o1o} that $ \bm{\mathcal{R}_s}  \in \PP_{n-2,M} $. From \eqref{oo1} we have that $ \bm{\mathcal{R}_s} \in \mathcal{H}^{M}_{\alpha} $. Using the fact that the only polynomial in $ \PP_{n-2,M}$ which is also in $ \mathcal{H}^{M}_{\alpha} $ is the zero polynomial we deduce that $\bm{\mathcal{R}_s} = 0 $. Similarly using \eqref{o2o} and \eqref{oo2} we also have $\mathcal{R}_b  = 0 $. 

For $ \alpha >  0  $ the definition of $ Y^{M+1}_{\alpha} $ implies that $ \PPp(0) =  \PPbp(0) = 0$, so Proposition \ref{pro:exi:pol} implies that $ (\PPu, \PPp) = (\PPbu, \PPbp) $ and the uniqueness for the regular problem implies that $ (\vu_{\text{reg}}, p_{\text{reg}}) = (\bar{\vu}_{\text{reg}}, \bar{p}_{\text{reg}})  $ as well.      

For $ \alpha <  0  $, we deduce that $ \PPu = \PPbu $ while $ \PPbp = \PPp + \mathfrak{c} $ for some constant $ \mathfrak{c} \in \mathbb{R}$. It remains to show that $ \vu_{\text{reg}} =  \bar{\vu}_{\text{reg}} $ and $ p = \bar{p} $. First we notice that the solutions $ \vu_{\text{reg}} $ and $ \overline{\vu}_{\text{reg}} $ only depend on the Helmholtz projection (see Section \ref{chap:4}) of the source terms and $ \P(\overline{\QQf^n} ) = \P(\QQf^n ) + \P(\nabla(\mathfrak{c}\zeta ) ) =  \P(\QQf^n ) $. We deduce that $ \vu_\text{reg} = \overline{\vu}_\text{reg} $ and thus $\vu = \overline{\vu} $. This implies that the difference $ p_{\text{reg}} - \overline{p}_{\text{reg}}$ satisfies
\begin{equation*}
  \begin{cases}
    -\Delta(p_{\text{reg}} - \overline{p}_{\text{reg}}) = \div(\mathfrak{c} \nabla \zeta )  & \mbox{in } \Om, \\
    \nabla (p_{\text{reg}} - \overline{p}_{\text{reg}}) \cdot \vn = \mathfrak{c} \nabla \zeta  \cdot \vn  = 0 & \mbox{on }\dOm',
  \end{cases}
\end{equation*}
which has a unique solution in $ \mathscr{Y}^{M+1}_{\alpha} $ given by $  p_{\text{reg}} -\overline{p}_{\text{reg}} = \mathfrak{c} \zeta \in \mathscr{Y}^{M+1}_{\alpha} $. Therefore,
\begin{equation*}
\overline{p} = \QQbp^n + \overline{p}_{\text{reg}} = \zeta\left( \PPp^n +\mathfrak{c} \right)+ \overline{p}_\text{reg} =  \zeta  \PPp^n + p_\text{reg} = p.
\end{equation*}
Thus $(\vu,p)=(\overline{\vu},\overline{p})$ and this concludes the proof.
\end{proof}

%%%%%%%%%%%%%%%%%%%%%%%%%%%%%%%%%%%%%%%%%%%%%%%%%%%%%%%%%%%%%%%%%%%%%%%%%%%%%%%%%%%%%%%%%%%%%%%%%%%%%%%%%%%%%%%%%%%%%%%%%%%%%%%%%%%%%%%%%%%%%%%%%%%%%%%%%%%%%%%%%%%%%%%%%%%%

\section{Construction of a strong solution}\label{sec:3}
In this section we prove existence of a unique strong solution to the Stokes equations \eqref{sys:1} with a regular source term as is described in Theorem \ref{Floris:theo}. The proof can be subdivided into three steps which we will be outlined below. All the details will be given in Sections \ref{chap:5}-\ref{section:SSSP}.

\subsection*{Step 1: Variational formulation (Section \ref{chap:5})} To obtain a solution to \eqref{sto:sys:inhom} with a regular source term in a weighted Sobolev space, we employ the Lax-Milgram theorem. After testing the equation with a test function $r^{-2\alpha}\vv$ with $\vv\in\TT$ in the $L^2$ inner product and integration by parts, a variational formulation of the problem is obtained. However, due to the non-scaling invariant Navier-slip boundary condition, it is not clear if the bilinear form in this variational formulation is coercive. To circumvent this issue, we instead test the equation with a test function containing derivatives.

For $\vv\in \TT$, we define
\begin{equation*}
    \vv_{\mathrm{test}}:=  \vv-  {\br |\alpha| \theta^3}(r \dr )^2\vv  -  c_3{\br |\alpha| \theta^3} r^2 \Delta \vv,
\end{equation*}
for some universal $c_3>0$ to be chosen later. Let $\alpha\neq 0$ and $\P$ be the Helmholtz projection from Section \ref{chap:4}.
Testing \eqref{sto:sys:inhom} in $(\cdot,\cdot)_{L^2(\Om)}$ with the test function
$\P(r^{-2\alpha} \vv_{\mathrm{test}})$ leads to the variational problem
\begin{equation}
\label{eq:var_form_ch3}
B(\vu, \vv):=B_1(\vu,\vv)+ {\br |\alpha| \theta^3} B_2(\vu,\vv)+ c_3  {\br |\alpha| \theta^3}  B_3(\vu,\vv) = (\vf,\P(r^{-2\alpha} \vv_{\text{test}}))_{L^2(\Omega)} + \langle g, v_r \rangle_{\alpha},
\end{equation}
where $B_1$, $B_2$ and $B_3$ will be derived in Section \ref{chap:5} and 
\begin{equation*}
\langle g, v_{r} \rangle_{\alpha} := \int_{\dOm'}r^{-2\alpha} g \vr \dd s + {\br |\alpha| \theta^3}  \int_{\dOm'}r^{-2\alpha} \big((r\dr-2\alpha +1) g \big) (r\dr \vr)\dd s.
\end{equation*}
is a pairing that will be derived in Section \ref{sec:var:problem}.
Moreover, the pressure $ p $ is the unique solution in $ \mathfrak{Y}^1_{\alpha,\theta} $ of
\begin{equation}
\label{pres:equ:var}
\nabla p = \Delta \vu + \vf - \P(\Delta \vu + \vf).
\end{equation}
\subsection*{Step 2: Coercivity of the bilinear form (Section \ref{chap:proof_coer})} The bilinear form $B$ in \eqref{eq:var_form_ch3} satisfies the conditions of the Lax-Milgram theorem using the space $\mathfrak{X}^2_{\alpha,\theta}$ as defined in \eqref{eq:def_frakX}. The proofs of the two propositions below are given in Section \ref{chap:proof_coer}.

\begin{prop}[Coercivity]\label{prop:coercivity_ch5} 
Let $ \eps \in (0,1) $. There exist constants $c, c_3> 0 $ such that for any $ \theta \in (0, (1-\eps)\pi) $ and $ \alpha \in I_{\eps} \setminus \mathbb{Z} $ satisfying $ |\alpha \theta | < c $, we have the coercivity estimate
\begin{equation*}
    B(\vu,\vu)=B_1(\vu,\vu)+ |\alpha|\theta^3 B_2(\vu,\vu)+ c_3|\alpha|\theta^3B_3(\vu,\vu)\geq C \|\vu\|^2_{\mathfrak{X}^2_{\alpha,\theta}} \qquad \text{ for all } \vu \in\TT,
\end{equation*}
for some universal constant $ C $.
\end{prop}

\begin{prop}[Boundedness]\label{prop:boundedness_ch5}
Let $ \eps \in (0,1)$. There exist constants $c, c_3 > 0 $ such that for any $ \theta \in (0, (1-\eps)\pi) $ and $ \alpha \in I_{\eps} \setminus \mathbb{Z} $ satisfying $ |\alpha \theta | < c $, we have
\begin{equation*}
    |B(\vu,\vv)|\leq C \|\vu\|_{\mathfrak{X}^2_{\alpha,\theta}}\|\vv\|_{\mathfrak{X}^2_{\alpha,\theta}} \qquad \text{ for all } \vu,\vv\in\TT,
\end{equation*}
for some universal constant $ C $.
\end{prop}

\subsection*{Step 3: Strong solution (Section \ref{section:SSSP})} It remains to prove that the solution to the variational problem \eqref{eq:var_form_ch3} also satisfies the original Stokes equations \eqref{sto:sys:inhom} and is in fact a strong solution. To this end, we need to verify that the set of test functions of the form $\vv_{\text{test}}$ is of sufficiently high resolution in order to apply the fundamental lemma of calculus of variations, i.e., we study the surjectivity of the mapping $\vv\mapsto \vv_{\text{test}}$. At this point it is crucial that all the terms in $\vv_{\text{test}}$ have the same scaling in $r$ and that we did not apply the non-scaling invariant Navier-slip boundary condition to derive the third bilinear form. This will lead to a test function problem that is subject to Dirichlet boundary conditions which is easier to solve than a problem with Navier-slip boundary conditions.\\

The following proposition is proved in Section \ref{section:SSSP}.
{\br
\begin{prop}
\label{equi:of:sol}
Let $ \eps \in (0,1) $. There exists a constant $ c > 0 $ such that for any $ \theta \in (0, (1-\eps)\pi) $ and $ \alpha \in I_{\eps} \setminus \mathbb{Z} $ satisfying $ |\alpha \theta | < c $, for  $c_3>0$ the constant in \eqref{eq:var_form_ch3} and for any $ \vf \in \mathcal{H}^{0}_{\alpha-1} $, $ g \in \mathscr{X}^{0}_{\alpha}$, the following statement holds: $ (\vu,p)\in (\mathfrak{X}^2_{\alpha,\theta}, \mathfrak{Y}^1_{\alpha,\theta} ) $ satisfies \eqref{sto:sys:inhom} almost everywhere, if and only if $ (\vu,p)\in (\mathfrak{X}^2_{\alpha,\theta}, \mathfrak{Y}^1_{\alpha,\theta} )  $ is a variational solution in the sense that they satisfy  \eqref{eq:var_form_ch3} for all $\vv\in\TT $ and \eqref{pres:equ:var}.
\end{prop}
}

By combining the results from the above steps we can finish the proof of Theorem \ref{Floris:theo}.
\begin{proof}[Proof of Theorem \ref{Floris:theo}]
For any $ \vf \in \mathcal{H}^0_{\alpha-1} $ and $ g \in \mathscr{X}^0_{\alpha}$ there exists a unique $ \vu \in \mathfrak{X}_{\alpha,\theta}^2 $ such that
\begin{equation*}
B(\vu, \vv) = (\vf, \P(r^{-2\alpha}\vv_{\text{test}}))_{L^2(\Omega)} + \langle g, v_r\rangle_{\alpha} \qquad \text{ for all } \vv \in \TT.
\end{equation*}
This follows immediately by the Lax-Milgram theorem using Propositions \ref{prop:coercivity_ch5}, \ref{prop:boundedness_ch5}, density of $\TT$ in $\mathfrak{X}^2_{\alpha,\theta}$ (by definition), Proposition \ref{prop:densH^2alpha} and the estimates
\begin{equation*}
|  ( \vf,\P(r^{-2\alpha} \vv_{\text{test}}))_{L^2(\Omega)} | \leq C \|  \vf \|_{\mathcal{H}^0_{\alpha-1}} \| \vv\|_{\mathfrak{X}^2_{\alpha,\theta}} \quad \text{ and } \quad |\langle g, v_r\rangle_{\alpha} | \leq \| g \|_{\mathscr{X}^0_{\alpha}}\|\vv\|_{\mathfrak{X}^2_{\alpha,\theta}}.
\end{equation*}
Finally, Proposition \ref{equi:of:sol} ensures that the solution $(\vu,p)\in \mathfrak{X}^2_{\alpha,\theta}\times \mathfrak{Y}^1_{\alpha,\theta}$ (where  $p$ is the solution to \eqref{pres:equ:var}), satisfies the Stokes problem \eqref{sto:sys:inhom}.
\end{proof}

 %%%%%%%%%%%%%%%%%%%%%%%%%%%%%%%%%%%%%%%%%%%%%%%%%%%%%%%%%%%%%%%%%%%%%%%%

\section{Estimates on the Helmholtz projection}\label{sec:Estimates_Helmholtz}
In this section we prepare for proving coercivity of the bilinear forms $ B_1 $, $ B_2$ and $ B_3$ in Section \ref{chap:proof_coer}. We derive estimates for the commutators $[\P, \Delta] $  and  $[\P, r^{-2\alpha}]$ which will appear in the bilinear forms $ B_1 $, $ B_2$ and $ B_3$ in Section \ref{chap:5}. We recall that the commutator of two operators $A_1$ and $A_2$ is given by $[A_1,A_2]=A_1A_2-A_2A_1$.

\subsection{Estimates on the commutator $[\P, \Delta]$}
The commutator $[\P, \Delta]$ can have singularities at
 $ \pm \frac{\pi}{\theta}  $. Therefore, we define for $\eps \in (0, 1)$ and $\theta\in(0,(1-\eps)\pi)$ the interval
 $$ \mathfrak{I}_{\eps} := \left[ -(1-\eps)\tfrac{\pi}{\theta},(1-\eps)\tfrac{\pi}{\theta} \right] $$
and we show uniform estimates for $\Re \lambda\in \mathfrak{I}_\eps $.

We start with the following lemma in the case that $\vv $ is divergence free, but $ \vv \cdot \nn \neq 0$ on $\dOm'$.
\begin{lem}\label{lem:est_P_divw0}
Let $\eps \in (0, 1 ) $,  $ \theta \in (0, (1-\eps)\pi) $ and let $ \vv \in C^{\infty}_{\mathrm{c}}(\overline{\Omega}\setminus \{0\})$ be such that $ \div \vv =0 $ in $\Om$. Then for $ \alpha \in \mathfrak{I}_{\eps} \setminus \{0\} $, we have
\begin{equation*}
	\|\P \vv \|_{\alpha} \leq C_{\eps} \|\vv \|_{\alpha}.
\end{equation*}
\end{lem}	

\begin{proof}
To prove the estimate in the statement of the lemma, it suffices, by definition of $ \P $, to estimate
\begin{equation*}
\|\nabla \Phi \|^2_{\alpha} = \int_0^{\theta} \int_{\Re \lambda = \alpha } |\lambda|^2|\hat{\Phi}(\lambda,\varphi)|^2+ |\partial_{\varphi} \hat{\Phi}(\lambda,\varphi)|^2 \,  \dd \Im \lambda \dd \varphi,
\end{equation*}
where by Proposition \ref{prop:Helm_rep_Green} we have
\begin{equation*}
\hat{\Phi}(\lambda,\varphi) = \frac{\hat{v}_{\varphi}(\lambda-1,0)}{\lambda \sin(\lambda \theta)} \cos(\lambda(\theta-\varphi))- \frac{\hat{v}_{\varphi}(\lambda-1,\theta)}{\lambda \sin(\lambda \theta)} \cos(\lambda \varphi).
\end{equation*}
We only consider the estimates for $ \| \partial_r \Phi \|^2_{\alpha} $ since the estimates for $ \| \tfrac{1}{r}\partial_{\varphi} \Phi \|^2_{\alpha} $ are similar.

We start with some preliminary computations.
Write $ \lambda = \alpha + i t $, so that by an elementary computation we obtain for $ \vartheta \in \{ \varphi, \theta-\varphi\} $
\begin{align*}\label{eq:sincos_h}
\left|\frac{\cos(\lambda \vartheta)}{\sin(\lambda \theta)}\right|^2 = \frac{\cos^2(\alpha \vartheta) \cosh^2(t\vartheta) + \sin^2(\alpha \vartheta)\sinh^2(t\vartheta)}{\cos^2(\alpha \theta) \sinh^2(t\theta) + \sin^2(\alpha \theta)\cosh^2(t\theta)}.
\end{align*}
To bound the above expression, we notice that
\begin{equation}\label{eq:D} | \cos^2(\alpha \theta)| \geq \frac{1}{4} \quad \text{ if } \quad \alpha \theta \in \bigcup_{ k \in \mathbb{Z}}\left( - \tfrac{\pi}{4}+k\pi, \tfrac{\pi}{4}+k\pi \right) =: D,
\end{equation}
while $ | \sin^2(\alpha \theta)| \geq \frac{1}{4} $ if $ \alpha \theta \in \mathbb{R} \setminus D $. Moreover, for $x>0$ we have $  \cosh^2(tx) \sim e^{2|t|x}   $, while $ \sinh^2(tx) \sim e^{2|t|x}   $  if $ |tx | \geq \tfrac{1}{4} $.
We deduce that
\begin{equation*}
\left|\frac{\cos(\lambda \vartheta)}{\sin(\lambda \theta)}\right|^2 \leq \begin{cases} \frac{\cos^2(\alpha \vartheta) \cosh^2(t\vartheta) + \sin^2(\alpha \vartheta)\sinh^2(t\vartheta)}{\cos^2(\alpha \theta) \sinh^2(t\theta) } \leq C e^{-2|t|(\theta-\vartheta)} \quad \text{ if } \alpha \theta \in D \text{ and } |t\theta|\geq 1/4 , \\
\frac{\cos^2(\alpha \vartheta) \cosh^2(t\vartheta) + \sin^2(\alpha \vartheta)\sinh^2(t\vartheta)}{ \sin^2(\alpha \theta)\cosh^2(t\theta)} \leq C e^{-2|t|(\theta-\vartheta)} \quad \text{ if } \alpha \theta \in \mathbb{R} \setminus D.  \end{cases}
\end{equation*}
Therefore, if  $ (\alpha \theta  ,t \theta ) \notin  D \times (-\tfrac{1}{4},\tfrac{1}{4}) $, we have
\begin{equation}
\label{01}
	\int_0^{\theta} \left|\frac{\cos(\lambda \vartheta)}{\sin(\lambda \theta)}\right|^2  \dd \varphi \leq \int_0^{\theta} e^{-2|t|(\theta-\vartheta)}  \dd \varphi \leq C \frac{1- e^{-2|t|\theta}}{|t|},
\end{equation}
which can be bounded by
\begin{equation}
\label{02}
\frac{1- e^{-2|t|\theta}}{|t|} \leq C\frac{(|\alpha|\theta +1 )}{|\lambda|} \quad \text{ and } \quad \frac{1- e^{-2|t|\theta}}{|t|} \leq C \theta.
\end{equation}

To continue, we show a trace-type estimate for divergence-free vector fields in the wedge. Let $ \eta \in C^{\infty}_{\mathrm{c}}((-1,1)) $ be a symmetric cut-off function such that $\eta(0) = 1$ and let $ \vartheta \in \{ 0, \theta \}$. By the fundamental theorem of calculus and the divergence-free condition, we have
\begin{align}
|\hat{v}_{\varphi}(\lambda, \vartheta)|^2 = & \; -\int_0^{\theta} \partial_{\varphi}\Big(\eta\big(\tfrac{\varphi - \vartheta }{\theta}\big) |\hat{v}_{\ph}(\lambda, \varphi)|^2\Big) \dd \varphi\nonumber\\
=&\;  -\int_0^{\theta} \theta^{-1}\eta'\big(\tfrac{\varphi - \vartheta}{\theta}\big) |\hat{v}_{\ph}(\lambda, \varphi)|^2  \dd \varphi \nonumber  - 2 \Re\int_0^{\theta} \eta \big(\tfrac{\varphi- \vartheta}{\theta}\big) \partial_{\varphi}\hat{v}_{\ph}(\lambda, \varphi)\overline{\hat{v}_{\ph}(\lambda, \varphi)}  \dd \varphi  \nonumber  \\
= & \;  -\int_0^{\theta}\theta^{-1} \eta'\big(\tfrac{\varphi- \vartheta}{\theta}\big) |\hat{v}_{\ph}(\lambda, \varphi)|^2  \dd \varphi - 2 \Re\int_0^{\theta} \eta \big(\tfrac{\varphi- \vartheta}{\theta}\big) (\lambda +1)\hat{v}_{r}(\lambda, \varphi)\overline{\hat{v}(\lambda, \varphi)}  \dd \varphi \nonumber   \\
= & \; (\theta^{-1} + |\lambda+1|) \int_0^{\theta} |\hat{\vv}(\lambda, \varphi)|^2  \dd \varphi. \label{03}
\end{align}

We now return to the bound of $  \|\partial_r \Phi \|_{\alpha} $.
For $ \alpha \theta \in \mathbb{R} \setminus D $, we bound using \eqref{01}, \eqref{02} and \eqref{03}
\begin{equation}\label{04}
  \begin{aligned}
\|\partial_r \Phi \|^2_{\alpha} = & \; \int_0^{\theta} \int_{\Re \lambda = \alpha } \bigg|\frac{\hat{v}_{\varphi}(\lambda-1,0)}{ \sin(\lambda \theta)} \cos(\lambda(\theta-\varphi)) - \frac{\hat{v}_{\varphi}(\lambda-1,\theta)}{ \sin(\lambda \theta)} \cos(\lambda \varphi)\bigg|^2  \dd \varphi  \dd \Im \lambda\\
\leq&\; C(1+|\alpha|\theta)\|\vv\|_{\alpha}^2 \leq C\|\vv\|_{\alpha}^2,
\end{aligned}
\end{equation}
where in the last step we have used that $ |\alpha|\theta \leq \pi $ for $ \alpha \in \mathfrak{J}_{\eps} $.

For $ \alpha \theta \in D $, we have to consider the cases $ |t\theta | \geq 1/4 $ and $ | t\theta | < 1/4 $. In the case $ |t\theta | \geq 1/4 $ we can follow the same strategy as for $ \alpha \theta \in \mathbb{R} \setminus D $ and deduce \eqref{04}. In the case $ |t\theta  | < 1/4 $, we are allowed to replace the trigonometric functions with their Taylor expansions to deduce the desired bound. For  $ \alpha \in \mathfrak{I}_{\eps} \setminus \{0\}$ it holds that
\begin{equation*}
|\sin(\lambda \theta)| \sim_{\eps} |\lambda \theta|, \quad  \quad |1-\cos(\lambda(\theta-\varphi))| \sim_{\eps} \Big|\frac{(\lambda(\theta-\varphi))^2}{2}\Big|,
\end{equation*}
and
\begin{align*}
& \;\frac{\hat{v}_{\varphi}(\lambda-1,0)}{ \sin(\lambda \theta)} \cos(\lambda(\theta-\varphi)) - \frac{\hat{v}_{\varphi}(\lambda-1,\theta)}{ \sin(\lambda \theta)} \cos(\lambda \varphi) \\ = & \;  \frac{\hat{v}_{\varphi}(\lambda-1,0)}{ \sin(\lambda \theta)} (\cos(\lambda(\theta-\varphi))-1) +
\frac{\hat{v}_{\varphi}(\lambda-1,0)-\hat{v}_{\varphi}(\lambda-1,\theta)}{ \sin(\lambda \theta)}
 - \frac{\hat{v}_{\varphi}(\lambda-1,\theta)}{ \sin(\lambda \theta)} (\cos(\lambda \varphi)-1)\\
=&\!\!\!\: :  \;I_1+I_2+I_3.
\end{align*}
To bound $ \|\partial_r \Phi \|_{\alpha}^2 $, we estimate the above terms separately. 
The estimates of $ I_1 $ and $ I_3 $ are trivial using \eqref{03}. To estimate $I_2$ note that
\begin{align*}
 \hat{v}_{\varphi}(\lambda-1,0) - \hat{v}_{\varphi}(\lambda-1,\theta) = -\int_0^{\theta} \partial_{\varphi}\hat{v}_{\varphi}(\lambda-1,\varphi)  \dd \varphi = \int_0^{\theta} \lambda\hat{v}_{r}(\lambda-1,\varphi)  \dd \varphi,
\end{align*}
since $\div\vv=0$.
We obtain
\begin{align*}
& \; \int_0^{\theta} \int_{\Re \lambda = \alpha,\,   |t\theta| < \tfrac{1}{4} } \left| \frac{\hat{v}_{\varphi}(\lambda-1,0)-\hat{v}_{\varphi}(\lambda-1,\theta)}{ \sin(\lambda \theta)} \right|^2  	\dd \Im \lambda \dd \varphi \\ = & \; \int_0^{\theta} \int_{\Re \lambda = \alpha,\,  |t\theta | < \tfrac{1}{4}} \bigg| \frac{ \int_0^{\theta} \lambda\hat{v}_{r}(\lambda-1,\varphi')  \dd \varphi'}{ \sin(\lambda \theta)} \bigg|^2  	\dd \Im \lambda \dd \varphi \\
\leq & \; C_{\eps} \int_{\Re \lambda = \alpha  ,\, |t\theta | <\tfrac{1}{4}} \frac{ |\lambda|^2 \theta^2 \| \hat{\vv} (\lambda-1,\cdot) \|_{L^2(0,\theta)}^2}{|\lambda|^2\theta^2}  	\dd \Im \lambda   \leq C_{\eps} \| \vv \|_{\alpha}^2.\qedhere
\end{align*}

\end{proof}

We continue with the estimate on the commutator of the Helmholtz projection and the Laplacian. Note that $\P \del \vv$ for $\vv\in\TT$ is well defined by Lemma \ref{lem:est_P_divw0} using that $\div \del\vv = 0$ (see Lemma \ref{lem:app_curl}).
\begin{lem}\label{lem:est_P_delta_w}
	Let $\eps \in (0, 1) $ and $ \theta \in (0, (1-\eps)\pi) $.
 Then for  $  \alpha - 1 \in \mathfrak{I}_{\eps} \setminus \{0\}  $ and $\vv\in\TT$, we have
\begin{align*}
	\| [\P, \Delta] \vv \|_{\alpha-1}^2 \leq  & \, C_{\eps}\big(\|\nabla (r \dr \vv) \|_{\alpha}^2 + \|\nabla  \vv \|_{\alpha}^2 + \left\| \tfrac{v_{\ph}}{r} \right\|_{\alpha}^2 \big) \\
	& \,  + C_{\eps}\| \tfrac{1}{r} \dph^2 \vv  \|_{\alpha}\big(\|\nabla (r \dr \vv) \|_{\alpha} + \|\nabla  \vv \|_{\alpha} +  \left\| \tfrac{v_{\ph}}{r} \right\|_{\alpha}\big) \\ &  \, +  \tfrac{C}{\theta^2}\| \tfrac{1}{r} \dph v_r \|_{\alpha}^2 + \tfrac{C}{\theta}\| \tfrac{1}{r} \dph^2 v_r \|_{\alpha}\| \tfrac{1}{r} \dph v_r \|_{\alpha}.
\end{align*}
\end{lem}	

\begin{proof} Note that $\P \vv =\vv $, so that $ [\P, \Delta] \vv = -\grad \Phi$, where $\Phi$ is the potential in the definition of $\P \del \vv$. It suffices to estimate
\begin{equation*}
\|\nabla \Phi \|^2_{\alpha-1} = \int_0^{\theta} \int_{\Re \lambda = \alpha-1} |\lambda|^2|\hat{\Phi}(\lambda,\varphi)|^2+ |\partial_{\varphi} \hat{\Phi}(\lambda,\varphi)|^2 \,  \dd \Im \lambda \dd \varphi
\end{equation*}
where by Proposition \ref{prop:Helm_rep_Green} and Lemma \ref{lem:app_curl}(iii) we have
\begin{equation}
\label{eq:to:use:to:def}
\hat{\Phi}(\lambda,\varphi) = \frac{\hat{(\Delta \vv)_{\varphi} }(\lambda-1,0)}{\lambda \sin(\lambda \theta)} \cos(\lambda(\theta-\varphi))- \frac{\hat{(\Delta \vv)_{\varphi}}(\lambda-1,\theta)}{\lambda \sin(\lambda \theta)} \cos(\lambda \varphi).
\end{equation}
Using that $ \div \vv = 0 $ and $ v_{\ph} = 0 $ on $ \partial \Omega' $, we have on the boundary
\begin{align*}
    (\del \vv)_{\ph}\;\;&\stackrel{\mathclap{\eqref{eqapp:del_polar}}}{=}\;\; r^{-2}\left[((r\dr)^2+\dph^2) v_{\varphi}+2\dph v_r-v_{\varphi} \right] = -r^{-2}(r\dr-1)\dph v_r.
\end{align*}
In Mellin variables, the above expression rewrites
\begin{align}
\label{eq:to:use:to:est}
\hat{(\Delta \vv)_{\varphi}}(\lambda,\vartheta ) = -(\lambda+1)\dph \hat{v}_r(\lambda+2,\vartheta),
\end{align}
where  $ \vartheta \in \{ \varphi, \theta-\varphi\} $.
In this proof, we only show the estimates for $ \| \partial_r \Phi \|_{\alpha-1} $, the estimates for $ \|\frac{1}{r} \partial_{\varphi} \Phi \|_{\alpha-1} $ are derived similarly. 

To show the result we use the following trace type estimate. Let $ \eta \in C^{\infty}_{\mathrm{c}}((-1,1)) $ be a symmetric cut-off function such that $\eta(0) = 1$ and let $ \vartheta \in \{ 0, \theta \}$. By the fundamental theorem of calculus and the Cauchy-Schwarz inequality we have the following estimate
  \begin{align}
|\dph \hat{v}_{r}(\lambda, \vartheta)|^2 = & \; -\int_0^{\theta} \partial_{\varphi}\Big(\eta\big(\tfrac{\varphi- \vartheta}{\theta}\big) |\dph \hat{v}_r|^2(\lambda, \varphi)\Big)  \dd \varphi\nonumber \\
=&\; -\int_0^{\theta} \theta^{-1} \eta'\big(\tfrac{\varphi- \vartheta}{\theta}\big) |\dph\hat{v}_r|^2(\lambda, \varphi)  \dd \varphi  - 2 \Re\int_0^{\theta} \eta \big(\tfrac{\varphi- \vartheta}{\theta}\big) \partial_{\varphi}^2\hat{v}_r(\lambda, \varphi)\overline{\dph \hat{v}_r(\lambda, \varphi)}  \dd \varphi \label{1.1} \\
\lesssim  & \; \frac{1}{\theta} \int_0^{\theta} |\dph \hat{v}_r(\lambda, \varphi)|^2 \, \dd \varphi + \Big(\int_0^{\theta} |\dph^2 \hat{v}_r(\lambda, \varphi)|^2\dd \ph\Big)^{\frac{1}{2}}\Big(\int_0^{\theta} |\dph \hat{v}_r(\lambda, \varphi)|^2 \dd\ph\Big)^{\frac{1}{2}}.\nonumber
\end{align}

We now go back to the bound of $ \|\partial_r \Phi \|_{\alpha-1} $. Write $ \lambda = \alpha-1 +  it  $ with $t\in\RR$. Let $D$ be as in \eqref{eq:D}, then for $ |\alpha-1|\theta \in \RR\setminus D$ \eqref{01} holds. Using \eqref{eq:to:use:to:def} and \eqref{eq:to:use:to:est}, we estimate
\begin{align}
\|\partial_r \Phi \|^2_{\alpha-1} = & \; \int_0^{\theta} \int_{\Re \lambda = \alpha-1} \bigg|\frac{\lambda \dph \hat{v}_{r}(\lambda+1,0)}{ \sin(\lambda \theta)} \cos(\lambda(\theta-\varphi)) - \frac{\lambda \dph \hat{v}_{r}(\lambda+1,\theta)}{ \sin(\lambda \theta)} \cos(\lambda \varphi)\bigg|^2  \dd \Im \lambda  \dd \varphi \nonumber \\
\stackrel{\mathclap{\eqref{01}}}{\lesssim}&\; \int_{\Re \lambda = \alpha-1}|\lambda|^2\left(|\dph \hat{v}_r(\lambda+1,0)|^2+ |\dph \hat{v}_r(\lambda+1,\theta)|^2 \right) \frac{1-e^{-2|t|\theta}}{2|t|}  \dd \Im \lambda \nonumber \\
\stackrel{\mathclap{\eqref{02},\eqref{1.1}}}{\lesssim}&\;\;\;\quad(1+|\alpha|\theta)\bigg[\int_{\Re \lambda = \alpha-1}\int_0^{\theta}|\lambda|^2 |\dph \hat{v}_r(\lambda+1,\ph)|^2 \dd \ph  \dd \Im \lambda \nonumber \\
& \; +\int_{\Re \lambda = \alpha-1}\Big(\int_0^{\theta}|\lambda|^2 |\dph \hat{v}_r(\lambda+1,\ph)|^2\dd \ph \Big)^{\frac{1}{2}}\Big(\int_0^{\theta} |\dph^2 \hat{v}_r(\lambda+1,\ph)|^2 \dd \ph \Big)^{\frac{1}{2}} \dd \Im \lambda\bigg] \nonumber \\
  \lesssim &\; \|\nabla (r \dr \vv) \|_{\alpha}^2 + \|\nabla  \vv \|_{\alpha}^2 +\left\| \frac{v_{\ph}}{r} \right\|_{\alpha}^2   + \| \tfrac{1}{r} \dph^2 \vv  \|_{\alpha}\Big(\|\nabla (r \dr \vv) \|_{\alpha} + \|\nabla  \vv \|_{\alpha} + \left\| \frac{v_{\ph}}{r} \right\|_{\alpha} \Big). \label{1.4}
\end{align}
When $ |\alpha-1|\theta \in D $, we have to consider two cases $ |\theta t| \geq 1/4 $ and $ | \theta t| < 1/4 $. In the first case we can follow the strategy above and deduce the bound \eqref{1.4}. For $ |\theta  t| < 1/4 $, we replace the trigonometric functions with their Taylor expansions to deduce the desired bound. For $ \alpha - 1  \in \mathfrak{I}_{\eps} \setminus \{0\} $ we have
\begin{equation*}
|\sin(\lambda \theta)| \sim_{\eps} |\lambda \theta| \quad \text{ and } \quad |\cos(\lambda(\theta-\varphi))| \sim_{\eps} 1.
\end{equation*}
Hence
\begin{align*}
\int_0^{\theta} \int_{\Re \lambda = \alpha-1, \,|t \theta| < \frac{1}{4}}  & \; \bigg|\frac{\lambda \dph \hat{v}_{r}(\lambda+1,0)}{ \sin(\lambda \theta)} \cos(\lambda(\theta-\varphi))  - \frac{\lambda \dph \hat{v}_{r}(\lambda+1,\theta)}{ \sin(\lambda \theta)} \cos(\lambda \varphi)\bigg|^2   \dd \Im \lambda \dd \varphi  \\
\leq C_{\eps}   & \;  \int_0^{\theta} \int_{\Re \lambda = \alpha-1, \,|t \theta| < \frac{1}{4}} \frac{|\lambda|^2  \dph \hat{v}_{r}(\lambda+1,0)|^2}{ |\lambda \theta|^2} + \frac{|\lambda|^2| \dph \hat{v}_{r}(\lambda+1,\theta)|^2}{ |\lambda \theta|^2}   \dd \Im \lambda\dd \varphi  \\
\stackrel{\mathclap{\eqref{1.1}} }{\leq} & \;\;\; C_{\eps} \tfrac{1}{\theta^2}\big\| \tfrac{1}{r} \dph v_r \big\|_{\alpha}^2 + \tfrac{1}{\theta}\big\| \tfrac{1}{r} \dph^2 v_r \big\|_{\alpha}\big\| \tfrac{1}{r} \dph v_r \big\|_{\alpha}.\qedhere
\end{align*}
\end{proof}

\subsection{Estimates on the commutator $[\P, r^{-2\alpha}]$}\label{subsec:Helm_rep_Fourier}
In view of the fact that $\P$ is symmetric on unweighted spaces (Lemma \ref{lem:commute_rdr_P}), we will encounter the commutator $[\P, r^{-2\alpha}]\vv$ with $\alpha\neq 0$ and $\vv\in\TT$ in Section \ref{chap:5}. Note that $[\P, r^{-2\alpha}]\vv = - \nabla \Phi $ where $\Phi$ (in the sense of Definition \ref{def:Helmholtz}) satisfies
\begin{alignat*}{2}
    \del\Phi&=\div r^{-2\alpha}\vv=-2\alpha r^{-2\alpha-1}\vr \qquad&&\text{ in }\Om,\\
    \partial_{\nn}\Phi&=\nn\cdot r^{-2\alpha}\vv=0\qquad && \text{ on } \dOm',
\end{alignat*}
which can be written more conveniently in polar coordinates as
\begin{subequations}\label{eq:ProblemPhiFourier}
\begin{alignat}{2}
    \left((r\dr)^2+\dph^2\right)\Phi&=-2\alpha r^{-2\alpha+1}\vr=:g \qquad&&\text{ in }\Om,\label{eq:ProblemPhiFourier1}\\
    \dph\Phi&=0\qquad && \text{ on } \dOm'.\label{eq:ProblemPhiFourier2}
\end{alignat}
\end{subequations}
We will derive estimates on $\Phi$ using the Fourier expansion of $\widehat{\Phi}$ in the angle. The Fourier expansions are in this case easier to work with than the Green's function representation from Proposition \ref{prop:Helm_rep_Green}.

For $k\in\NN$ we define the orthonormal systems
\begin{equation}
\label{cos:sin:base}
    \ecos_k(\ph):=\begin{cases}\frac{1}{\sqrt{\theta}}&\; k=0,\\\sqrt{\frac{2}{\theta}}\cos\left(\frac{k\pi\ph}{\theta}\right)&\; k\geq 1,\end{cases}\quad \text{ and }\quad
    \esin_k(\ph):=\sqrt{\tfrac{2}{\theta}}\sin\left(\tfrac{k\pi\ph}{\theta}\right)\quad k\geq 1,
\end{equation}
which satisfy
\begin{equation*}
    \int_0^{\theta}\ecos_k(\ph)\ecos_{\ell}(\ph)\dd \ph= \int_0^{\theta}\esin_k(\ph)\esin_{\ell}(\ph)\dd \ph =\begin{cases}1&\; \text{if } k=\ell,\\0&\; \text{else}.\end{cases}
\end{equation*}
Moreover, an $L^2$-function $g:(0,\theta)\to \RR$ admits Fourier expansions of the form
\begin{align*}
  g(\ph)&=\sum_{k=0}^{\infty}g_k \ecos_k(\ph)\quad \text{ in }L^2(0,\theta),\quad \text{ where } g_k=\int_0^{\theta}g(\pht)\ecos_k(\pht)\dd \pht,\\
  g(\ph)&=\sum_{k=0}^{\infty}g_k \esin_k(\ph)\quad \text{ in }L^2(0,\theta),\quad \text{ where } g_k=\int_0^{\theta}g(\pht)\esin_k(\pht)\dd \pht.
\end{align*}
In addition, for the coefficients $g_k$ in both Fourier expansions we have Bessel's identity given by
\begin{equation}\label{eq:Bessel_Identity}
    \sum_{k=1}^{\infty}|g_k|^2=\int_0^{\theta}|g(\ph)|^2\dd \ph.
\end{equation}

\begin{lem} \label{lem:Fourier_Rep_Phi}
Let $\vv\in\TT$. The Mellin transform of the solution $\Phi$ of problem \eqref{eq:ProblemPhiFourier} has the form
\begin{equation*}\label{eq:repPhi_1}
    \widehat{\Phi}(\lambda,\ph)=-2\alpha\sum_{k=1}^{\infty}\frac{\widehatvrk(\lambda+2\alpha-1)}{\lambda^2-\left(\frac{k\pi}{\theta}\right)^2}\ecos_k(\ph)\qquad \text{ for }\ph\in[0,\theta],\end{equation*}
and its derivative is given by
\begin{equation}\label{eq:repPhi_2}
    \dph\widehat{\Phi}(\lambda,\ph)=2\alpha\sum_{k=1}^{\infty}\frac{k\pi}{\theta}\cdot\frac{\widehatvrk(\lambda+2\alpha-1)}{\lambda^2-\left(\frac{k\pi}{\theta}\right)^2}\esin_k(\ph)\qquad \text{ for }\ph\in[0,\theta],
\end{equation}
with $\Re\lambda\in\big(-\frac{\pi}{\theta},\frac{\pi}{\theta}\big)$ and $\lambda\neq-2\alpha$.
\end{lem}
\begin{proof}
Taking the Mellin transform of \eqref{eq:ProblemPhiFourier} gives
\begin{subequations}\label{eq:Phi_Fourier_Mellin}
\begin{alignat}{2}
    (\lambda^2+\dph^2)\widehat{\Phi}(\lambda,\ph)&=\widehat{g}(\lambda,\ph)  \qquad&&\text{ in }\Om,\label{eq:Phi_Fourier_Mellina}\\
    \dph \widehat{\Phi}(\lambda,\ph)&=0 \qquad && \text{ on } \dOm',
\end{alignat}
\end{subequations}
which is a non-homogeneous second-order ODE with homogeneous Neumann boundary conditions. The function $\widehat{g}(\lambda,\ph)$ admits a Fourier expansion in the angle $\ph$ of the form
\begin{equation*}
        \widehat{g}(\lambda,\ph)=\sum_{k=0}^{\infty}\widehat{g}_k(\lambda)\ecos_k(\ph)\quad \text{ in }L^2(0,\theta),
\end{equation*}
where the Fourier coefficients are given by
\begin{equation}\label{eq:Fourier_coef_g}
    \widehat{g}_k(\lambda)=-2\alpha\widehatvrk(\lambda+2\alpha-1)
\end{equation}
with
\begin{equation*}
   \widehatvrk(\lambda)=\int_0^{\theta}\widehat{v}_r(\lambda,\pht)\ecos_k(\pht)\dd\pht \quad\text{ and }\quad  \widehat{v}_r(\lambda,\ph)=\sum_{k=0}^{\infty}\widehatvrk(\lambda)\ecos_k(\ph)\quad \text{ in }L^2(0,\theta).
\end{equation*}
The condition $\div\vv=0$ reads in Mellin variables $(\lambda+1)\widehatvr(\lambda,\ph)+\dph\widehatvph(\lambda,\ph)=0$. Integrating this expression over the angle yields
\begin{equation}
\label{no:k:zero}
    (\lambda+1)\int_0^{\theta}\widehatvr(\lambda,\ph)\dd \ph=-\int_0^{\theta}\dph\widehatvph(\lambda,\ph)\dd\ph=\widehatvph(\lambda,0)-\widehatvph(\lambda,\theta)=0,
\end{equation}
by the boundary condition $\vph=0$ on $\dOm'$.

Problem \eqref{eq:Phi_Fourier_Mellin} has a series solution of the form
\begin{equation*}
    \widehat{\Phi}(\lambda,\ph)=\sum_{k=0}^{\infty}\widehat{\Phi}_k(\lambda) \ecos_k(\ph) \quad \text{almost everywhere},
\end{equation*}
 and inserting this into \eqref{eq:Phi_Fourier_Mellina} and using the orthogonality of the cosines gives that the coefficients satisfy
\begin{equation*}
    \Big(\lambda^2-\big(\tfrac{k\pi}{\theta}\big)^2\Big)\widehat{\Phi}_k(\lambda)=\widehat{g}_k(\lambda).
\end{equation*}
By \eqref{eq:Fourier_coef_g} and \eqref{no:k:zero} this leads to the following series representation
\begin{equation*}
    \widehat{\Phi}(\lambda,\ph)=-2\alpha\sum_{k=1}^{\infty}\frac{\widehatvrk(\lambda+2\alpha-1)}{\lambda^2-\left(\tfrac{k\pi}{\theta}\right)^2}\ecos_k(\ph).\end{equation*}
Note that
\begin{equation*}
    \sum_{k=1}^{\infty}|k^2\widehat{\Phi}_k(\lambda)|^2<\infty
\end{equation*}
and therefore the series converges in $H^2(0,\theta)$ which embeds into $C^1([0,\theta])$. This implies that the series converges pointwise and the series can be differentiated to obtain \eqref{eq:repPhi_2}
which again converges in $H^1(0,\theta)$ and therefore $\dph\widehat{\Phi}$ also converges pointwise.
\end{proof}

The commutator $ [\P, r^{-2\alpha}] $ may have singularities at $ - \tfrac{\pi}{\theta} + 1  $ and $ \tfrac{\pi}{\theta} -1  $. Therefore, recall from \eqref{int:Ieps} that for $\eps \in (0, 1 )$ and $\theta\in(0,(1-\eps)\pi)$ we defined the interval
 $$ I_{\eps} = [- (1-\eps)\tfrac{\pi}{\theta} + 1, (1-\eps)\tfrac{\pi}{\theta} - 1].$$
 The following lemma provides estimates on the commutator $ [\P, r^{-2\alpha}] $ required in Section \ref{chap:proof_coer}.
\begin{lem} \label{lem:estimate_Fourier_Phi}
Let  $\eps \in (0, 1) $ and $ \theta \in (0, (1-\eps)\pi) $. Assume that $\alpha\in I_{\eps} \setminus \{0\}$. Then we have the estimates
\begin{equation*}
  \begin{aligned}
\big\| \tfrac{1}{r}[\P, r^{-2 \alpha}] \vv \big\|_{-\alpha} &\leq C_{\eps} |\alpha| \theta  \left\| \vv \right\|_{\alpha+1}, 
\\
\big\| \nabla [\P, r^{-2 \alpha}] \vv \big\|_{-\alpha} &\leq C_{\eps} |\alpha| \left\|\vv \right\|_{\alpha+1}. 
\end{aligned}
\end{equation*}
Moreover, we have
\begin{equation}
\label{ezelina}
([\P, r^{-2\alpha}]\vv)_{\varphi} = 0 \quad \text{ on } \partial \Omega'.
\end{equation}

\end{lem}

\begin{proof}
First of all note that for any $ \tilde{\lambda} = \alpha + i t $ with $t\in\RR$, we have
\begin{align*}
 \Big|(\tilde{\lambda}- 2 \alpha +1 )^2 -\left(\tfrac{k \pi}{\theta}\right)^2\Big|^2 = \big(t^2+\left(\tfrac{k \pi}{\theta}\right)^2-(1-\alpha)^2\big)^2 + 4t^2(\alpha-1)^2.
\end{align*}
If $ \Re \tilde{\lambda} = \alpha \in I_{\eps} $, then
\begin{equation}
\label{est:laredo}
|\alpha-1 | \leq (1-\eps) \frac{\pi}{\theta} \quad \text{ and } \quad |\alpha | \leq (1-\eps) \frac{\pi}{\theta},
\end{equation}
and therefore using \eqref{est:laredo}, we deduce that
\begin{equation}\label{eq:estHelm1}
  \begin{aligned}
\left|(\tilde{\lambda}- 2 \alpha +1 )^2 -\left(\tfrac{k \pi}{\theta}\right)^2\right|^2 \geq & \, \big(t^2+\left(\tfrac{k \pi}{\theta}\right)^2-(1-\alpha)^2\big)^2 \\
\geq & \, \big(t^2 + (1-(1-\eps)^2) \left(\tfrac{k \pi}{\theta}\right)^2\big)^2\\
\geq &\, \big(t^2 + \eps(1-\eps) \left(\tfrac{k \pi}{\theta}\right)^2\big)^2 \geq c_{\eps}(\tfrac{k \pi}{\theta})^4.
\end{aligned}
\end{equation}
Using again \eqref{est:laredo} in the previous estimate yields
\begin{equation}\label{eq:estHelm2}
  \begin{aligned}
\left|(\tilde{\lambda}- 2 \alpha +1 )^2 -\left(\tfrac{k \pi}{\theta}\right)^2\right|^2 \geq  & \,\big(t^2 + \eps(1-\eps) \left(\tfrac{k \pi}{\theta}\right)^2\big)^2\\ \geq &\,  \big(t^2 + \eps \left(\alpha-1\right)^2\big)^2 \geq c_{\eps}|\tilde{\lambda}-1|^4,
\end{aligned}
\end{equation}
and similarly
\begin{equation}\label{eq:estHelm3}
\left|(\tilde{\lambda}- 2 \alpha +1 )^2 -\left(\tfrac{k \pi}{\theta}\right)^2\right|^2 \geq  \, \left(t^2 + \eps\alpha^2\right)^2 \geq c_{\eps}|\tilde{\lambda}|^4.
\end{equation}

Recall that $ [\P, r^{-2\alpha}]\vv = - \nabla \Phi $, where $ \Phi $ is defined in Lemma \ref{lem:Fourier_Rep_Phi}. By properties of the Mellin transform, \eqref{eq:Bessel_Identity} and Lemma \ref{lem:Fourier_Rep_Phi}, we obtain
\begin{align*}
\big\| \tfrac{1}{r}([\P,r^{-2\alpha}] \vv)_r\big\|^2_{-\alpha} = & \, \big\| \tfrac{1}{r}\partial_r \Phi\big\|^2_{-\alpha} = 4 \alpha^2 \int_{\Re \lambda = -\alpha +1 }|\lambda|^2 \sum_{k=1}^{\infty}\frac{|\widehatvrk(\lambda+2\alpha-1)|^2}{\left|\lambda^2-\left(\frac{k\pi}{\theta}\right)^2\right|^2} \dd\Im \lambda\\
= & \, 4 \alpha^2 \int_{\Re \tilde{\lambda} = \alpha  }|\tilde{\lambda}-2\alpha +1 |^2 \sum_{k=1}^{\infty}\frac{|\widehatvrk(\tilde{\lambda})|^2}{\left|(\tilde{\lambda}-2\alpha +1)^2-\left(\frac{k\pi}{\theta}\right)^2\right|^2} \dd\Im \tilde{\lambda} \\
\leq & \, C_{\eps}\alpha^2\theta^2 \| \vv \|_{\alpha+1}^2,
\end{align*}
where we have used \eqref{eq:estHelm1} and \eqref{eq:estHelm3}.
Similarly, we deduce that
\begin{align*}
\big\| \tfrac{1}{r}([\P,r^{-2\alpha}] \vv)_{\varphi}\big\|^2_{-\alpha} = & \, \big\| \tfrac{1}{r^2}\partial_{\varphi} \Phi\big\|^2_{-\alpha} = 4 \alpha^2 \int_{\Re \lambda = -\alpha +1 } \sum_{k=1}^{\infty}\Big|\frac{k\pi}{\theta}\Big|^2 \frac{|\widehatvrk(\lambda+2\alpha-1)|^2}{\left|\lambda^2-\left(\frac{k\pi}{\theta}\right)^2\right|^2}\dd\Im \lambda \\
= & \, 4 \alpha^2 \int_{\Re \tilde{\lambda} = \alpha  } \sum_{k=1}^{\infty}\Big(\frac{k\pi}{\theta}\Big)^2\frac{|\widehatvrk(\tilde{\lambda})|^2}{\left|(\tilde{\lambda}-2\alpha +1)^2-\left(\frac{k\pi}{\theta}\right)^2\right|^2} \dd\Im \tilde{\lambda} \\
\leq & \, C_{\eps}\alpha^2\theta^2 \| \vv \|_{\alpha+1}^2.
\end{align*}
Therefore, $ \| [\P, r^{-2\alpha}] \vv \|_{-\alpha}^2 \leq C_{\eps} \alpha^2 \theta^2 \| \vv \|_{\alpha+1}^2$.
The estimates for the first order derivatives are similar, for example
\begin{align*}
\big\| \partial_r ([\P,r^{-2\alpha}] \vv)_r\big\|^2_{-\alpha} = & \, \|  \partial^2_r \Phi \|^2_{-\alpha} = 4 \alpha^2 \int_{\Re \lambda = -\alpha -1 }|\lambda+1|^2 |\lambda+2|^2\sum_{k=1}^{\infty}\frac{|\widehatvrk(\lambda+2\alpha+1)|^2}{\left|(\lambda+2)^2-\left(\frac{k\pi}{\theta}\right)^2\right|^2} \dd\Im \lambda\\
= & \, 4 \alpha^2 \int_{\Re \tilde{\lambda} = \alpha  }|\tilde{\lambda}-2\alpha |^2 |\tilde{\lambda}-2\alpha +1 |^2 \sum_{k=1}^{\infty}\frac{|\widehatvrk(\tilde{\lambda})|^2}{\left|(\tilde{\lambda}-2\alpha +1)^2-\left(\frac{k\pi}{\theta}\right)^2\right|^2} \dd\Im \tilde{\lambda} \\
\leq & \, C_{\eps}\alpha^2 \| \vv \|_{\alpha+1}^2,
\end{align*}
where we have used \eqref{eq:estHelm2} and \eqref{eq:estHelm3} in the last step.
The bounds on $ \partial_r ([\P,r^{-2\alpha}] \vv)_{\varphi} $,  $ \tfrac{1}{r}\partial_{\varphi} ([\P,r^{-2\alpha}] \vv)_r $ and  $ \tfrac{1}{r}\partial_{\varphi} ([\P,r^{-2\alpha}] \vv)_{\varphi} $ follow similarly. Using \eqref{eqapp:grad_polar} and the estimates on $ \| \tfrac{1}{r}[\P, r^{-2 \alpha}] \vv \|^2_{-\alpha} $ we deduce the result.
Finally, \eqref{ezelina} is a consequence of $ \nabla \Phi \cdot \vn = 0 $ on $ \partial \Omega'$.
\end{proof}

\begin{cor} \label{cor:estimate_Fourier_Phi:extra}
Let $\eps \in (0, 1 ) $ and $ \theta \in (0, (1-\eps)\pi) $. Assume that $\alpha\in I_{\eps} \setminus \{ 0 \}$. Then we have the estimates
\begin{equation*}
  \begin{aligned}
\big\| \tfrac{1}{r}[\P, r^{-2 \alpha}] \vv \big\|_{-\alpha} &\leq C_{\eps} |\alpha| \theta^{2}  \left\| \nabla \vv \right\|_{\alpha}, 
\\
\big\| \nabla [\P, r^{-2 \alpha}] \vv \big\|_{-\alpha} &\leq C_{\eps} |\alpha| \theta \left\|\nabla \vv \right\|_{\alpha}.
\end{aligned}
\end{equation*}
\end{cor}

\begin{proof}
The proof is a consequence of Lemmata \ref{lem:new_estimates_uph_ur} and \ref{lem:estimate_Fourier_Phi}.
\end{proof}

%%%%%%%%%%%%%%%%%%%%%%%%%%%%%%%%%%%%%%%%%%%%%%

 \section{Variational formulation for the Stokes equations with Navier slip}
 \label{chap:5}

In this section we derive a variational formulation for the Stokes equations  \eqref{sto:sys:inhom} as is described in Section \ref{sec:3}, Step 1.
After applying the Helmholtz projection to \eqref{sto:sys:inhom} and noting that $\P\grad p=0$ by definition of the Helmholtz projection, we deduce the projected stationary Stokes problem
\begin{subequations} \label{eq:SmoothStokesProj_ch5}
\begin{alignat}{2}
    -\P\del \vu &= \P\vc{f} \qquad&& \text{ in }\Om,\label{eq:SmoothStokesProj_ch5a}\\
    \uph&=0\qquad&&  \text{ on } \dOm',\label{eq:SmoothStokesProj_ch5b}\\
   \ur+\partial_{\nn}\ur&= g \qquad&&\text{ on } \dOm'.\label{eq:SmoothStokesProj_ch5c}
\end{alignat}
\end{subequations}
Let $\vv\in\TT$, where we recall the definition of this space in \eqref{eq:test_function}. As motivated in Section \ref{sec:3}, we test \eqref{eq:SmoothStokesProj_ch5a} in $(\cdot,\cdot)_{L^2(\Omega)}$ with \begin{equation}\label{eq:testfunctions}
    r^{-2\alpha} \vv_{\mathrm{test}}:= r^{-2\alpha}(\vv- {\br |\alpha| \theta^3}  (r \dr)^2 \vv  -  c_3{\br |\alpha| \theta^3}  r^2 \Delta \vv) \qquad \text{ for some }c_3>0,
\end{equation}
to obtain a variational formulation of \eqref{eq:SmoothStokesProj_ch5}. The corresponding bilinear forms $B_1$, $B_2$ and $B_3$ will be derived in subsequent sections.

\subsection{Bilinear form \texorpdfstring{$B_1$}{B1}}\label{sec:weighted_bilienar_form}
Equation \eqref{eq:SmoothStokesProj_ch5a} tested against $r^{-2\alpha}\vv $ with $ \vv \in\TT$ in the inner product $(\cdot,\cdot)_{L^2(\Omega)} $ reads
\begin{equation}
\label{eq:to:B1:Mamo}
    (-\P \del\vu, r^{-2\alpha} \vv)_{L^2(\Omega)}=(\P\vc{f}, r^{-2\alpha}\vv)_{L^2(\Omega)}.
\end{equation}
Using that the projection $ \P $ is symmetric with respect to $(\cdot,\cdot)_{L^2(\Om)}$ (Lemma \ref{lem:commute_rdr_P}), we compute
\begin{align*}
    (-\P\del\vu,r^{-2\alpha}\vv)_{L^2(\Om)}&=(-\del\vu, \P r^{-2\alpha}\vv)_{L^2(\Om)}\\
    %&=\int_{\Om}(-\del\vu)\cdot(r^{-2\alpha}\vv-\grad\Phi_1)\dd x\\
    &=\int_{\Om}(-\del\vu)\cdot (r^{-2\alpha}\vv)\dd x+\int_{\Om}(-\del\vu)\cdot [\P, r^{-2\alpha}]\vv \dd x \\
    & = : I_1^{(1)}+I_2^{(1)},
\end{align*}
With the divergence theorem $I_1^{(1)}$ becomes
\begin{align*}
    I_1^{(1)} = -\int_{\partial\Om'}(r^{-2\alpha}\vv)\cdot\partial_{\nn}\vu\dd s+\int_{\Om}(\grad r^{-2\alpha}\vv):\grad \vu\dd x.
\end{align*}
By the product rule for vector fields (see Appendix \ref{app:polar_vec})
\begin{equation}\label{eq:prod_vec}
    \grad(r^{-2\alpha}\vv):\grad\vu  =r^{-2\alpha}\grad\vv:\grad\vu-2\alpha r^{-2\alpha-1}\vv \cdot\dr \vu,
\end{equation}
and using the boundary conditions in \eqref{eq:test_function} and \eqref{eq:SmoothStokesProj_ch5c}, we obtain
\begin{equation*}
    I_1^{(1)}= - \int_{\partial\Om'}r^{-2\alpha}v_r g \dd s + \int_{\partial\Om'}r^{-2\alpha}v_r\ur\dd s+ \int_{\Om}r^{-2\alpha}\grad\vv:\grad\vu\dd x-2\alpha\int_{\Om}r^{-2\alpha-1}\vv \cdot\dr\vu\dd x.
\end{equation*}
Again using the divergence theorem,  the second integral $I_2^{(1)}$ becomes
\begin{equation*}
    \begin{aligned}
    I_2^{(1)}%=&\;- \int_{\Om}(\del\vu)\cdot [\P, r^{-2\alpha}] \vv \dd x\\
    =&\; - \int_{\partial\Om'}\dn\vu\cdot [\P, r^{-2\alpha}] \vv \dd s + \int_{\Om}\big(\grad [\P, r^{-2\alpha}]\vv \big):\grad \vu\dd x\\
    \stackrel{\mathclap{\eqref{ezelina}}}{=}&\;-\int_{\partial\Om'}\dn\ur([\P, r^{-2\alpha}] \vv)_r \dd s +\int_{\Om}\big(\grad [\P, r^{-2\alpha}] \vv ):\grad \vu\dd x.
\end{aligned}
\end{equation*}
By combining the expressions for $I_1^{(1)}$ and $I_2^{(1)}$, we rewrite \eqref{eq:to:B1:Mamo} in the form
\begin{equation}
\begin{aligned}\label{eq:B1:vvQ}
 & \int_{\partial\Om'}r^{-2\alpha}v_r\ur\dd s+ \int_{\Om}r^{-2\alpha}\grad\vv:\grad\vu\dd x-2\alpha\int_{\Om}r^{-2\alpha-1}\vv \cdot\dr\vu\dd x - \int_{\partial\Om'}(\dn\ur)([\P, r^{-2\alpha}] \vv)_r\dd s \\
    &  + \int_{\Om}\big(\grad [\P, r^{-2\alpha}] \vv \big):\grad \vu\dd x=  \int_{\Om}r^{-2\alpha} \P \vf \cdot \vv \dd x +  \int_{\partial\Om'}r^{-2\alpha}v_r g \dd s.
\end{aligned}
\end{equation}

\begin{rem}
\label{rem:no:weak}
For coercivity take $ \vu = \vv $ in \eqref{eq:B1:vvQ}, then we have control on $|u_r|_{\alpha}^2$ and $\|\grad \vu \|^2_{\alpha}$. The fourth term on the left-hand side of \eqref{eq:B1:vvQ} is difficult to deal with, since there is no control of derivatives on the boundary. A natural approach would be to apply the Navier-slip condition \eqref{eq:SmoothStokesProj_ch5c}. However, this changes the scaling, which we want to avoid.
\end{rem}
For obtaining a coercivity estimate, we apply  the fundamental theorem of calculus to
\begin{align*}
    \int_{\dOm'} (\dn\ur) ([\P, r^{-2\alpha}] \vv)_r \dd s
    =&\;\int_0^{\theta}\int_0^{\infty}\dph((\dph\ur)([\P, r^{-2\alpha}] \vv)_r) \ddrr\dd \ph.
\end{align*}
This requires control on the second-order derivative $\dph^2\ur$ in $\Om$, but there is only control on the first derivatives by $\|\grad\vu\|_{\alpha}^2$. Control on all second-order derivatives is obtained by introducing two additional bilinear forms as detailed in Sections \ref{subsec:2nd_bilinear_form} and \ref{subsec:vort_bilinear_form} below.

We define the bilinear form
\begin{equation}
\begin{aligned}\label{eq:B1}
    B_1(\vu,\vv)=&\int_{\partial\Om'}r^{-2\alpha}v_r\ur\dd s+ \int_{\Om}r^{-2\alpha}\grad\vv:\grad\vu\dd x -2\alpha\int_{\Om}r^{-2\alpha-1}\vv \cdot\dr\vu\dd x \\
    & - \int_{\Om} r^{-2} \dph((\dph\ur)([\P, r^{-2\alpha}] \vv)_r) \dd x + \int_{\Om}\big(\grad [\P, r^{-2\alpha}] \vv \big):\grad \vu\dd x=:\sum_{j=1}^5 T_j^{(1)}.
\end{aligned}
\end{equation}

\subsection{Bilinear form \texorpdfstring{$B_2$}{B2}}\label{subsec:2nd_bilinear_form}

We test \eqref{eq:SmoothStokesProj_ch5a} with
\begin{equation}\label{eq:testfunction_B2}
     r^{-2\alpha} \vv_2:= -  r^{-2\alpha} (r\dr)^2 \vv,\qquad \vv\in\TT,
\end{equation}
to obtain
\begin{equation}\label{eq:B2_3dervs}
  \big(-\P\del\vu, r^{-2\alpha} \vv_2\big)_{L^2(\Omega)}=(\P\vf,  r^{-2\alpha} \vv_2)_{L^2(\Omega)}.
\end{equation}
We rewrite
\begin{align*}
    (-\P \del\vu,  r^{-2\alpha} \vv_2)_{L^2(\Omega)}&=(-\del\vu, \P r^{-2\alpha}\vv_2)_{L^2(\Om)}\\
    &=\int_{\Om}(-\del\vu)\cdot (r^{-2\alpha}\vv_2)\dd x + \int_{\Om}(-\del\vu)\cdot [\P,r^{-2\alpha}]\vv_2\dd x \\
    & =:\,I^{(2)}_1 + I^{(2)}_2.
\end{align*}
With the divergence theorem and \eqref{eq:prod_vec},  $I^{(2)}_1$ becomes
\begin{align*}
    I^{(2)}_1
    =&\;-\int_{\partial\Om'}(r^{-2\alpha}\vv_2)\cdot\partial_{\nn}\vu\dd s+\int_{\Om}r^{-2\alpha}\grad \vv_2:\grad \vu\dd x-2\alpha\int_{\Om}r^{-2\alpha-1}\vv_2\cdot\dr\vu\dd x\\
    =&\;-\int_{\dOm'}r^{-2\alpha}(r\dr \vr)\big((r\dr-2\alpha +1)\dn\ur \big)\dd s+\int_{\Om}r^{-2\alpha}(\grad r\dr\vv):(\grad r\dr\vu)\dd x \\
    &\;  - 2 \alpha \int_{\Om}r^{-2\alpha}(\grad r \dr\vv) :\grad \vu\dd x+2\alpha \int_{\Om}r^{-2\alpha-1} ((r\dr)^2\vv) \cdot (\dr\vu)\dd x,
\end{align*}
where in the last step we have used the commutation relations \eqref{eqapp:com_rel} and applied integration by parts. Using the Navier-slip boundary condition \eqref{eq:SmoothStokesProj_ch5c} and again \eqref{eqapp:com_rel} gives
\begin{equation*}
    \begin{aligned}
        I_1^{(2)}=&\;- \int_{\dOm'}r^{-2\alpha}(r\dr \vr)\big((r\dr-2\alpha +1) g \big)\dd s+ \int_{\dOm'}r^{-2\alpha}(r\dr \vr)\big((r\dr-2\alpha+1)\ur \big)\dd s \\
    &\; +\int_{\Om}r^{-2\alpha}(\grad r\dr\vv):(\grad r\dr\vu)\dd x- 2 \alpha \int_{\Om}r^{-2\alpha}(\grad r \dr\vv) :\grad \vu\dd x \\
    & \; +2\alpha \int_{\Om}r^{-2\alpha-1} ((r\dr)^2\vv) \cdot (\dr\vu)\dd x.
    \end{aligned}
\end{equation*}
Again by the divergence theorem and fundamental theorem of calculus, the second integral becomes
\begin{align*}
    I^{(2)}_2=&\;\int_{\Om}(-\del\vu)\cdot [\P,r^{-2\alpha }]\vv_2 \dd x\;\;
    \stackrel{\mathclap{\eqref{ezelina}}}{=}\;\;-\int_{\partial\Om'}(\dn\ur)([\P,r^{-2\alpha}]\vv_2)_r \dd s+\int_{\Om}\big(\grad [\P,r^{-2\alpha}]\vv_2\big):\grad \vu\dd x \\
= & \,  -\int_{\Om}r^{-2}\dph\big((\dph \ur) ([\P,r^{-2\alpha}]\vv_2)_r\big) \dd x+\int_{\Om}\big(\grad [\P,r^{-2\alpha}]\vv_2\big):\grad \vu\dd x \\
= & \,  \int_{\Om}r^{-2}\dph\big((\dph \ur) ([\P,r^{-2\alpha}](r\dr)^2 \vv)_r\big) \dd x-\int_{\Om}\big(\grad [\P,r^{-2\alpha}](r\dr)^2 \vv\big):\grad \vu\dd x.
\end{align*}
By combining the expressions for $I_1^{(2)}$ and $I_2^{(2)}$ and by excluding the term involving $ g $, we obtain the bilinear form
\begin{equation}\label{eq:B2}%
\begin{aligned}
 B_2(\vu,\vv)=  &\; \int_{\dOm'}r^{-2\alpha}(r\dr \vr)\big((r\dr-2\alpha+1)\ur \big)\dd s
   +\int_{\Om}r^{-2\alpha}(\grad r\dr\vv):(\grad r\dr\vu)\dd x \\ &\;- 2 \alpha \int_{\Om}r^{-2\alpha}(\grad r \dr\vv) :\grad \vu\dd x  +2\alpha \int_{\Om}r^{-2\alpha-1} ((r\dr)^2\vv) \cdot (\dr\vu)\dd x  \\
    & \; +\int_{\Om}r^{-2}\dph\big((\dph \ur) ([\P,r^{-2\alpha}](r\dr)^2 \vv)_r\big) \dd x-\int_{\Om}\big(\grad [\P,r^{-2\alpha}](r\dr)^2 \vv\big):\grad \vu\dd x \\ = &:\sum_{j=1}^6 T^{(2)}_j,
\end{aligned}
\end{equation}
Hence, \eqref{eq:B2_3dervs} can be rewritten in the form
\begin{equation*}
B_2(\vu,\vv) = (\P\vf, r^{-2\alpha}\vv_2)_{L^2(\Omega)} + \int_{\dOm'}r^{-2\alpha}(r\dr \vr)\big((r\dr-2\alpha+1) g \big)\dd s.
\end{equation*}
From $B_2(\vu,\vu)$ we get control on $\|\grad r\dr\vu\|_{\alpha}^2$, i.e., the second-order derivatives in $r$ and the mixed derivatives. By using the divergence-free condition we also gain control on $\dph^2\uph$. However,  control on $\dph^2\ur$ is still missing and, even worse, $B_2$ gives an extra $\dph^2\ur$ term. The third bilinear form will give control on the last so far uncontrolled terms.

\subsection{Bilinear form \texorpdfstring{$B_3$}{B3}}
\label{subsec:vort_bilinear_form}
For the third bilinear form 
we test \eqref{eq:SmoothStokesProj_ch5a} with
\begin{equation}\label{eq:testfunction_vorticity_bilinear_form}
    r^{-2\alpha}\vv_3 :=- r^{-2\alpha+2} \Delta \vv, %=- r^{-2+2\alpha} \nabla^{\perp} \om_{\vv}, 
\qquad\vv\in\TT.
\end{equation}

We calculate
\begin{equation*}\label{eq:B3_1}
\begin{aligned}
    \left(-\P\del\vu, r^{-2\alpha}\vv_3\right)_{L^2(\Omega)}  =&\; \left(-\del\vu, r^{-2\alpha}\vv_3\right)_{L^2(\Omega)} + \left(-[\P, \Delta]\vu, r^{-2\alpha}\vv_3\right)_{L^2(\Omega)} \\
    &\; \int_{\Om}r^{-2\alpha +2}\del\vu \cdot  \del\vv \dd x +  \int_{\Om} r^{-2\alpha +2} [\P, \Delta]\vu \cdot \del\vv  \dd x.  \\
\end{aligned}
\end{equation*}
The third bilinear form is
\begin{equation}\label{eq:B3}
\begin{aligned}
   B_3(\vu,\vv) =& \int_{\Om}r^{-2\alpha+2}\del\vu\cdot \del\vv \dd x + \int_{\Om} r^{-2\alpha +2} [\P, \Delta]\vu \cdot \del\vv  \dd x
    =:\;\sum_{j=1}^2 T^{(3)}_j.
\end{aligned}
\end{equation}

%}

\subsection{The variational problem}
\label{sec:var:problem}
Combining the computations of the preceding three sections, we define for an appropriate constant $c_3>0$ the bilinear form
\begin{equation}
\label{eq:B_ch5}
    B(\vu,\vv):=B_1(\vu,\vv)+{\br |\alpha| \theta^3} B_2(\vu,\vv)+c_3 {\br |\alpha| \theta^3}  B_3(\vu,\vv),
\end{equation}
which thus arises from testing the equation $-\P\del\vu=\P\vc{f}$ in $(\cdot,\cdot)_{L^2(\Omega)}$ with the test function $r^{-2\alpha}\vv_{\mathrm{test}}$ as defined in \eqref{eq:testfunctions}.
Define the pairing
\begin{equation*}%\label{eq:pairing_g}
\langle g, v_{r} \rangle_{\alpha} := \int_{\dOm'}r^{-2\alpha} g \vr \dd s + {\br |\alpha| \theta^3}  \int_{\dOm'}r^{-2\alpha} \big((r\dr-2\alpha +1) g \big) (r\dr \vr)\dd s.
\end{equation*}
Then the variational problem associated to \eqref{eq:SmoothStokesProj_ch5} is
\begin{equation}
\label{la:mil:equ:2}
B(\vu, \vv) = (\P\vf, r^{-2\alpha}\vv_{\text{test}})_{L^2(\Omega)} + \langle g, v_r \rangle_{\alpha},\qquad \vv\in\TT,
\end{equation}
where $ B $ is defined in \eqref{eq:B_ch5}, and $ B_1$, $ B_2 $ and $ B_3 $ are defined in \eqref{eq:B1}, \eqref{eq:B2} and \eqref{eq:B3}, respectively.

We conclude this section by noticing that any solution $ \vu \in \mathfrak{X}^2_{\alpha,\theta} $ of the system \eqref{eq:SmoothStokesProj_ch5} satisfies the variational formulation \eqref{la:mil:equ:2}.

\begin{lem}
\label{reg:sol:are:var}
Let $ \theta \in (0, \pi) $, $\eps \in (0, 1- \tfrac{\theta}{\pi} ) $  and $\alpha \in I_{\eps} \setminus \mathbb{Z} $. Suppose that $ \vf \in \mathcal{H}^0_{\alpha} $ and $g\in \mathscr{X}^0_{\alpha}$. If $ \vu \in \mathfrak{X}^2_{\alpha,\theta} $ satisfies  \eqref{eq:SmoothStokesProj_ch5} almost everywhere in $ \Omega $, then $ \vu $ satisfies the variational formulation  \eqref{la:mil:equ:2}.
\end{lem}

\begin{proof} This is a consequence of the fact that any almost everywhere solution $ \vu $ of  \eqref{eq:SmoothStokesProj_ch5} satisfies  $ (-\P \Delta \vu, r^{-2\alpha}\vv_{\text{test}})_{L^2(\Omega)} = ( \P\vf, r^{-2\alpha} \vv_{\text{test}})_{L^2(\Omega) }$, which equals  \eqref{la:mil:equ:2} after integration by parts.
\end{proof}

%%%%%%%%%%%%%%%%%%%%%%%%%%%%%%%%%%%%%%%%%%%%%%%%%%%%%%%%%%%%

\section{Coercivity and boundedness of the bilinear form}
\label{chap:proof_coer}

This section is devoted to the proof of the coercivity and boundedness estimate as stated in Propositions \ref{prop:coercivity_ch5} and \ref{prop:boundedness_ch5} from Section \ref{sec:3}, Step 2.

Throughout this section we fix  $\eps\in (0, 1)$ and assume that
\begin{equation}\label{eq:conditions}
\theta\in(0,(1-\eps)\pi) \text{ and } \;\alpha\in I_{\eps} \setminus \ZZ, \quad \text{where} \quad  I_{\eps} = [- (1-\eps)\tfrac{\pi}{\theta} + 1, (1-\eps)\tfrac{\pi}{\theta} - 1].
\end{equation}
Moreover, let $\vu\in \TT$ be as defined in \eqref{eq:test_function} and we will use the expressions in polar coordinates as given in Appendix \ref{app:polar}.
In Sections \ref{sec:6_B1}, \ref{sec:6_B2} and \ref{sec:6_B3} we focus on the coercivity estimates of the bilinear forms $B_1$, $B_2$ and $B_3$ separately. Finally, in Section \ref{sec:Proof_props} we combine all estimates to prove Propositions \ref{prop:coercivity_ch5} and \ref{prop:boundedness_ch5}.

\subsection{Estimates for \texorpdfstring{$B_1$}{B1}}\label{sec:6_B1}
Consider the first bilinear form \eqref{eq:B1} as derived in Section \ref{sec:weighted_bilienar_form}. After integration by parts, it reads
\begin{equation}\label{eq:B1_uu}
    \begin{aligned}
    B_1(\vu,\vu)=&\;|\ur|^2_{\alpha}+ \|\grad \vu\|^2_{\alpha}-2\alpha^2\int_{\Om}r^{-2\alpha-2}|\vu|^2 \dd x
    -\int_{\Om} r^{-2} {\dph} \left[(\dph u_r)([\P, r^{-2\alpha}]\vu)_r \right]\dd x \\ &+\int_{\Om}\big(\grad[\P, r^{-2\alpha}]\vu\big):\grad \vu\dd x=:\sum_{j=1}^5T^{(1)}_j.
\end{aligned}
\end{equation}

\begin{prop}\label{prop:B1} There exist $C_1,C_2\in(0,\infty)$ such that for all  $\alpha$ and $ \theta$ subject to \eqref{eq:conditions} and all $c_3>0$, we have
\begin{align*}
    B_1(\vu,\vu)\geq&\;|\ur|^2_{\alpha}+\big(1- C_1|\alpha|\theta -  \tfrac{C_2 {\br |\alpha| \theta }}{c_3} \big)\|\grad\vu\|^2_{\alpha} -\frac{ c_3 {\br |\alpha| \theta^3}}{8}\big\|\tfrac{1}{r} \dph^2 \ur \big\|^2_{\alpha} \quad \text{ for all } \vu\in\TT.
\end{align*}
\end{prop}
This incomplete coercivity estimate for $B_1$ is a consequence of Lemmata \ref{lem:coercivity_B1_T3}-\ref{lem:coercivity_B1_TildeS} below in which we estimate the term $T_3^{(1)}$-$T_5^{(1)}$ from \eqref{eq:B1_uu}.

\begin{lem}[Estimate of $T_{3}^{(1)} $]\label{lem:coercivity_B1_T3} 
There exists a $C_1\in(0,\infty)$ such that for all $\alpha $ and $ \theta$ subject to \eqref{eq:conditions}, we have
\begin{align*}
    \left|2\alpha^2\int_{\Om}r^{-2\alpha-2}|\vu|^2 \dd x \right|&\leq C_1 |\alpha|\theta \|\grad\vu\|^2_{\alpha}\qquad \text{ for all }\vu\in\TT.
\end{align*}
\end{lem}
\begin{proof}
  By the Cauchy-Schwarz inequality, Young's inequality and Lemma \ref{lem:new_estimates_uph_ur}, we obtain
  \begin{align*}
    \left|2\alpha^2\int_{\Om}r^{-2\alpha-2}|\vu|^2 \dd x \right|&\leq  C_0\alpha^2\theta^2\|\grad\vu\|^2_{\alpha}\leq C_1 |\alpha|\theta\|\grad\vu\|^2_{\alpha}.\qedhere
\end{align*}
\end{proof}

\begin{lem}[Estimate of $T_4^{(1)}$]\label{lem:coercivity_B1_T4}
There exist $C_2,C_3\in(0,\infty)$ such that for all  $\alpha$ and $ \theta$ subject to \eqref{eq:conditions} and all $c_3>0$, we have
\begin{align*}
   & \left|\int_{\partial\Om} r^{-2} {\dph} \left[(\dph u_r)([\P, r^{-2\alpha}]\vu)_r \right]\dd x \right|\\
    &\leq\; \frac{c_3 {\br |\alpha| \theta^3} }{8} 	\big\|\tfrac{1}{r}\dph^2 \ur  \big\|^2_{\alpha}
    +\big( \tfrac{C_2 {\br |\alpha| \theta} }{c_3}+ C_3|\alpha|\theta \big)\|\grad\vu\|^2_{\alpha}\quad \text{ for all }\vu\in\TT.
\end{align*}
\end{lem}
\begin{proof}
Note that
\begin{align}\label{eq:lemma_92}
    T_4^{(1)}=\;\int_{\Om}r^{-2}(\dph^2\ur)([\P, r^{-2\alpha}]\vu)_r \dd x+\int_{\Om}r^{-2}(\dph\ur)\dph ([\P, r^{-2\alpha}]\vu)_r \dd x.
\end{align}
For the first integral in \eqref{eq:lemma_92}, applying the Cauchy-Schwarz inequality, Young's inequality, and Corollary \ref{cor:estimate_Fourier_Phi:extra}, we obtain
\begin{align*}
    &\left|\int_{\Om}r^{-2}(\dph^2\ur)([\P, r^{-2\alpha}]\vu)_r\dd x\right|\\
    \leq &\; \frac{c_3 {\br |\alpha| \theta^3}}{8}\int_{\Om}r^{-2\alpha}(r^{-1}\dph^2\ur)^2\dd x+ \frac{2}{c_3 {\br |\alpha| \theta^3} }\int_{\Om}r^{2\alpha}(r^{-1}([\P, r^{-2\alpha}]\vu)_r)^2\dd x\\
    \leq&\; \frac{c_3 {\br |\alpha| \theta^3}}{8} \big\|\tfrac{1}{r}  \dph^2 \ur \big\|^2_{\alpha}+ \frac{C_2 {\br |\alpha| \theta} }{c_3}\|\grad\vu\|^2_{\alpha}.
\end{align*}
Similarly, for the second integral in \eqref{eq:lemma_92}
\begin{align*}
    &\left|\int_{\Omega}r^{-2}(\dph\ur)\dph ([\P, r^{-2\alpha}]\vu)_r\dd x\right|\\
    \leq&\; \frac{|\alpha|\theta}{2}\int_{\Om}r^{-2\alpha}(r^{-1}\dph\ur)^2\dd x+\frac{1}{2 |\alpha|\theta}\int_{\Om}r^{2\alpha}(r^{-1}(\dph[\P, r^{-2\alpha}]\vu)_r)^2\dd x\\
    \leq&\;|\alpha|\theta \int_{\Om }r^{-2\alpha}\big[(r^{-1}(\dph\ur-\uph))^2+r^{-1}\uph^2\big]\ddrr\dd\ph +\frac{1}{2|\alpha|\theta}\int_{\Om}r^{2\alpha}(r^{-1}(\dph[\P, r^{-2\alpha}]\vu)_r)^2\dd x\\
    \leq&\; C_3|\alpha|\theta\|\grad\vu\|_{\alpha}^2,
\end{align*}
where we have used Lemma \ref{lem:new_estimates_uph_ur} and Corollary \ref{cor:estimate_Fourier_Phi:extra} in the last step.
\end{proof}

\begin{lem}[Estimate of $T_5^{(1)}$]\label{lem:coercivity_B1_TildeS}
 There exists a $C_4\in(0,\infty)$ such that for all $\alpha $ and $ \theta$ subject to \eqref{eq:conditions}, we have
\begin{align*}
    \left| \int_{\Om}\big(\grad [\P, r^{-2\alpha}]\vu \big):\grad \vu\dd x   \right|
    \leq   C_4 |\alpha| \theta \|\grad\vu\|^2_{\alpha}\qquad \text{ for all }\vu\in\TT.
\end{align*}

\end{lem}

\begin{proof}
By the Cauchy-Schwarz inequality and Corollary \ref{cor:estimate_Fourier_Phi:extra}, we have
\begin{equation*}
    \left| \int_{\Om}\big(\grad [\P, r^{-2\alpha}]\vu \big):\grad \vu\dd x   \right|
    \leq    \left|\int_{\Om} r^{2\alpha}\left| \grad [\P, r^{-2\alpha}]\vu \right|^2 \dd x \right|^{\frac{1}{2}} \|\grad\vu\|_{\alpha}\lesssim   |\alpha| \theta  \|\grad\vu\|^2_{\alpha}.\qedhere
\end{equation*}
\end{proof}

\subsection{Estimates for \texorpdfstring{$B_2$}{B2}}\label{sec:6_B2}
Consider the second bilinear form \eqref{eq:B2} as derived in Section \ref{subsec:2nd_bilinear_form}
\begin{equation}\label{eq:B2_uu}
  \begin{aligned}
  B_2(\vu,\vu)=  &\; \int_{\dOm'}r^{-2\alpha}(r\dr \ur)\big((r\dr-2\alpha+1)\ur \big)\dd s
   +\|\grad r\dr\vu\|^2_{\alpha} \\ &\; - 2 \alpha \int_{\Om}r^{-2\alpha}(\grad r \dr\vu) :\grad \vu\dd x +2\alpha \int_{\Om}r^{-2\alpha-1} ((r\dr)^2\vu) \cdot (\dr\vu)\dd x  \\
    & \; - \int_{\Om} r^{-2}\dph[(\dph\ur)([\P, r^{-2\alpha}](r\dr)^2\vu)_r]\dd x-\int_{\Om}\big(\grad [\P, r^{-2\alpha}](r\dr)^2\vu \big):\grad \vu\dd x \\ = & :\sum_{j=1}^6 T^{(2)}_j.
\end{aligned}
\end{equation}

We have the following partial coercivity estimate for $B_2$.
\begin{prop}\label{prop:B2}
There exist $D_1,D_2,D_3\in(0,\infty)$ such that for all  $\alpha$ and $ \theta$ subject to \eqref{eq:conditions} and all $c_3>0$, we have
\begin{align*}
    B_2(\vu,\vu)\geq&\;|r\dr\ur|^2_{\alpha} -2(\alpha-\tfrac{1}{2})^2 |\ur|^2_{\alpha}+ \big( 1-  D_1 {\br |\alpha|} \theta - {\br \tfrac{D_2 |\alpha \theta|^2  }{c_3}} \big) \|\grad r\dr\vu\|^2_{\alpha} \\ & \; - \frac{D_3{\br |\alpha|}}{\theta} \|\grad\vu\|^2_{\alpha} -  \frac{c_3}{8}\big\|\tfrac{1}{r} \dph^2 \ur \big\|^2_{\alpha}\qquad \text{ for all }\vu\in\TT.
\end{align*}
\end{prop}

This proposition is a consequence of Lemmata \ref{lem:coercivity_B2_T1}-\ref{lem:coercivity_B2_T6} below, in which we estimate the terms $T_1^{(2)}$ and $T_3^{(2)}$-$T_6^{(2)}$ from \eqref{eq:B2_uu}.
\begin{lem}[Reformulation of $T_1^{(2)}$] \label{lem:coercivity_B2_T1}
For $\alpha $ and $ \theta $ subject to \eqref{eq:conditions} we have
\begin{equation*}
    T_1^{(2)}=|r\dr\ur|^2_{\alpha}-2(\alpha-\tfrac{1}{2})^2|\ur|^2_{\alpha}\qquad\text{ for all }\vu\in\TT.
\end{equation*}
\end{lem}
\begin{proof}
This is immediate with integration by parts
\begin{equation*}\label{eq:B2_T1}
    (1-2\alpha)\int_{\dOm'}r^{-2\alpha}(r\dr\ur)\ur\dd s=\frac{1-2\alpha}{2}\int_{\dOm'}r^{-2\alpha+1}\dr(\ur^2)\dd s
    =-2(\alpha-\tfrac{1}{2})^2\int_{\dOm'}r^{-2\alpha}\ur^2\dd s.\qedhere
\end{equation*}
\end{proof}
\begin{lem}[Estimate of $T_3^{(2)}+T_4^{(2)}$]\label{lem:coercivity_B2_T34}
There exist $D_1,D_2\in(0,\infty)$ such that for all  $\alpha$ and $ \theta$ subject to \eqref{eq:conditions}, we have
\begin{align*}
        \big|T_3^{(2)}\big|+\big|T_4^{(2)}\big|
        \leq   D_1 {\br |\alpha|} \theta \|\grad r\dr\vu\|^2_{\alpha}+   \frac{D_2{\br |\alpha|}}{\theta} \|\grad \vu\|^2_{\alpha}    \qquad\text{ for all }\vu\in\TT.
\end{align*}
\end{lem}
\begin{proof}
  The estimate is a  direct consequence of the Cauchy-Schwarz inequality and Young's inequality.
\end{proof}
%
%%%%%%%%%%%%%%%%%%%%%%%%%%%%%%%%%B2_T5%%%%%%%%%%%%%%%%
%
To estimate $T_5^{(2)}$ we follow the same strategy as in Lemma \ref{lem:coercivity_B1_T4}.
\begin{lem}[Estimate of $T_5^{(2)}$]\label{lem:coercivity_B2_T5}
There exist $D_3,D_4,D_5\in(0,\infty)$ such that for all  $\alpha$ and $ \theta$ subject to \eqref{eq:conditions} and all $c_3>0$, we have
\begin{align*}
    &\left|\int_{\Om} r^{-2}\dph\big((\dph\ur)([\P, r^{-2\alpha}](r\dr)^2 \vu)_r\big)\dd x \right|\\
    \leq & \; \frac{c_3}{8}\big\|\tfrac{1}{r}\dph^2 \ur  \big\|^2_{\alpha}
    + \frac{D_4{\br |\alpha|}}{\theta} \|\grad\vu\|^2_{\alpha}+ {\br |\alpha| \theta\big(   \tfrac{D_3 |\alpha| \theta}{c_3} + D_5 \big)} \| \grad (r \dr \vu )\|^2_{\alpha} \qquad    \text{ for all }\vu\in\TT.
\end{align*}
\end{lem}

\begin{proof}
Recall that
\begin{align}
    T_5^{(2)}
   =&\;\int_{\Om}r^{-2}(\dph^2\ur)([\P, r^{-2\alpha}](r\dr)^2 \vu)_r\dd x \ph+\int_{\Om}r^{-2}(\dph\ur)\dph ([\P, r^{-2\alpha}](r\dr)^2 \vu)_r \dd x.\label{eq:lemma_92:Max}
\end{align}
For the first integral in \eqref{eq:lemma_92:Max}, applying the Cauchy-Schwarz inequality, Young's inequality, and Lemma \ref{lem:estimate_Fourier_Phi} gives
\begin{align*}
    \left|\int_{\Om}r^{-2}(\dph^2\ur)([\P, r^{-2\alpha}](r\dr)^2 \vu)_r\dd x\right| \leq &\; \frac{c_3}{8} \int_{\Om}r^{-2\alpha}(r^{-1}\dph^2\ur)^2\dd x \\ & \; +  \frac{2 }{c_3}\int_{\Om}r^{2\alpha}(r^{-1}([\P, r^{-2\alpha}](r\dr)^2 \vu)_r)^2\dd x\\
    \leq &\;
     \frac{c_3 }{8}\big\|\tfrac{1}{r}  \dph^2 \ur\big\|^2_{\alpha}+ \frac{D_3 \alpha^2\theta^2}{c_3}\|\dr (r \dr \vu)\|^2_{\alpha}.
\end{align*}
Similarly, for the second integral in \eqref{eq:lemma_92:Max}
\begin{align*}
    &\left|\int_{\om}r^{-2}(\dph\ur)\dph([\P, r^{-2\alpha}](r\dr)^2 \vu)_r \dd x\right|\\
    \leq&\;\frac{{\br |\alpha|}}{2\theta}\int_{\Om}r^{-2\alpha}(r^{-1}\dph\ur)^2\dd x+\frac{\theta}{2{\br |\alpha|}}\int_{\Om}r^{2\alpha}(r^{-1}\dph ([\P, r^{-2\alpha}](r\dr)^2 \vu)_r)^2\dd x\\
    \leq&\;\frac{{\br |\alpha|}}{\theta} \int_{\Om}r^{-2\alpha}\big[(r^{-1}(\dph\ur-\uph))^2+\uph^2\big]\dd x +\frac{\theta}{2{\br |\alpha|}}\int_{\Om}r^{2\alpha}(r^{-1}\dph ([\P, r^{-2\alpha}](r\dr)^2 \vu)_r)^2\dd x\\
    \leq&\; \frac{D_4{\br |\alpha|} }{\theta} \|\grad\vu\|_{\alpha}^2 + D_5 {\br |\alpha|} \theta \| \dr (r \dr \vu )\|^2_{\alpha},
\end{align*}
where we have used Lemmata \ref{lem:new_estimates_uph_ur} and \ref{lem:estimate_Fourier_Phi} in the last step.
\end{proof}
%%%%%%%%%%%%%%%%%%%%%%%%%%%%%%%%%%%B2_T6%%%%%%%%%%%%%%%%
\begin{lem}[Estimate of $T_6^{(2)}$]\label{lem:coercivity_B2_T6}
There exists a $D_6\in(0,\infty)$ such that for all $\alpha $ and $ \theta $ subject to \eqref{eq:conditions}, we have
\begin{align*}
    \left| \int_{\Om}\big(\grad [\P, r^{-2\alpha}](r\dr)^2 \vu \big):\grad \vu\dd x   \right|\leq & \; D_6 {\br |\alpha|} \theta\| \nabla(r\dr \vu)\|^2_{\alpha} + \frac{{\br |\alpha|}}{\theta} \|\grad\vu\|^2_{\alpha}\qquad \text{ for all }\vu\in\TT.
\end{align*}
%where $D_6>0$ is a constant that does not depend on $ \alpha $. % and $ \theta$.
\end{lem}

\begin{proof}
The Cauchy-Schwarz inequality, Young's inequality and Lemma \ref{lem:estimate_Fourier_Phi} give
\begin{align*}
    \left| \int_{\Om}\big(\grad [\P, r^{-2\alpha}](r\dr)^2 \vu \big):\grad \vu\dd x   \right|
    \leq  & \;  \frac{{\br |\alpha|}}{\theta} \|\grad\vu\|_{\alpha}^2 + \frac{\theta}{{\br |\alpha|}}\left|\int_{\Om} r^{2\alpha}\left|\grad [\P, r^{-2\alpha}](r\dr)^2 \vu \right|^2 \dd x \right|  \\  \leq & \;\frac{{\br |\alpha|}}{\theta}  \|\grad\vu\|^2_{\alpha} + D_6 {\br |\alpha|} \theta \| \dr(r\dr \vu)\|^2_{\alpha}.\qedhere
\end{align*}
\end{proof}

\subsection{Estimates for \texorpdfstring{$B_3$}{B3}}\label{sec:6_B3}
Consider the third bilinear form \eqref{eq:B3} as derived in Section \ref{subsec:vort_bilinear_form}
 \begin{align*}
   B_3(\vu,\vu) =&\; \|r \del\vu \|^2_{\alpha}+\int_{\Om}r^{-2\alpha+2} [\P, \Delta]\vu \cdot \del\vu \dd x =:\sum_{j=1}^2 T^{(3)}_j.
\end{align*}
We obtain the following final estimate required for proving coercivity in the next section.
\begin{prop}\label{prop:B3}
There exist $E_1,E_2\in(0,\infty)$ such that for all $\alpha $ and $ \theta $ subject to \eqref{eq:conditions}, we have
\begin{align*}
    B_3(\vu,\vu)\geq&\;\tfrac{1}{2}\|r\del\vu \|^2_{\alpha} -\tfrac{1}{8}\big\|\tfrac{1}{r} \dph^2 \ur \big\|^2_{\alpha}
     -\tfrac{E_1}{\theta^2} \|\grad \vu\|^2_{\alpha}- E_2 \|\grad r\dr\vu\|^2_{\alpha} \quad\text{ for all }\vu\in\TT.
\end{align*}
\end{prop}

\begin{proof}
From Lemmata \ref{lem:new_estimates_uph_ur} and \ref{lem:est_P_delta_w} we deduce that
\begin{align*}
\| [\P, \Delta]\vu \|_{\alpha-1 }^2  \leq \tfrac{1}{4}\big\|\tfrac{1}{r} \dph^2 \ur \big\|^2_{\alpha} + \tfrac{2 E_1}{\theta^2}  \| \nabla \vu \|_{\alpha}^2 + 2E_2 \| \nabla r \dr \vu \|_{\alpha}^2.
\end{align*}
Therefore, we have
\begin{align*}
\left| \int_{\Om}r^{-2\alpha+2} [\P, \Delta]\vu \cdot \del\vu \dd x \right| \leq & \;  \tfrac{1}{2}\|r\del\vu\|^2_{\alpha} + \tfrac{1}{2}\|[\P, \Delta]\vu\|^2_{\alpha-1} \\
\leq & \; \tfrac{1}{2}\|r\del\vu\|^2_{\alpha} + \tfrac{1}{8}\big\|\tfrac{1}{r} \dph^2 \ur \big\|^2_{\alpha} + \tfrac{ E_1}{\theta^2}  \| \nabla \vu \|_{\alpha}^2 + E_2 \| \nabla r \dr \vu \|_{\alpha}^2.\qedhere
\end{align*}
\end{proof}

 \subsection{Coercivity and boundedness of the bilinear form $B$}\label{sec:Proof_props}
We are now in the position to prove the coercivity and boundedness estimate in Propositions \ref{prop:coercivity_ch5} and \ref{prop:boundedness_ch5} by combining Propositions \ref{prop:B1}, \ref{prop:B2} and \ref{prop:B3}.

\begin{proof}[Proof of Proposition \ref{prop:coercivity_ch5}]
Recall that we need to prove the coercivity estimate
\begin{equation*}
    B(\vu,\vu)\stackrel{\eqref{eq:B_ch5}}{=}B_1(\vu,\vu)+{\br |\alpha |\theta^3} B_2(\vu,\vu)+c_3 {\br |\alpha |\theta^3} B_3(\vu,\vu)\geq C \|\vu\|^2_{\mathfrak{X}^2_{\alpha,\theta}} \qquad \text{ for all }\vu\in\TT,
\end{equation*}
where $\mathfrak{X}^2_{\alpha,\theta}$ is as defined in \eqref{eq:def_frakX} and $c_3>0$ is a constant independent of $\alpha$ and $\theta$.
The strategy is to absorb the terms with $\dph^2\ur$ into $B_3$ using the estimate
\begin{equation}\label{eq:est_dph^2ur}
  \begin{aligned}
    	\big\|\tfrac{1}{r}\dph^2 \ur  \big\|^2_{\alpha} \stackrel{\eqref{eqapp:del_polar}}{=} &\;\ \big\| r (\del \vu)_r - \partial_r (r \dr u_r) + 2(\dph u_{\varphi} + u_r ) - u_r  \big\|^2_{\alpha}  \\
\leq &\; \| r \del\vu\|_{\alpha}^2+ \| \grad r \dr \vu\|_{\alpha}^2 + C_3 \| \grad  \vu\|_{\alpha}^2,
\end{aligned}
\end{equation}
where in the last step we have used Lemma \ref{lem:new_estimates_uph_ur} and the constant $C_3$ is independent of $\alpha$ and $\theta$.
Combining the estimate from \eqref{eq:est_dph^2ur} with those from Propositions \ref{prop:B1}, \ref{prop:B2} and \ref{prop:B3} gives
{\br
\begin{align*}
    &\;B_1(\vu,\vu)+ |\alpha| \theta^3 B_2(\vu,\vu)+c_3 |\alpha| \theta^3 B_3(\vu,\vu)\\
    \geq&\; K_0 |\ur|^2_{\alpha}+ |\alpha| \theta^3 |r\dr\ur|^2_{\alpha}
    +K_1\|\grad\vu\|^2_{\alpha}
    + |\alpha| \theta^3 K_2\|\grad r\dr \vu\|_{\alpha}^2+\frac{1}{8}c_3 |\alpha| \theta^3\|r\Delta\vu\|^2_{\alpha},
\end{align*}
where
\begin{align*}
K_0 = & \; 1-2 |\alpha| \theta^3(\alpha-\tfrac{1}{2})^2, \\
 K_1=& \; 1 - C_1 |\alpha|\theta - \tfrac{C_2 }{c_3}|\alpha| \theta -  D_3 |\alpha \theta|^2 - c_3 | \alpha\theta| E_1 -  c_3 C_3 |\alpha| \theta^3 ,
 \\ K_2=&\;  1 - D_1 | \alpha| \theta  -  \tfrac{ D_2 |\alpha \theta|^2}{c_3}  -  c_3(E_2+1) .
\end{align*}
For coercivity we need $K_0, K_1,K_2>0$. To this end, choose $c_3 >0$  small enough so that the condition $ c_3(E_2+1)  < \tfrac{1}{4} $ is satisfied.

If we choose $  |\alpha \theta| < c  $ with $ c $ sufficiently small, we obtain $ K_0, K_1, \, K_2\geq \half$.
The result follows upon noting that for $\alpha \in \RR\setminus \ZZ$ we have the equivalence 
\begin{equation*}%
\|\vu \|_{\mathfrak{X}^2_{\alpha,\theta} }^2  \sim_{\alpha}  |\ur|^2_{\alpha}+\theta^3|r\dr\ur|^2_{\alpha}+\|\grad\vu\|^2_{\alpha}+\theta^3\|\grad r\dr\vu\|^2_{\alpha}+\theta^3\|r\del \vu \|^2_{\alpha}.\qedhere
\end{equation*}

}

\end{proof}

\begin{proof}[Proof of Proposition \ref{prop:boundedness_ch5}] We prove the boundedness of the bilinear form $B$, i.e.,
\begin{equation*}
    |B(\vu,\vv)|\leq C\|\vu\|_{\mathfrak{X}^2_{\alpha,\theta}}\|\vv\|_{\mathfrak{X}^2_{\alpha,\theta}}\qquad \text{ for all }\vu,\vv\in\TT.
\end{equation*}
First consider the terms in the bilinear form $B_1$ as defined in \eqref{eq:B1}. It is immediate that $T_1^{(1)}$, $T^{(1)}_2$ and $T^{(1)}_3$can be bounded by applying the Cauchy-Schwarz inequality and in addition by Hardy's inequality for $T^{(1)}_3$. From \eqref{eq:lemma_92}, \eqref{eq:est_dph^2ur}, Lemma \ref{lem:estimate_Fourier_Phi}, Corollary \ref{cor:estimate_Fourier_Phi:extra} and Lemma \ref{lem:new_estimates_uph_ur} it follows
\begin{equation*}
  \big|T^{(1)}_4\big|\leq \big\|\tfrac{1}{r}\dph^2 \ur\big\|_{\alpha}\big\|\tfrac{1}{r}([\P, r^{-2\alpha}]\vv)_r\big\|_{-\alpha}+
  \big\|\tfrac{1}{r}\dph\ur\big\|_{\alpha}\big\|\tfrac{1}{r}\dph ([\P, r^{-2\alpha}]\vv)_r \big\|_{-\alpha}\leq C\|\vu\|_{\mathfrak{X}^2_{\alpha,\theta}}\|\vv\|_{\mathfrak{X}^2_{\alpha,\theta}}.
\end{equation*}
For the last term $T_5^{(1)}$,
using the Cauchy-Schwarz inequality and Corollary \ref{cor:estimate_Fourier_Phi:extra} gives
\begin{align*}
    \big|T_5^{(1)}\big|\leq \Big(\int_{\Om}r^{2\alpha}|(\grad [\P, r^{-2\alpha}]\vv)|^2\dd x\Big)^{\half}\Big(\int_{\Om}r^{-2\alpha}|\grad\vu|^2\dd x\Big)^{\half}\leq C\|\vu\|_{\mathfrak{X}^2_{\alpha,\theta}}\|\vv\|_{\mathfrak{X}^2_{\alpha,\theta}}.
\end{align*}
We continue with boundedness of the terms in  the bilinear form $B_2$ as defined in \eqref{eq:B2}. Again, boundedness of $T_1^{(2)}$, $T^{(2)}_2$, $T^{(2)}_3$ and $T^{(2)}_4$ follow immediately from the Cauchy-Schwarz inequality. The terms  $T^{(2)}_5$ and $T^{(2)}_6$ can be bounded in a similar way as $T^{(1)}_4$ and $T^{(1)}_5$, respectively.
Finally, the bound of $B_3$ is a consequence of the Cauchy-Schwarz inequality and Lemma \ref{lem:est_P_divw0}.
\end{proof}

%%%%%%%%%%%%%%%%%%%%%%%%%%%%%%%

\section{Equivalence of strong and variational solutions}
\label{section:SSSP}

In this section we prove Proposition \ref{equi:of:sol} from Section \ref{sec:3}, Step 3. We start with some technical lemmata.
Recall that $\prescript{k}{}{\mathcal{H}}^{k}_{\alpha}$ is the closure of $C_{\mathrm{c}}^{\infty}(\overline{\Om}\setminus\{0\})$ with respect to $\llbracket \cdot\rrbracket_{k,\alpha}$ as defined in Section \ref{subsec:mainresults}.
\begin{lem}
 \label{app:sequence}
{\br
Let $ \eps \in (0,1) $. There is a constant $ c > 0 $ such that for any $ \theta \in (0, (1-\eps)\pi) $ and $ \alpha \in I_{\eps} \setminus \mathbb{Z} $ satisfying $ |\alpha \theta | < c $, for  $c_3>0$ the constant in \eqref{eq:var_form_ch3} and for any  $ \ww \in \TT $, there exists a sequence $ (\vv_k)_{k \geq 1} $ such that }
 \begin{enumerate}[label=(\roman*)]

 \item $ \vv_k \in \{ \vv \in \TT: \vv = 0 $ on $ \partial \Omega' \} \qquad$ for all $k\geq 1$,

 \item $  \P\left[ r^{-2\alpha}\left(\vv_k + {\br |\alpha|\theta^3}\vv_{k,2}+ c_3{\br |\alpha|\theta^3}\vv_{k,3} \right)\right] \longrightarrow \ww \qquad$ in $ \prescript{0}{}{\mathcal{H}}^{0}_{1-\alpha}  $ as $k\longrightarrow\infty$,

 \end{enumerate}
 where $\vv_{k,2}=- (r\dr)^2 \vv_k$ and $\vv_{k,3}=-r^2\del\vv_k$ (cf. \eqref{eq:testfunction_B2} and \eqref{eq:testfunction_vorticity_bilinear_form}).
 \end{lem}
Lemma \ref{app:sequence} is a consequence of the following two steps. Firstly, in Lemma \ref{lem:sol_dual_problem} we prove that for any $ \ww \in \TT$ there exists a solution $ \vv $ to the test function problem
\begin{equation}\label{eq:dual_problem1}
\begin{aligned}
r^{-2\alpha}\left(\vv- {\br |\alpha|\theta^3} (r\partial_r)^2\vv - c_3 {\br |\alpha|\theta^3} r^{2}\Delta \vv   \right) + \nabla p = & \, \ww \quad && \text{ in } \Omega, \\
\div \vv = & \, 0 \quad && \text{ in } \Omega, \\
\vv = & \, 0 \quad && \text{ on } \partial \Omega.
\end{aligned}
\end{equation}
Note that this problem has a scaling invariant boundary condition. Secondly, this solution $\vv$ can be approximated by a sequence $(\vv_k)_{k\geq 1}$ in $\TT$ such that $ \vv_k = 0 $ on $ \partial \Omega' $ and that the convergence in Lemma \ref{app:sequence} holds.

\begin{lem}\label{lem:sol_dual_problem}
Let $ \eps \in (0,1) $. There is a constant $ c > 0 $ such that for any $ \theta \in (0, (1-\eps)\pi) $ and $ \alpha \in I_{\eps} \setminus \mathbb{Z} $ satisfying $ |\alpha \theta | < c $, for  $c_3>0$ the constant in \eqref{eq:var_form_ch3}, there exists a unique solution $ (\vv,p) \in \prescript{2}{}{\mathcal{H}}^{2}_{\alpha-1} \times \prescript{1}{}{\mathcal{H}}^{1}_{1-\alpha} $ to \eqref{eq:dual_problem1}. Furthermore 
\begin{equation*}
\llbracket \vv \rrbracket_{2, \alpha-1}+ \llbracket p \rrbracket_{1, 1-\alpha} \leq C_{\alpha, \theta} \left(\llbracket \ww \rrbracket_{0,1-\alpha} + \llbracket \partial_{\varphi} \ww \rrbracket_{0,1-\alpha} \right).
\end{equation*}
\end{lem}

\begin{proof} For notational convenience we write $\widehatvr:=\widehatvr(\lambda,\ph)$ and $\widehatvph : = \widehatvph(\lambda,\ph)$ or we omit the $\ph$-dependence from the notation.  For $I\subset \RR 	\setminus \{-1\}$ an open interval, let
\begin{equation}\label{eq:Omhat}
  \mathcal{S} := \{ (\lambda, \varphi) \in \mathbb{C} \times (0,\theta) : \Re \lambda \in I \}.
\end{equation}
Problem \eqref{eq:dual_problem1} reads in Mellin variables (see \eqref{eqapp:del_polar})
\begin{subequations}
\begin{alignat}{2}
-  c_3{\br |\alpha|\theta^3}((\lambda^2 + \partial_{\varphi}^2)\hat{v}_{r} - 2 \partial_{\varphi} \hat{v}_{\varphi} -\hat{ v}_r) -  {\br |\alpha|\theta^3}  \lambda^2 \hat{v}_r + \hat{v}_r + (\lambda -2\alpha+1)\hat{ p} = & \, \hat{w}_r(\lambda -2 \alpha) &&  \text{ in }\mathcal{S}, \label{dynamics:66:1}\\
-  c_3{\br |\alpha|\theta^3}((\lambda^2 + \partial_{\varphi}^2)\hat{v}_{\varphi} + 2 \partial_{\varphi} \hat{v}_{r} -\hat{v}_{\varphi}) -  {\br |\alpha|\theta^3} \lambda^2 \hat{v}_{\varphi} +\hat{v}_{\varphi} +\partial_{\varphi} \hat{p} = & \, \hat{w}_{\varphi}(\lambda-2\alpha)  && \text{ in } \mathcal{S}, \label{dynamics:66:2} \\
(\lambda + 1)\hat{v}_r + \partial_{\varphi} \hat{v}_{\varphi } = & \, 0  && \text{ in } \mathcal{S}, \label{dynamics:66:3}\\
\hat {v}_{r } = \hat{v}_{\varphi} = & \,  0  && \text{ on } \{ 0, \theta \}.  \label{dynamics:66:4}
\end{alignat}
\end{subequations}
Then using \eqref{dynamics:66:3} in $(\lambda+1)[\dph \eqref{dynamics:66:1}-(\lambda -2\alpha +1 )\eqref{dynamics:66:2}]$ gives for $ \widehatvph$ the equation
\begin{subequations}\label{sys:hat:after:teaching}
\begin{alignat}{2}
a_1 \partial_{\varphi}^4\hat{v}_{\varphi} + a_2 \partial_{\varphi}^2 \hat{v}_{\varphi}  + a_3  \hat{v}_{\varphi} &=  \, (\lambda+1 )\big[ \partial_{\varphi}\hat{w}_{r}(\lambda-2\alpha) -(\lambda-2\alpha +1)\hat{w}_{\varphi}(\lambda-2\alpha)\big] && \text{ in } \mathcal{S}, \label{sys:hat:after:teaching:1}\\
  \hat{v}_{\varphi} = \partial_{\varphi}\hat{v}_{\ph} &=  \, 0 && \text{ on } \partial \mathcal{S},  \label{sys:hat:after:teaching:2}
\end{alignat}
\end{subequations}
where
\begin{gather*}
a_1 := c_3{\br |\alpha|\theta^3}, \quad
a_2:=   c_3{\br |\alpha|\theta^3}\big( (\lambda+1)^2 +(\lambda-1)(\lambda-2\alpha+1)\big) +  {\br |\alpha|\theta^3} \lambda^2-1 \quad \text{ and }\\
a_3 := (\lambda -2\alpha +1 )(\lambda +1 )\big( c_3{\br |\alpha|\theta^3} (\lambda^2 -1) +  {\br |\alpha|\theta^3}  \lambda^2 - 1\big).
\end{gather*}
To solve \eqref{sys:hat:after:teaching}, set $ \lambda = \alpha + i s $, with $ s \in \mathbb{R} $. Then we obtain
\begin{align*}
\Re\: a_1 =    & \,  c_3{\br |\alpha|\theta^3} > 0,  \\
\Re \: a_2 = & \, ( 4\alpha-2 s^2) c_3{\br |\alpha|\theta^3}+(\alpha^2-s^2) {\br |\alpha|\theta^3}-1, \\
\Re \: a_3 = & \, ({\br s^4- \alpha^4 + 2 \alpha^2 - 4s^2|\alpha|})c_3{\br |\alpha|\theta^3}+ ({\br s^4-\alpha^4 -s^2 + \alpha^2 - 4s^2|\alpha|}){\br |\alpha|\theta^3} + s^2 + (\alpha^2 -1).
\end{align*}
Thus, for any fixed $ c_3 $ and $ \delta > 0$, there exists  $ c(c_3,\delta )  > 0  $ sufficiently small such that for any $ |\alpha \theta| <   c $, we have
\begin{align*}
\Re \: a_2 <  \,- c_3{\br |\alpha|\theta^3} s^2 - (1-\delta) \quad \text{ and }\quad
\Re \: a_3 >  \, \frac{c_3{\br |\alpha| \theta^3}}{2}s^4- (1+\delta).
\end{align*}
Let $ \psi \in C^{\infty}_{\mathrm{c}}((0,\theta); \mathbb{C}) $. Multiply \eqref{sys:hat:after:teaching:1} by $ \overline{\psi} $, integrate over $ (0,\theta) $ and after integration by parts we obtain the variational formulation
\begin{align*}
&\;a \int_0^{\theta} \partial_{\varphi}^2   \hat{v}_{\varphi} \partial_{\varphi}^2 \overline{\psi} \, \dd \varphi -  b \int_0^{\theta} \partial_{\varphi} \hat{v}_{\varphi} \partial_{\varphi} \overline{\psi}  \, \dd\varphi  + d  \int_0^{\theta}  \hat{v}_{\varphi} \overline{\psi}  \, \dd\varphi \\
= & \; - \int_0^{\theta} (\lambda+1)\hat{w}_r(\lambda-2\alpha) \partial_{\varphi}\overline{\psi} + (\lambda +1) (\lambda -2\alpha+1 ) \hat{w}_{\varphi}(\lambda-2\alpha) \overline{\psi} \, \dd \varphi.
\end{align*}
The left-hand side of the above equation is a bounded, coercive bilinear map on the closure of $ C^{\infty}_{\mathrm{c}}((0,\theta); \mathbb{C})$ in $ H^2(0,\theta ) $. The coercivity follows from
\begin{align*}
&\Re \: a_1  \int_0^{\theta}| \partial_{\varphi}^2  \hat{v}_{\varphi}|^2 \dd \varphi -  \Re \: a_2 \int_0^{\theta} |\partial_{\varphi} \hat{v}_{\varphi}|^2  \dd\varphi  + \Re \: a_3 \int_0^{\theta}  |\hat{v}_{\varphi}|^2    \dd\varphi \\
 \geq &\;  c_3{\br |\alpha| \theta^3} \int_0^{\theta}| \partial_{\varphi}^2  \hat{v}_{\varphi}|^2  \dd \varphi+ \big( c_3{\br |\alpha| \theta^3}s^2 +(1-\delta) \big) \int_0^{\theta} |\partial_{\varphi} \hat{v}_{\varphi}|^2  \dd\varphi  +\Big(\frac{ c_3{\br |\alpha| \theta^3}}{2}s^4- (1+\delta)\Big) \int_0^{\theta}  |\hat{v}_{\varphi}|^2    \dd\varphi \\
 \geq &\;  c_3{\br |\alpha| \theta^3} \int_0^{\theta}| \partial_{\varphi}^2  \hat{v}_{\varphi}|^2  \dd \varphi+ c_3{\br |\alpha| \theta^3} s^2 \int_0^{\theta} |\partial_{\varphi} \hat{v}_{\varphi}|^2  \dd\varphi  +\frac{ c_3{\br |\alpha| \theta^3}}{2}s^4 \int_0^{\theta}  |\hat{v}_{\varphi}|^2    \dd\varphi,
\end{align*}
where the last inequality holds by Poincar\'e's inequality from Lemma \ref{lem:poincare}, for $\theta < (1-\eps)\pi $ and a sufficiently small $ \delta $. Application of the Lax-Milgram theorem implies the existence of a unique solution. Moreover,
\begin{align*}
 \int_0^{\theta}| \partial_{\varphi}^2  \hat{v}_{\varphi}|^2 \, \dd \varphi+  \int_0^{\theta} |\lambda|^2|\partial_{\varphi} \hat{v}_{\varphi}|^2 \, \dd\varphi  + \int_0^{\theta}  |\lambda|^4|\hat{v}_{\varphi}|^2   \, \dd\varphi \leq C_{\alpha, \theta} \int_{0}^{\theta}|\hat{\ww}(\lambda-2\alpha)|^2 \dd \varphi.
\end{align*}
Using \eqref{sys:hat:after:teaching:1} and the fact that $ \left| \tfrac{a_2}{\lambda^2} \right| + \left| \tfrac{a_3}{\lambda^4} \right| \leq C_{\alpha,\theta}  $, we also deduce that 
\begin{equation}
\tfrac{\| \partial_{\varphi}^4  \hat{v}_{\varphi} \|_{L^2(0,\theta)}^2}{|\lambda|^4}  \leq C_{\alpha,\theta}\big(\| \hat{\ww}(\lambda-2\alpha) \|_{L^2(0,\theta)}^2+\| \partial_{\varphi} \hat{\ww}(\lambda-2\alpha) \|_{L^2(0,\theta)}^2 \big). \label{mild:1}
\end{equation}
Moreover, by interpolation, it follows
\begin{align}
\tfrac{\| \partial_{\varphi}^3  \hat{v}_{\varphi} \|_{L^2(0,\theta)}^2}{|\lambda|^2} \leq & \, C_{\alpha,\theta}\Big(\tfrac{\| \partial_{\varphi}^4  \hat{v}_{\varphi} \|_{L^2(0,\theta)}^2}{|\lambda|^4} + \| \partial_{\varphi}^2  \hat{v}_{\varphi} \|_{L^2(0,\theta)}^2\Big)  \nonumber \\  \leq & \,  C_{\alpha,\theta}\big(\| \hat{\ww}(\lambda-2\alpha) \|_{L^2(0,\theta)}^2+\| \partial_{\varphi} \hat{\ww}(\lambda-2\alpha) \|_{L^2(0,\theta)}^2 \big).  \label{mild:2}
\end{align}
We recover $ v_r $ from \eqref{dynamics:66:3} and from \eqref{mild:1} and \eqref{mild:2} we deduce the desired estimates. The pressure $ p $ is defined via \eqref{dynamics:66:1}, and the proof of the estimate is straight-forward.

\end{proof}
\begin{proof}[Proof Lemma \ref{app:sequence}]
 Let  $ \ww \in \TT$, then Lemma \ref {lem:sol_dual_problem} ensures the existence of a solution $ \vv \in \prescript{2}{}{\mathcal{H}}^{2}_{\alpha-1} $ to \eqref{eq:dual_problem1} such that $ \div \vv = 0 $ in $ \Omega $ and $ \vv = 0 $ on $ \partial \Omega' $. Proposition \ref{prop:ind:spaces} ensures that there exists a sequence $ ( \vv_k )_{k\geq 1} $ such that $  \vv_k \in \TT$ with $ \vv_k = 0 $ on $ \partial \Omega' $ for all $k\geq 1$ and
  $ \vv_k \longrightarrow  \vv $ in $\prescript{2}{}{\mathcal{H}}^{2}_{\alpha-1} $ as $k\longrightarrow\infty$.
By definition of $ \prescript{2}{}{\mathcal{H}}^{2}_{\alpha-1}  $ (Section \ref{subsec:mainresults}), this convergence implies that for $ 0 \leq j + \ell \leq 2 $, $ j, \ell \geq 0 $ we have
 $ (r\partial_{r})^j \partial_{\varphi}^{\ell} \vv_{k } \in \prescript{0}{}{\mathcal{H}}^{0}_{1+\alpha}. $
Since $ \vv_2 $ and $ \vv_3 $ are linear combinations of $ r\partial_{r} $ and $ \partial_{\ph } $ derivatives of $ v_{r} $ and $ v_{\varphi} $, we deduce that
 \begin{equation*}
  r^{-2\alpha}\left(\vv_k + {\br |\alpha| \theta^3} \vv_{k,2}+c_3{\br |\alpha| \theta^3} \vv_{k,3}\right) \longrightarrow   r^{-2\alpha}\left(\vv + {\br |\alpha| \theta^3}\vv_{2}+ c_3{\br |\alpha| \theta^3}\vv_{3}\right)\quad \text{ in } \prescript{0}{}{\mathcal{H}}^{0}_{1-\alpha},
 \end{equation*}
 as $k\longrightarrow\infty$. Continuity of the projection $ \P $ in $\prescript{0}{}{\mathcal{H}}^{0}_{2-\alpha} $ by Lemma \ref{lem:est_P_divw0} implies that
  \begin{equation*}
 \P\left[ r^{-2\alpha}\left(\vv_k + {\br |\alpha| \theta^3}\vv_{k,2}+c_3{\br |\alpha| \theta^3}\vv_{k,3}\right)\right] \longrightarrow   \P\left[r^{-2\alpha}\left(\vv + {\br |\alpha| \theta^3}\vv_{2}+ c_3{\br |\alpha| \theta^3}\vv_{3}\right)\right] = \ww,
 \end{equation*}
 in $\prescript{0}{}{\mathcal{H}}^{0}_{1-\alpha}$ as $k\longrightarrow\infty$.
 \end{proof}

 %%%%%%%%%%%%%%%%%%%%%%%%%%%%%%%%%%%%%%%%%%%%%%%%%%%%%%%%%%%

We can now finish the proof of Proposition \ref{equi:of:sol}.

\begin{proof}[Proof of Proposition \ref{equi:of:sol}]
It is enough to show that $\vu$ is a strong solution to the projected Stokes equations \eqref{eq:SmoothStokesProj_ch5} if and only if $\vu$ is a solution to the variational problem \eqref{la:mil:equ:2}. In fact, if $ \vu $ satisfies \eqref{eq:SmoothStokesProj_ch5}, then the pressure $ p $ is the unique solution in $ \mathscr{Y}^1_{\alpha} $ of
\begin{equation*}
\nabla p = \Delta \vu + \vf -\P(\Delta \vu + \vf),
\end{equation*}
which is exactly \eqref{pres:equ:var}.

In Lemma \ref{reg:sol:are:var} we have already proved that an almost everywhere solution $\vu$ to \eqref{eq:SmoothStokesProj_ch5}
 satisfies the variational problem
\begin{equation}\label{eq:var_problem_CH7}
B(\vu, \vv) = (\P\vf, r^{-2\alpha}\vv_{\text{test}})_{L^2(\Omega)} + \langle g, v_r \rangle_{\alpha} \qquad \text{ for any }  \vv \in \TT,
\end{equation}
where $ \vv_{\text{test}} $ is defined in \eqref{eq:testfunctions}.

 It remains to verify that a solution $ \vu $ to \eqref{eq:var_problem_CH7} is also an almost everywhere solution to the Stokes problem \eqref{eq:SmoothStokesProj_ch5} with Navier slip.
Let  $ \ww \in \TT$. Then there exists a sequence $ (\vv_k)_{k\geq 1} $ that satisfies the assumptions of Lemma \ref{app:sequence}. From \eqref{eq:var_problem_CH7} we deduce that
\begin{align*}
 0 =  \int_{\Omega} (- \P \Delta \vu - \P\vf) \cdot  \P\left[ r^{-2\alpha}\left(\vv_k +  {\br |\alpha|\theta^3}\vv_{k,2}+ c_3{\br |\alpha|\theta^3}\vv_{k,3} \right)\right]  \dd x \longrightarrow  \int_{\Omega} (- \P \Delta \vu - \P\vf) \cdot \ww  \dd x.
 \end{align*}
 This implies that for all $ \ww \in \TT$
 \begin{equation*}
\label{goal:to show}
 \int_{\Omega} (- \P \Delta \vu -\P \vf) \cdot \ww \, \dd x = 0.
 \end{equation*}
By the fundamental lemma of calculus of variations we obtain $ -\P \Delta \vu = \P\vf $ almost everywhere in $ \Omega $.

We next verify that the Navier-slip condition holds. For any $\vv\in \TT$, $\vu$ satisfies the equation
\begin{equation}\label{eq:result_b}
\begin{aligned}
        T_1^{(1)}+ &\;{\br |\alpha|\theta^3}T^{(2)}_1+\sum_{j=2}^5 T^{(1)}_j +{\br |\alpha|\theta^3}\sum_{j=2}^6 T^{(2)}_j+ c_3{\br |\alpha|\theta^3}\sum_{j=1}^2 T^{(3)}_j\\&=(\vc{f},r^{-2\alpha}(\vv+{\br |\alpha|\theta^3}\vv_2+ c_3{\br |\alpha|\theta^3} \vv_3))_{L^2(\Omega)} + \langle g, v_r\rangle_{\alpha}.
        \end{aligned}
\end{equation}
Recall from the derivation of the bilinear forms in Section \ref{chap:5} that we only applied the Navier-slip boundary condition to get $T_1^{(1)}$ and $T_1^{(2)}$. Using the smoothness of the test function $\vv$, we undo the integration by parts for $B_1$, $B_2$ and $B_3$ as in Section \ref{sec:weighted_bilienar_form}, but in the opposite direction. This gives
\begin{align*}
    \sum_{j=2}^5T^{(1)}_j +{\br |\alpha|\theta^3}\sum_{j=2}^6T^{(2)}_j+ c_3{\br |\alpha|\theta^3}\sum_{j=1}^4T^{(3)}_j
    =&\; \big(-\P\del\vu, r^{-2\alpha}(\vv+{\br |\alpha|\theta^3}\vv_2+ c_3{\br |\alpha|\theta^3}\vv_3)\big)_{L^2(\Om)}\\&\;-{\br |\alpha|\theta^3}\int_{\dOm'}r^{-2\alpha}((r\dr)^2\vr)( \dn \ur) \dd s\\&\;+\int_{\dOm'}r^{-2\alpha}\vr\dn\ur\dd s.
\end{align*}
 Substituting this into \eqref{eq:result_b} and using that $ -\P \del\vu = \P\vf $ is satisfied almost everywhere in $\Om$
gives
\begin{equation*}
    \int_{\dOm'}r^{-2\alpha}\big(\vr(\ur+\dn\ur-g) -  |\alpha|\theta^3((r\dr)^2\vr) (\ur+\dn\ur-g)\big)\dd s=0.
\end{equation*}
We obtain that $\ur+\dn\ur= g$ on $\dOm'$ almost everywhere if enough test functions are generated. To this end, it suffices to show that for $ w \in C^{\infty}_{\mathrm{c}}((0,\infty)) $ there exist a $\vr \in C^{\infty}((0,\infty)) $ with $ v_r(0) = 0  $ and decay to zero at infinity, which is a solution to
\begin{equation*}
    r^{-2\alpha}\big(1 -{\br |\alpha|\theta^3} (r\dr)^2\big)\vr=w.
\end{equation*}
In Mellin variables this equation has the solution
$$ \hat{v}_r(\lambda) = \frac{1}{1-{\br |\alpha|\theta^3} \lambda^2} \hat{w}(\lambda-2\alpha).$$
For $|\alpha \theta| $ small enough, we have
$$   \frac{1}{1-{\br |\alpha|\theta^3} (\Re \lambda)^2} \leq C , $$
so that the inverse Mellin transform can be used to obtain the desired solution $ \vr$. Hence, with the fundamental lemma of calculus of variations we conclude that $\ur+\dn\ur=g$ almost everywhere on $\dOm'$.
\end{proof}

%%%%%%%%%%%%%%%%%%%%%%%%%%%%%%%%%%%%%%%%%%%%%%%%%%%%%%%%%%%%%%%%%%%%%%%%%%%%%%%%

\section{The Stokes equations with free-slip boundary conditions}\label{sec:higher_reg}

In this section we study the Stokes equations in a wedge with non-homogeneous free-slip boundary conditions, i.e.,
\begin{subequations}\label{SsS:merlin}
\begin{alignat}{5}
-r^{-2}[((r \partial_r)^2 + \partial_{\varphi}^2)u_r - 2\partial_{\varphi} u_{\varphi} - u_r] + \partial_r p = \, & f_r \quad && \text{ for } r > 0, \varphi \in (0,\theta), \\
-r^{-2}[((r\partial_r)^2 + \partial_{\varphi}^2)u_{\varphi} + 2 \partial_{\varphi} u_r - u_{\varphi}] + r^{-1}\partial_{\varphi} p = & \, f_{\varphi} \quad && \text{ for } r > 0, \varphi \in (0,\theta), \\
(r \partial_r +1) u_r + \partial_{\varphi} u_{\varphi} = & \, 0 \quad && \text{ for } r > 0, \varphi \in (0,\theta), \\
u_{\varphi} = & \, 0  \quad && \text{ for } r > 0, \varphi \in \{0,\theta\}, \\
\partial_{\varphi} u_{r} = & \, \mathfrak{g}  \quad && \text{ for } r > 0, \varphi \in \{0,\theta\},
\end{alignat}
\end{subequations}
where $ f_r $, $ f_{\varphi}$ and $ \mathfrak{g} $ are given data. Studying this system is required to gain higher regularity for the regular problem as was discussed in Section \ref{sec:reg_problem}. As mentioned before, to our knowledge a closed solution representation in Mellin variables for the original system with Navier slip is not available and cannot be expected. Therefore, we study the system above and consider as data $\mathfrak{g}=\pm(g - r\ur) $, where $\ur$ is the strong solution determined in Theorem \ref{Floris:theo}. \\

Again, for notational convenience we write $\hat{h}:=\hat{h}(\lambda,\ph)$ or we omit the $\ph$-dependence from the notation. In Section \ref{sec:1}, we derive a representation formula for solutions to the above system in Mellin variables and in Sections \ref{sec:92} and \ref{sec:93} we study the regularity of these solution and complete the proof of Proposition \ref{my:step} in Section \ref{sec:2}.

\subsection{A representation formula}
\label{sec:1}

We rewrite the above system in Mellin variables, that is
\begin{subequations}\label{SS:Mellin}
\begin{alignat}{5}
(\lambda^2 + \partial_{\varphi}^2)\hat{u}_r(\lambda) -2 \partial_{\varphi} \hat{u}_{\varphi}(\lambda) -\hat{u}_r(\lambda) - (\lambda-1) \hat{ p}(\lambda-1) = \, & - \hat{f}_r(\lambda-2) \quad && \text{ in } \mathcal{S}, \label{SS:Mellin_1}  \\
(\lambda^2 + \partial_{\varphi}^2)\hat{u}_{\varphi}(\lambda) + 2 \partial_{\varphi} \hat{u}_r(\lambda) -\hat{ u}_{\varphi}(\lambda)  - \partial_{\varphi} \hat{p}(\lambda-1) = & \, -\hat{f}_{\varphi}(\lambda-2) \quad && \text{ in } \mathcal{S},   \label{SS:Mellin_2} \\
(\lambda +1) \hat{u}_r(\lambda)  + \partial_{\varphi} \hat{u}_{\varphi}(\lambda)  = & \, 0 \quad && \text{ in }  \mathcal{S},  \label{SS:Mellin_3}\\
\hat{u}_{\varphi}(\lambda)  = & \, 0  \quad && \text{ on }  \partial \mathcal{S},\label{SS:Mellin_4}  \\
\partial_{\varphi} \hat{u}_{r}(\lambda)  = & \, \hat{\mathfrak{g}}(\lambda)   \quad && \text{ on }  \{0,\theta \},   \label{SS:Mellin_5}
\end{alignat}
\end{subequations}
where $ \mathcal{S}$  is defined in \eqref{eq:Omhat}. Then using \eqref{SS:Mellin_3} in $-(\lambda+1)[\dph \eqref{SS:Mellin_1}-(\lambda -1 )\eqref{SS:Mellin_2}]$ gives the equation for $\hat{u}_{\varphi}(\lambda,\ph)$
\begin{equation}\label{fourth:order:equ}
\begin{aligned}
\partial_{\varphi}^4 \hat{u}_{\varphi} + 2(\lambda^2+1)\partial_{\varphi}^2 \hat{u}_{\varphi} +(\lambda^2-1)^2 \hat{u}_{\varphi}  &=   (\lambda^2-1)\hat{f}_{\varphi}(\lambda-2)- (\lambda+1) \partial_\varphi \hat{f}_r(\lambda-2)  && \text{ in }  \mathcal{S},  \\
\hat{u}_{\varphi} =  \partial_{\varphi}^2 \hat{u}_{\varphi}  +  (\lambda + 1 ) \hat{\mathfrak{g}} & = 0 && \text{ on } \partial \mathcal{S}.
\end{aligned}
\end{equation}
With a solution $ \hat{u}_{\varphi} $ it is straightforward to recover $ \hat{u}_r $ and $\hat{p}$ via \eqref{SS:Mellin_3} and \eqref{SS:Mellin_1}, respectively.

We determine a Green's function that allows us to write the solution $ \hat{u}_{\varphi} $ of \eqref{fourth:order:equ} in terms of the source terms and boundary conditions.

\begin{lem}
\label{Lemma:green:fun}
There exists a Green's function $ G(\lambda, \varphi, \varphi' ) : \mathbb{C} \times (0,\theta)^2 \to \mathbb{R}$, symmetric in $ \varphi $ and $\varphi' $, such that
\begin{equation*}
    \begin{aligned}
\partial_{\varphi}^4 G + 2(\lambda^2+1)\partial_{\varphi}^2 G +(\lambda^2-1)^2 G  &=   \delta_{\varphi}  \quad&& \text{ in }  \mathcal{S}, \nonumber\\
G =  0 \quad \text{ and }\quad \partial_{\varphi}^2 G &= 0 \quad  &&\text{ for  } \varphi \in \{ 0, \theta \}. \nonumber
\end{aligned}
\end{equation*}
Specifically,
\begin{align*}
G (\lambda, \varphi, \varphi') =  \begin{cases}  \frac{\sin((\lambda-1)(\theta- \varphi)) \sin((\lambda-1)\varphi')}{4(\lambda-1)\lambda \sin((\lambda-1)\theta)}- \frac{\sin((\lambda + 1)(\theta- \varphi)) \sin((\lambda+1)\varphi')}{4(\lambda+1)\lambda \sin((\lambda+1)\theta)} \quad \text{ for } \varphi' \leq \varphi , \\
\\
\frac{\sin((\lambda-1)(\theta- \varphi')) \sin((\lambda-1)\varphi)}{4(\lambda-1)\lambda \sin((\lambda-1)\theta)} - \frac{\sin((\lambda + 1)(\theta- \varphi')) \sin((\lambda+1)\varphi)}{4(\lambda+1)\lambda \sin((\lambda+1)\theta)}  \quad \text{ for } \varphi' \geq \varphi.
\end{cases}
\end{align*}
\end{lem}
\begin{proof}
The uniqueness follows from standard ODE theory. Thus the formula can be verified a posteriori.
\end{proof}

The Green's function satisfies the following property.
\begin{lem}
\label{prop:4} Let $G$ be the Green's function from Lemma \ref{Lemma:green:fun}.
For any fixed $ \varphi, \varphi' \in [0,\theta] $, the functions $$ G(\lambda, \varphi, \varphi') , \quad \partial_{\varphi'} G(\lambda, \varphi, \varphi')  \quad\text{ and }\quad\partial_{\varphi}\partial_{\varphi'} G(\lambda, \varphi, \varphi') $$ 
are holomorphic in
 \begin{equation*} \Sigma:= \{ \lambda = \beta + i s:  (|\beta| +1)\theta < \pi, s\in\RR \}. \end{equation*}
\end{lem}
\begin{proof}

For $ \varphi' \leq \varphi $, the function $ G  $ from Lemma \ref{Lemma:green:fun} is a sum and product of holomorphic functions as long as $\lambda\in \Sigma\setminus\{-1,0,1\}$.

Recall that $ \tfrac{2}{\pi} |x| \leq |\sin(x)| \leq  |x| $ for $ |x |  \leq \pi/2 $. For $ |\lambda +1 | < \tfrac{1}{2}$ we deduce that
\begin{align*}
\left|  \frac{\sin((\lambda + 1)(\theta- \varphi)) \sin((\lambda+1)\varphi')}{4(\lambda+1)\lambda \sin((\lambda+1)\theta)} \right| \leq C \frac{| \lambda +1 |^2|\theta-\varphi||\varphi'|}{|\lambda +1 |^2\theta} \leq C  \frac{|\theta-\varphi||\varphi'|}{\theta}, \label{b:b:prop:4}
\end{align*}
and therefore $ G $ is holomorphic at $ \lambda = - 1 $.  For $ \lambda =  1 $ one can argue analogously. For $\lambda=0$ note that $G(\lambda, \varphi, \varphi') = \lambda^{-1}\tilde{G}(\lambda,\ph,\ph')$ where
\begin{align*}
\tilde{G}(\lambda, \varphi, \varphi') = & \, \frac{\sin((\lambda-1)(\theta- \varphi)) \sin((\lambda-1)\varphi')}{4(\lambda-1) \sin((\lambda-1)\theta)}- \frac{\sin((\lambda + 1)(\theta- \varphi)) \sin((\lambda+1)\varphi')}{4(\lambda+1) \sin((\lambda+1)\theta)}.
\end{align*}%
Then $\tilde{G}$ is holomorphic at $\lambda=0$ and $\tilde{G}(\lambda,\ph,\ph')|_{\lambda=0}=0$.
This implies that $ G $ is holomorphic in a neighbourhood of $ \lambda = 0 $. To show that $ \partial _{\varphi' }G $ and  $ \partial _{\varphi }\partial _{\varphi' }G $ are holomorphic on $\Sigma$ we argue similarly.
\end{proof}

Using the Green's function from Lemma \ref{Lemma:green:fun} we obtain an expression for the solution of \eqref{fourth:order:equ}.
\begin{cor}
\label{cor:repr}
The unique classical  solution to \eqref{fourth:order:equ} is
\begin{equation*}\label{form:sol}
  \begin{aligned}
\hat{u}_{\varphi}(\lambda,\ph) = & \,  \int_{0}^{\theta}  (\lambda+1)  \partial_{\varphi'}G(\lambda, \varphi, \varphi ' ) \hat{f}_r(\lambda-2, \varphi') + (\lambda^2-1)G(\lambda, \varphi, \varphi ' )\hat{f}_{\varphi}(\lambda-2, \varphi')  \dd \varphi' \\ & + (\lambda + 1 ) \hat{\mathfrak{g}}(\lambda, \theta) \partial_{\varphi'}G(\lambda,\varphi, \theta) -  (\lambda + 1 ) \hat{\mathfrak{g}}(\lambda, 0) \partial_{\varphi'}G(\lambda,\varphi, 0),
\end{aligned}
\end{equation*}
where
\begin{gather*}
\partial_{\varphi'}G(\lambda,\varphi, \theta) =  \frac{\sin((\lambda + 1)\varphi) }{4 \lambda \sin((\lambda+1)\theta)} - \frac{\sin((\lambda-1)\varphi) }{4 \lambda \sin((\lambda-1)\theta)}
\end{gather*}
and
\begin{gather*}
\partial_{\varphi'}G(\lambda,\varphi, 0) =\frac{-\sin((\lambda + 1)(\theta-\varphi))}{4 \lambda \sin((\lambda+1)\theta)} + \frac{\sin((\lambda-1)(\theta-\varphi))}{4 \lambda \sin((\lambda-1)\theta)} .
\end{gather*}
\end{cor}
From the above representation of $ \hat{u}_{\varphi} $ and the divergence-free condition \eqref{SS:Mellin_3}, we also obtain a representation for $ \hat{u}_r $:
\begin{equation} \label{u:r:espression}
   \begin{aligned}
\hat{u}_{r}(\lambda,\ph) = & \,  \int_{0}^{\theta} \partial_{\varphi} \partial_{\varphi'} G(\lambda, \varphi, \varphi ' ) \hat{f}_r(\lambda-2, \varphi') + (\lambda-1)\partial_{\varphi}G(\lambda, \varphi, \varphi ' )\hat{f}_{\varphi}(\lambda-2, \varphi')  \dd \varphi' \\ & + \hat{\mathfrak{g}}(\lambda, \theta) \partial_{\varphi} \partial_{\varphi'} G(\lambda, \varphi, \theta) - \hat{\mathfrak{g}}(\lambda, 0) \partial_{\varphi} \partial_{\varphi'} G(\lambda,\varphi, 0).
\end{aligned}
\end{equation}

In the subsequent two sections we derive estimates on the solution to \eqref{SS:Mellin} and \eqref{fourth:order:equ}. For this we decompose $\vu=\vu^b+\vu^s$, where $\vu^b$ only has nonzero boundary data $ \mathfrak{g} $  (and $\vf=0$) and $\vu^s$ only has a nonzero source term $\vf$ (and $\mathfrak{g}=0$). Estimates for $\vu^b$ and $\vu^s$ are derived in Section \ref{sec:92} and \ref{sec:93}, respectively.

We show uniform estimates for $\Re \lambda $ in the interval $$ I_{\eps} = \left[ -(1-\eps)\tfrac{\pi}{\theta} + 1,(1-\eps)\tfrac{\pi}{\theta}- 1 \right],\quad \eps\in(0,1- \tfrac{\theta}{\pi })\text{ and }\theta\in(0,\pi), $$
which avoids the singularities of the Green's function at
 $ \frac{\pi}{\theta}- 1  $ and $ -\frac{\pi}{\theta} + 1$.

\subsection{Regularity of \eqref{SS:Mellin} and \eqref{fourth:order:equ} with $\vf=0$}\label{sec:92}

First, we use the Green's function representation formula to study the regularity of solutions of system \eqref{fourth:order:equ} in the case the source term is zero, i.e.,
\begin{equation}\label{fourth:order:equ:f=0}
\begin{aligned}
\partial_{\varphi}^4 \hat{u}_{\varphi} + 2(\lambda^2+1)\partial_{\varphi}^2 \hat{u}_{\varphi} +(\lambda^2-1)^2 \hat{u}_{\varphi}  &=   0  && \text{ in }  \mathcal{S},  \\
\hat{u}_{\varphi} =  \partial_{\varphi}^2 \hat{u}_{\varphi}  +  (\lambda + 1 ) \hat{\mathfrak{g}} & = 0 && \text{ on } \partial \mathcal{S}.
\end{aligned}
\end{equation}
In particular, we derive estimates of the $ \prescript{M+2}{}{\mathcal{H}}^{M+2}_{\alpha} $-norm of $ \vu $ with respect of the $ \prescript{M+1}{}{\mathcal{B}}^{M+1}_{\alpha+1} $-norm of $ \mathfrak{g} $. Recall that the $ \prescript{M+1}{}{\mathcal{B}}^{M+1}_{\alpha+1} $-norm is defined as the infimum of all the $ \prescript{M+1}{}{\mathcal{H}}^{M+1}_{\alpha+1} $ extension of $ \mathfrak{g} $. So let us denote by $ \mathfrak{g}_{\ext} $ such a possible extension.
\begin{lem} \label{lem:Bound:Green_Rep:STOKES} Let  $ \theta \in (0, \pi) $, $\eps \in (0, 1- \tfrac{\theta}{\pi} ) $, $\Re \lambda = \alpha\in I_{\eps} \setminus\ZZ$ and $\ell\in\NN$.
The solution $ \hat{u}_{\varphi} $ of problem \eqref{fourth:order:equ:f=0} satisfies the estimate
\begin{equation}\label{bound:uphi:uphi} 
    \| \partial_{\ph}^{\ell} \hat{u}_{\varphi}(\lambda,\ph)\|_{L^2(0,\theta)} \leq  C_{\eps}\frac{1}{\min\{1,|\alpha|\}^{\ell}} \frac{1}{ |\lambda|^{2-\ell}}(\|\lambda \hat{\mathfrak{g}_{\ext}}\|_{L^2(0,\theta)}+ \|\partial_{\varphi} \hat{\mathfrak{g}_{\ext}}\|_{L^2(0,\theta)}),\quad \Re\lambda\in I_{\eps}.
\end{equation}
Moreover, $ \hat{u}_r = -\frac{\partial_{\varphi} \hat{u}_{\varphi}}{\lambda+1} $ satisfies the estimate
\begin{equation}\label{bound:ur:ur} 
    \| \partial_{\ph}^{\ell} \hat{u}_{r}(\lambda,\ph)\|_{L^2(0,\theta)} \leq  C_{\eps}\frac{1}{\min\{1,|\alpha|\}^{\ell}} \frac{1}{ |\lambda|^{2-\ell}}(\|\lambda \hat{\mathfrak{g}_{\ext}}\|_{L^2(0,\theta)}+ \|\partial_{\varphi} \hat{\mathfrak{g}_{\ext}}\|_{L^2(0,\theta)}),\quad \Re\lambda\in I_{\eps}.
\end{equation}
\end{lem}

From the divergence-free condition $ \hat{u}_r = \frac{\partial_{\varphi} \hat{u}_{\varphi}}{\lambda+1} $ we can deduce bounds on $ \hat{u}_r $ from the one of $ \hat{u}_{\varphi} $.

\begin{cor}
\label{cor:est:bound:term}  Let $ \theta \in (0, \pi) $, $\eps \in (0, 1- \tfrac{\theta}{\pi} ) $, $\Re \lambda = \alpha\in \mathbb{R} \setminus\ZZ$ and $M\in\NN$ be such that $ M+ \alpha+1 \in I_{\eps}$.
Let $ \vu^b$ the solution of system \eqref{SS:Mellin} with $ \vf = 0 $. Then, we have the estimate
\begin{equation*}
\llbracket \vu^b \rrbracket_{M+2,\alpha} \leq \frac{C_{\eps,M}}{\min\{1, |\alpha+M+1|\}^{M+2}} [ \mathfrak{g} ]_{M+\frac{1}{2},\alpha+1}
\end{equation*}
\end{cor}
\begin{proof}
Let $ \mathfrak{g}_{\ext} \in \prescript{M+1}{}{\mathcal{H}}^{M+1}_{\alpha} $ such that $ \mathfrak{g}_{\ext}|_{\partial \Omega'} = \mathfrak{g} $.
Using Lemma \ref{lem:Bound:Green_Rep:STOKES}
 we have

\begin{align*}
\llbracket \vu^b \rrbracket_{M+2,\alpha}^2 = & \, \sum_{j+\ell= M+2} \int_0^{\theta} \int_{\Re \lambda = M + \alpha+1} |\lambda|^{2j}|\partial_{\ph}^{\ell} \hat{\vu^b}|^2  \dd \Im \lambda \dd\ph \\
=  & \,   \int_{\Re \lambda = M + \alpha+1} \sum_{j+\ell= M+2} |\lambda|^{2j}\|\partial_{\ph}^{\ell} \hat{\vu^b}\|^2_{L^2(0,\theta)}  \dd \Im \lambda \\
\leq & \, \int_{\Re \lambda = M + \alpha+1} \sum_{j+\ell= M+2} \frac{C_{\eps}^2}{\min\{1, |\alpha+M+1|\}^{2M+4}}\\
&\,\cdot|\lambda|^{2j+2\ell-4} \big( \|\lambda \hat{\mathfrak{g}_{\ext}}\|_{L^2(0,\theta)}^2+ \|\partial_{\varphi} \hat{\mathfrak{g}_{\ext}}\|_{L^2(0,\theta)}^2 \big)\dd \Im \lambda \\
\leq & \, \frac{C_{\eps,M}^2}{\min\{1, |\alpha+M+1|\}^{2M+4}} \int_{\Re \lambda = M +1 + \alpha+1 - 1} |\lambda|^{2M}\big( \|\lambda \hat{\mathfrak{g}_{\ext}}\|_{L^2(0,\theta)}^2+ \|\partial_{\varphi} \hat{\mathfrak{g}_{\ext}}\|_{L^2(0,\theta)}^2 \big) \dd \Im \lambda\\
\leq & \, \frac{C_{\eps,M}^2}{\min\{1, |\alpha+M+1|\}^{2M+4}} \llbracket \mathfrak{g}_{\ext} \rrbracket_{M+1,\alpha+1}^2.
\end{align*}
Taking the infimum over all the possible extensions $ \mathfrak{g}_{\ext} $ gives the result.
\end{proof}

It remains to prove Lemma \ref{lem:Bound:Green_Rep:STOKES}. 

\begin{proof} [Proof of Lemma \ref{lem:Bound:Green_Rep:STOKES}]
We start by proving the estimate for $ \hat{u}_{\varphi} $. Then we explain how to adapt the estimates to  $  \hat{u}_{r} =  -\tfrac{ \hat{u}_{\varphi}}{\lambda +1} $.

We divide the proof in two cases: $ \ell $ even and $ \ell $ odd. For $ \ell $ even, we recall that in the same spirit of \eqref{1.1}, the fundamental theorem of calculus and the Cauchy-Schwarz inequality imply the following trace estimate for $ \vartheta \in \{0, \theta\}$

\begin{align}\label{Mimi:trace:estimate}
|\lambda|| \hat{\mathfrak{g}_{\ext}}(\lambda, \vartheta)|^2
\lesssim  & \; \frac{1}{|\theta \lambda|} \|\lambda \hat{\mathfrak{g}_{\ext}}(\lambda, \cdot)|\|^2_{L^2(0,\theta)} + \|\dph \hat{\mathfrak{g}_{\ext}}(\lambda, \cdot)\|_{L^2(0,\theta)}\|\lambda \hat{\mathfrak{g}_{\ext}}(\lambda, \cdot)\|_{L^2(0,\theta)}.
\end{align}
For notational convenience we omit in the sequel the $\lambda$-dependence from the notation. By the Green's function representation (Corollary \ref{cor:repr}) we have
\begin{align*}
\partial_{\varphi}^{\ell} \hat{u}_{\varphi} = & \,  (\lambda + 1 ) \hat{\mathfrak{g}_{\ext}}( \theta) \partial_{\varphi}^{\ell} \partial_{\varphi'}G(\varphi, \theta) -  (\lambda + 1 ) \hat{\mathfrak{g}_{\ext}}( 0) \partial_{\varphi}^{\ell}  \partial_{\varphi'}G(\varphi, 0).\nonumber
\end{align*}
Using H\"older's inequality and that $ |\lambda +1|^2 \leq (|\lambda|+ \tfrac{1}{|\lambda|})|\lambda|$, we deduce 
\begin{align*}
\| \partial_{\varphi}^{\ell} \hat{u}_{\varphi} \|_{L^2(0,\theta)}^2 
\leq & \, 2\Big(\big(|\lambda|+\frac{1}{|\lambda|}\big)\| \partial_{\varphi}^{\ell} \partial_{\varphi'}G(\varphi, \theta)\|_{L^2(0,\theta)}^2 + \big(|\lambda|+\frac{1}{|\lambda|}\big)  \| \partial_{\varphi}^{\ell} \partial_{\varphi'}G(\varphi, 0)\|_{L^2(0,\theta)}^2 \Big)
\\ & \, \quad \quad \quad \cdot \Big( \frac{1}{|\lambda \theta|}\|\lambda \hat{\mathfrak{g}_{\ext}}\|_{L^2(0,\theta)}^2+ \|\partial_{\varphi} \hat{\mathfrak{g}_{\ext}}\|_{L^2(0,\theta)}\|\lambda \hat{\mathfrak{g}_{\ext}}\|_{L^2(0,\theta)} \Big) \nonumber .
\end{align*}
Let  $ \lambda = \alpha + i t $.  Note that if $ |t \theta | \geq \frac{1}{2}$, then $ \frac{1}{|\lambda \theta |} \leq 2 $.
It is then enough to show that for $\vartheta \in \{ 0,\theta\} $
\begin{align}
\big(|\lambda|+\frac{1}{|\lambda|}\big)\| \partial_{\varphi}^{\ell} \partial_{\varphi'} G(\lambda, \cdot, \vartheta) \|_{L^2(0,\theta)}^2 \leq \frac{C_{\ell}}{|\lambda|^{4-2\ell}}, \qquad \ell \text{ even } \quad \text{ and }  \quad |t \theta| \geq \frac{1}{2} \label{Green:func:est:2:1} \\
\big(|\lambda|+\frac{1}{|\lambda|}\big)\| \partial_{\varphi}^{\ell} \partial_{\varphi'} G(\lambda, \cdot, \vartheta) \|_{L^2(0,\theta)}^2 \leq \frac{C_{\ell}}{\min\{1,|\alpha|\}^{2\ell}}\frac{|\lambda \theta|}{|\lambda|^{4-2\ell}}, \qquad \ell \text{ even } \quad \text{ and  }  \quad |t \theta| \leq \frac{1}{2}. \label{Green:func:est:2:2}
\end{align}
To prove \eqref{Green:func:est:2:1} for $ \ell = 0$, we notice that by Lemma \ref{Lemma:green:fun}
\begin{equation}\label{eq:derG}
\begin{aligned}
\partial_{\varphi'}G(\lambda, \varphi, \theta) =  & \,  \frac{\sin((\lambda + 1)\varphi) }{4 \lambda \sin((\lambda+1)\theta)} - \frac{\sin((\lambda-1)\varphi) }{4 \lambda \sin((\lambda-1)\theta)} \\
= & \,\frac{\sin(\lambda(\theta-\varphi))\sin(\theta+\varphi)-\sin(\lambda(\theta+\varphi)\sin(\theta-\varphi) }{4 \lambda \sin((\lambda+1)\theta)\sin((\lambda-1)\theta)}
\end{aligned}
\end{equation}
We start by showing a lower bound of the denominator. For all $ \eps > 0 $ there are constants $ c_{\eps}, C_{\eps} \in (0,\infty) $ such that for all $ z \in \mathbb{C} $ with $ |\Re z| \leq (1- \tfrac{\eps}{2})\pi	$ 	and $ |\Im z | \geq 1 $ we have  
\begin{equation}
\label{bound:sin:nis}
c_{\eps}e^{2|\Im z|} \leq  | \sin(\Re(z))\cosh(\Im z) |^2 + | \cos(\Re(z))\sinh(\Im z) |^2 = |\sin(z)|^2 \leq C_{\eps}e^{2|\Im z|},
\end{equation}
and similarly, for $ |\Re z| \leq  \left(1-\tfrac{\eps}{2}\right) \pi $, it holds
\begin{equation}
\label{bound:cos:nis}
c_{\eps}e^{2|\Im z|} \leq  | \cos(\Re(z))\cosh(\Im z) |^2 + | \sin(\Re(z))\sinh(\Im z) |^2 = |\cos(z)|^2 \leq C_{\eps}e^{2|\Im z|}.
\end{equation}
For $  \alpha  \theta = \Re \lambda \theta \in I_{\eps} $ and $ |t  \theta| = |\Im \lambda \theta | \geq \tfrac{1}{2} $, this yields
\begin{equation*}
|\lambda \sin((\lambda+1)\theta)\sin((\lambda-1)\theta) | \geq c_{\eps} |\lambda| e^{2|t|\theta},
\end{equation*}
and further
\begin{align*}
\left| \partial_{\varphi'}G(\lambda, \varphi, \theta) \right| \leq   & \, \frac{|\sin(\lambda(\theta-\varphi))|\sin(\theta+\varphi)+|\sin(\lambda(\theta+\varphi))|\sin(\theta-\varphi) }{4 |\lambda \sin((\lambda+1)\theta)\sin((\lambda-1)\theta)|} \\
\leq & \, C_{\eps}\frac{e^{|t|(\theta -\varphi)}\sin(\theta+\varphi)+ e^{|t|(\theta +\varphi)}\sin(\theta-\varphi)  }{|\lambda | e^{2|t|\theta}} \\
\leq & \, \frac{C_{\eps}}{|\lambda|} \left( e^{-|t|(\theta +\varphi)}\sin(\theta+\varphi) + e^{-|t|(\theta -\varphi)}\sin(\theta-\varphi) \right)
\end{align*}
Using that also $ e ^{-|t|(\theta\pm\varphi)} \leq e^{-|\lambda|(\theta\pm\varphi)} e^{\alpha(\theta\pm\varphi)} \leq C e^{-|\lambda|(\theta\pm\varphi)}  $ because $ \alpha \in I_{\eps} \subset  \left( -\tfrac{\pi}{\theta},\tfrac{\pi}{\theta} \right)$, we have
\begin{align}
\label{good:idea}
\left\| \partial_{\varphi'}G(\lambda, \cdot, \theta) \right\|_{L^2(0,\theta)}^2 \leq \frac{C_{\eps}}{|\lambda|^2}\Big(\int_0^{\theta} e^{-2|\lambda|(\theta +\varphi)}\sin^2(\theta+\varphi)  \, \dd \varphi + \int_0^{\theta} e^{-2|\lambda|(\theta -\varphi)}\sin^2(\theta-\varphi)  \, \dd \varphi\Big).
\end{align}
After integration by parts twice, we obtain
\begin{align}
\int_0^{\theta} & \,   e^{-2|\lambda|(\theta \pm \varphi)}  \sin^2(\theta \pm \varphi) \, \dd \varphi \leq    \int_0^{\theta} e^{-2|\lambda|(\theta \pm \varphi)}(\theta \pm \varphi)^2\, \dd \varphi \nonumber  \\
=  & \,  \Big[ \mp \frac{e^{-2|\lambda|(\theta \pm \varphi)}}{2|\lambda|}(\theta \pm \varphi)^2\Big]_0^{\theta} + \Big[ \mp \frac{e^{-2|\lambda|(\theta \pm \varphi)}}{2|\lambda|^2}(\theta \pm \varphi)\Big]_0^{\theta}+ \Big[ \mp \frac{e^{-2|\lambda|(\theta \pm \varphi)}}{4|\lambda|^3}\Big]_0^{\theta} \leq C \frac{1}{|\lambda|^3}, \label{decay:est:un:theta}
\end{align}
where we have used that $|\lambda \theta|^k e^{-2 |\lambda|\theta} $ is bounded for $ |\lambda|\theta \geq \tfrac{1}{2}$ and $ k  \in \{0, 1, 2 \}$. Inequality \eqref{good:idea} together with \eqref{decay:est:un:theta} implies
 \begin{align}
 \label{1233}
\left\| \partial_{\varphi'}G(\lambda, \cdot, \theta) \right\|_{L^2(0,\theta)}^2 \leq  \frac{C_\eps}{|\lambda|^5}.
\end{align}
If we apply $ |\sin^2(\theta\pm \varphi)| \leq 1 $ in \eqref{good:idea}, then
\begin{align}
\label{12333}
\left\| \partial_{\varphi'}G(\lambda, \cdot, \theta) \right\|_{L^2(0,\theta)}^2 \leq \frac{C_{\eps}}{|\lambda|^2}\Big(\int_0^{\theta} e^{-2|\lambda|(\theta +\varphi)}  \, \dd \varphi + \int_0^{\theta} e^{-2|\lambda|(\theta -\varphi)}  \, \dd \varphi\Big) \leq \frac{C_\eps}{|\lambda|^3}.
\end{align}
Using \eqref{1233} and \eqref{12333}, we deduce that
\begin{equation*}
\big(|\lambda| + \frac{1}{|\lambda|}\big)\left\| \partial_{\varphi'}G(\lambda, \cdot, \theta) \right\|_{L^2(0,\theta)}^2 \leq C_{\eps} \frac{1}{|\lambda|^{4}}.
\end{equation*}
This proves \eqref{Green:func:est:2:1} with $ \ell = 0$. We continue with the proof of \eqref{Green:func:est:2:1} for all $\ell\geq 0$ even. To estimate higher derivatives, we notice that for $ \ell $ even
\begin{align*}
\partial_{\ph}^{\ell}\partial_{\varphi'} & G(\lambda, \varphi, \theta) \\
 = & \,  (-1)^{\tfrac{\ell}{2}}\Big( \lambda^{\ell} + \binom{\ell}{2} \lambda^{\ell-2} + \dots+ \binom{\ell}{\ell} 1   \Big)\frac{\sin(\lambda(\theta-\varphi))\sin(\theta+\varphi)-\sin(\lambda(\theta+\varphi)\sin(\theta-\varphi) }{4 \lambda \sin((\lambda+1)\theta)  \sin((\lambda-1)\theta)} \\
& \,  -(-1)^{\tfrac{\ell}{2}}\Big(\ell \lambda^{\ell-1} + \binom{\ell}{3} \lambda^{\ell-3} + \dots+ \ell \lambda   \Big)\frac{\cos(\lambda(\theta+\varphi))\cos(\theta-\varphi)-\cos(\lambda(\theta-\varphi))\cos(\theta+\varphi)}{4 \lambda \sin((\lambda+1)\theta) \sin((\lambda-1)\theta)}.
\end{align*}
Using \eqref{bound:sin:nis} and the fact that $ \tfrac{1}{2} \leq |t\theta | \leq |\lambda \theta| $, we have
\begin{align}
|\partial_{\ph}^{\ell}\partial_{\varphi'} G(\lambda, \varphi, \theta)| \leq & \,  C_{\eps}|\lambda|^{\ell-1} \left( e^{-|t|(\theta +\varphi)}\sin(\theta+\varphi) + e^{-|t|(\theta -\varphi)}\sin(\theta-\varphi) \right) \label{Felipe:1}\\
& \, + C_{\eps}|\lambda|^{\ell-2} \left( e^{-|t|(\theta +\varphi)}\cos(\theta+\varphi) + e^{-|t|(\theta -\varphi)}\cos(\theta-\varphi) \right) \nonumber
\end{align}
Using \eqref{decay:est:un:theta} and
\begin{align}
\label{Felipe:2}
\int_0^{\theta} e^{-2|\lambda|(\theta \pm \varphi)}& \, \cos^2(\theta \pm \varphi) \, \dd \varphi \leq    \int_0^{\theta} e^{-2|\lambda|(\theta \pm \varphi)}\, \dd \varphi \leq C \frac{1}{|\lambda|},
\end{align}
we arrive at
\begin{align}
\label{Felipe:3}
\big(|\lambda| + \frac{1}{|\lambda|}\big)\big\| \partial_{\varphi}^{\ell} \partial_{\varphi'}G(\lambda, \cdot, \theta) \big\|_{L^2(0,\theta)}^2 \leq \left(|\lambda| + |\lambda|\theta^2 \right) \Big(|\lambda|^{2\ell - 2}\frac{C}{|\lambda|^3} + |\lambda|^{2\ell - 4}\frac{C}{|\lambda|}  \Big)
\leq  C_{\eps}|\lambda|^{2\ell-4}.
\end{align}
This finishes the proof of \eqref{Green:func:est:2:1}.

For the proof of \eqref{Green:func:est:2:2} note that
\begin{equation}
\label{bound:sin}
c_{\eps} |z| \leq | \sin(z)| \leq  C_{\eps}|z|, \quad \text{ when }   |\Re z| \leq  \left(1-\frac{\eps}{2}\right) \pi \quad \text{ and } \quad |\Im z | \leq \frac{1}{2}.
\end{equation}
We use the above inequality to obtain the lower bound
\begin{equation*}
| \lambda \sin((\lambda+1)\theta)\sin((\lambda-1)\theta) | \geq c_{\eps}^2| \lambda (\lambda^2-1)\theta^2|
\end{equation*}
for any  $ \lambda $ such that $ \Re \lambda \in I_{\eps} $ and $ |\theta \Im \lambda |  \leq \tfrac{1}{2} $.

First, assume that $ |\lambda -1 | \leq \frac{1}{4} $ or $ |\lambda +1 | \leq \frac{1}{4} $. Using that $|t \varphi|\leq | t \theta| \leq \frac{1}{2}$,  $ \Re \lambda \in I_{\eps} $ and \eqref{bound:sin}, we have
\begin{align*}
\left|\partial_{\varphi'}G(\lambda, \varphi, \theta)\right| =  & \,  \left|\frac{\sin((\lambda + 1)\varphi) }{4 \lambda \sin((\lambda+1)\theta)} - \frac{\sin((\lambda-1)\varphi) }{4 \lambda \sin((\lambda-1)\theta)}  \right|\\
= & \,  \left|\frac{\sin((\lambda + 1)\varphi) \sin((\lambda-1)\theta)-\sin((\lambda - 1)\varphi) \sin((\lambda+1)\theta)}{4 \lambda \sin((\lambda+1)\theta)\sin((\lambda-1)\theta)}\right| \\
\leq & \, C \frac{|\lambda +1||\lambda -1 |\varphi \theta}{|\lambda ||\lambda +1||\lambda -1 |} \leq  C\frac{\varphi \theta }{|\lambda|} \leq C\frac{\pi^2 }{|\lambda|} \leq C\frac{1}{|\lambda|^2},
\end{align*}
where in the last step we have used that $ \frac{1}{|\lambda|} \geq C $ for $ |\lambda -1 | \leq \frac{1}{4} $ or $ |\lambda + 1 | \leq \frac{1}{4} $. Since $ |\lambda| + \tfrac{1}{|\lambda|} \leq 4 |\lambda|$, we deduce
\begin{align*}
\left(|\lambda| + \tfrac{1}{|\lambda|} \right)^{\tfrac{1}{2}}\|\partial_{\varphi'}G(\lambda, \cdot, \theta) \|_{L^2(0,\theta)} \leq C \Big(\int_0^{\theta} \frac{1}{|\lambda|^{3}} \, \dd \varphi \Big)^{\frac{1}{2}}\leq C \frac{|\lambda \theta|^{\frac{1}{2}}}{ |\lambda|^{4}}.
\end{align*}

Now assume that $ \Re \lambda  \in I_{\eps} $, $ |t \theta|\leq \frac{1}{2} $, $ |\lambda -1 | \geq \frac{1}{4} $ and $ |\lambda +1 | \geq \frac{1}{4} $. It follows that
\begin{equation*}
\frac{1}{|\lambda -1 |} \leq C \frac{1}{|\lambda  |} ,  \quad \frac{1}{|\lambda + 1 |} \leq C \frac{1}{|\lambda |} \quad \text{ and } \quad \frac{1}{|\lambda-1|} \sim \frac{1}{|\lambda+1|}.
\end{equation*}
Moreover, from Appendix \ref{app:Aux_est} we have
\begin{equation*}
\left|\frac{\sin((\lambda+1)\varphi)}{4 \lambda \sin((\lambda+1)\theta)}-\frac{\sin((\lambda-1)\varphi)}{4 \lambda \sin((\lambda-1)\theta)} \right| \leq C_{\eps}\frac{1}{|\lambda|^2},
\end{equation*}
The previous two inequalities imply that
\begin{align*}
|\lambda|^{\tfrac{1}{2}}\|\partial_{\varphi'}G(\lambda, \cdot, \theta) \|_{L^2(0,\theta)} \leq C \Big(\int_0^{\theta} \frac{1}{|\lambda|^{3}} \, \dd \varphi \Big)^{\frac{1}{2}}\leq C \frac{|\lambda \theta|^{\frac{1}{2}}}{ |\lambda|^{2}}.
\end{align*}
To estimate $ |\lambda|^{-\frac{1}{2}}\|\partial_{\varphi'}G(\lambda, \cdot, \theta) \|_{L^2(0,\theta)} $, use \eqref{key:estimate:luce:pulsata} to obtain
\begin{align*}
\left|\partial_{\varphi'}G(\lambda, \varphi, \theta)\right| =  & \,  \left|\frac{\sin((\lambda + 1)\varphi) }{4 \lambda \sin((\lambda+1)\theta)} - \frac{\sin((\lambda-1)\varphi) }{4 \lambda \sin((\lambda-1)\theta)}  \right| \leq \frac{C_{\eps}}{|\lambda|}.
\end{align*}
By combining all the estimates above, we find
\begin{align*}
|\lambda|^{-\frac{1}{2}}\|\partial_{\varphi'}G(\lambda, \cdot, \theta) \|_{L^2(0,\theta)} \leq C_{\eps} \Big(\int_0^{\theta} \frac{1}{|\lambda|^{3}} \, \dd \varphi \Big)^{\frac{1}{2}}\leq C \frac{|\lambda \theta|^{\frac{1}{2}}}{ |\lambda|^{2}},
\end{align*}
which proves \eqref{Green:func:est:2:2} for $\ell=0$. We continue with the proof of \eqref{Green:func:est:2:2} for all $\ell\geq 0$ even.
To estimate higher derivatives, we notice that for $ \ell $ even
\begin{align*}
\partial_{\ph}^{\ell}\partial_{\varphi'} & G(\lambda, \varphi, \theta) \\
 = & \,  (-1)^{\tfrac{\ell}{2}}\Big( \lambda^{\ell} + \binom{\ell}{2} \lambda^{\ell-2} + \dots+ \binom{\ell}{\ell} 1   \Big)\Big(\frac{\sin((\lambda+1)\varphi)}{4 \lambda \sin((\lambda+1)\theta)}-\frac{\sin((\lambda-1)\varphi)}{4 \lambda \sin((\lambda-1)\theta)}\Big) \\
& \, + (-1)^{\tfrac{\ell}{2}}\Big(\ell \lambda^{\ell-1} + \binom{\ell}{3} \lambda^{\ell-3} + \dots+ \ell \lambda   \Big)\Big(\frac{\sin((\lambda+1)\varphi)}{4 \lambda \sin((\lambda+1)\theta)}+\frac{\sin((\lambda-1)\varphi)}{4 \lambda \sin((\lambda-1)\theta)}\Big).
\end{align*}
Using \eqref{key:estimate:luce:pulsata:-1} we deduce
\begin{align*}
| \partial_{\ph}^{\ell}\partial_{\ph'} G(\lambda, \ph, \theta)| \leq C_{\eps}\frac{|\lambda|^{\ell-2}}{\min\{ 1, |\alpha|\}^{\ell}} + C_{\eps}\frac{|\lambda|^{\ell-2}}{\min\{ 1, |\alpha|\}^{\ell-2}}.
\end{align*}
This implies 
\begin{align*}
|\lambda|^{\tfrac{1}{2}}\|\partial_{\varphi}^{\ell}\partial_{\varphi'}G(\lambda, \cdot, \theta) \|_{L^2(0,\theta)} \leq C_{\eps} \frac{1}{\min\{ 1, |\alpha|\}^{\ell}} \Big(\int_0^{\theta} \frac{1}{|\lambda|^{3-2\ell}} \, \dd \varphi \Big)^{\frac{1}{2}}\leq C_{\eps} \frac{1}{\min\{ 1, |\alpha|\}^{\ell}}\frac{|\lambda \theta|^{\frac{1}{2}}}{ |\lambda|^{2-\ell}},
\end{align*}
and similarly
\begin{align*}
|\lambda|^{-\frac{1}{2}}\|\partial_{\varphi}^{\ell}\partial_{\varphi'}G(\lambda, \cdot, \theta) \|_{L^2(0,\theta)} \leq C_{\eps} \frac{1}{\min\{ 1, |\alpha|\}^{\ell}} \Big(\int_0^{\theta} \frac{1}{|\lambda|^{3-2\ell}} \, \dd \varphi \Big)^{\frac{1}{2}}\leq C_{\eps}\frac{1}{\min\{ 1, |\alpha|\}^{\ell}} \frac{|\lambda \theta|^{\frac{1}{2}}}{ |\lambda|^{2-\ell}}.
\end{align*}

For $  \ell $ odd, the estimate follows from interpolation.

We now show the bounds \eqref{bound:ur:ur} for $ \hat{u}_r$. 
% TO GET THE RIGHT ESTIMATES WE ACTUALLY NEED TO MIMIC THE SAME AS FOR THE UPH TERM
If $ \Re \lambda \in I_{\eps} \setminus \left(-\frac{3}{2}, - \frac{1}{2} \right)$, then \eqref{bound:ur:ur} 
is a direct consequence of \eqref{bound:uphi:uphi} and $ \hat{u}_r = \frac{\partial_{\ph} u_{\ph}}{\lambda+1}$. Noting that we have $ \frac{1 }{|1+\lambda|} \leq C \frac{1}{|\lambda|} $ yields the estimate
\begin{align*}
\| \partial_{\ph}^{\ell} \hat{u}_{r}(\lambda,\cdot)\|_{L^2(0,\theta)} \leq & \, \frac{1}{|1+\lambda|}    \| \partial_{\ph}^{\ell+1} \hat{u}_{\ph}(\lambda,\cdot)\|_{L^2(0,\theta)} \\ \leq & \,   C_{\eps}\frac{1}{\min\{1,|\alpha|\}^{\ell+1}} \frac{1}{ |\lambda|^{2-\ell}}\left(\|\lambda \hat{\mathfrak{g}_{\ext}}\|_{L^2(0,\theta)}+ \|\partial_{\varphi} \hat{\mathfrak{g}_{\ext}}\|_{L^2(0,\theta)}\right).
\end{align*}
For $ \Re \lambda \in \left(-\frac{3}{2}, - \frac{1}{2} \right) $ we mimic the proof of \eqref{bound:uphi:uphi}. Recall that (see \eqref{u:r:espression})
\begin{equation*}
\hat{u}_{r}(\lambda,\ph) =  \hat{\mathfrak{g}_{\ext}}(\lambda, \theta) \partial_{\varphi} \partial_{\varphi'} G(\varphi, \theta) - \hat{\mathfrak{g}_{\ext}}(\lambda, 0) \partial_{\varphi} \partial_{\varphi'} G(\varphi, 0).
\end{equation*}
With H\"older's inequality and \eqref{Mimi:trace:estimate} we have
\begin{align}
\| \partial_{\ph}^{\ell} \hat{u}_r (\lambda,\cdot)\|_{L^2(0,\theta)}^2 \leq & \, 2 \Big(\frac{1}{|\lambda|}\|\partial_{\ph}^{\ell+1}\partial_{\ph'} G(\cdot, 0)\|_{L^2(0,\theta)}^2+ \frac{1}{|\lambda|}\|\partial_{\ph}^{\ell+1}\partial_{\ph'} G(\cdot, \theta)\|_{L^2(0,\theta)}^2 \Big) \label{estimate:Fel:2:4444} \\
 & \, \quad \quad \quad \cdot \Big( \frac{1}{|\lambda \theta|}\|\lambda \hat{\mathfrak{g}_{\ext}}\|_{L^2(0,\theta)}^2+ \|\partial_{\varphi} \hat{\mathfrak{g}_{\ext}}\|_{L^2(0,\theta)}\|\lambda \hat{\mathfrak{g}_{\ext}}\|_{L^2(0,\theta)} \Big) \nonumber .
\end{align}
For $ \ell $ odd we can use the bounds for \eqref{Green:func:est:2:1}, \eqref{Green:func:est:2:2} and show \eqref{bound:ur:ur}  for any $ \ell $ odd. An interpolation argument implies the result for any $ \ell \geq 1 $. We are left with the case $ \ell = 0 $. To address this case, we start by assuming $ |t \theta| \geq \frac{1}{2} $.
In this case we have (see \eqref{eq:derG})
\begin{align*}
\partial_{\ph}\partial_{\varphi'}  G(\lambda, \varphi, \theta) 
 =&  \,  -\lambda\frac{\cos(\lambda(\theta-\varphi))\sin(\theta+\varphi)+\cos(\lambda(\theta+\varphi)\sin(\theta-\varphi) }{4 \lambda \sin((\lambda+1)\theta)  \sin((\lambda-1)\theta)} \\
& \, + \frac{\sin(\lambda(\theta+\varphi))\cos(\theta-\varphi)+\sin(\lambda(\theta-\varphi))\cos(\theta+\varphi)}{4 \lambda \sin((\lambda+1)\theta) \sin((\lambda-1)\theta)}.
\end{align*}
Using \eqref{bound:cos:nis} and the fact that $ \tfrac{1}{2} \leq |t\theta | \leq |\lambda \theta| $, we have
\eqref{Felipe:1}. This together with \eqref{decay:est:un:theta} and \eqref{Felipe:2} implies that \eqref{Felipe:3} holds also for $ \ell = 0 $. The estimate reads
\begin{align*}
\frac{1}{|\lambda|}\left\| \partial_{\varphi} \partial_{\varphi'}G(\lambda, \cdot, \theta) \right\|_{L^2(0,\theta)}^2
\leq  C\frac{1}{|\lambda|^{4}},
\end{align*}
and together with \eqref{estimate:Fel:2:4444}, this implies the corollary if $ \ell = 0 $ and $ |t \theta| \geq \tfrac{1}{2}$.
Let us now move to the case $|t\theta| \leq \frac{1}{2}$. As in the proof of \eqref{bound:uphi:uphi} the case $ |\lambda-1|\leq \frac{1}{4} $ is easy so let us assume that  $|t\theta| \leq \frac{1}{2}$ and $ |\lambda-1|\geq \frac{1}{4} $. We rewrite
\begin{align*}
\hat{u}_{r}(\lambda, \ph) = & \, \Big(\frac{(\lambda+1)(\cos((\lambda+1)\ph))}{4 \lambda \sin((\lambda+1)\theta)} - \frac{(\lambda-1)(\cos((\lambda-1)\ph))}{4 \lambda \sin((\lambda-1) \theta)} \Big) \hat{\mathfrak{g}_{\ext}}(\lambda, \theta) \\
& \, -\Big(\frac{(\lambda+1)(\cos((\lambda+1)(\theta-\ph)))}{4 \lambda \sin((\lambda+1)\theta)} - \frac{(\lambda-1)(\cos((\lambda-1)(\theta-\ph)))}{4 \lambda \sin((\lambda-1)\theta)} \Big) \hat{\mathfrak{g}_{\ext}}(\lambda, 0).
\end{align*}
From \eqref{key:estimate:luce:pulsata:cos}, i.e.,
\begin{equation*}
\left| \frac{(\lambda+1)(\cos((\lambda+1)\ph))}{4 \lambda \sin((\lambda+1)\theta)} - \frac{(\lambda-1)(\cos((\lambda-1)\ph))}{4 \lambda \sin((\lambda-1) \theta)} \right| \leq C_{\eps} \theta,
\end{equation*}
we deduce
\begin{align*}
\|\hat{u}_r\|_{L^2(0,\theta)}^2 \leq & \, \int_0^\theta C_\eps^2\frac{\theta^2}{|\lambda|}(|\lambda||\hat{\mathfrak{g}_{\ext}}(\lambda, \theta)|^2 + |\lambda||\hat{\mathfrak{g}_{\ext}}(\lambda, 0)|^2) \, \dd \ph\\ 
\leq &\, C_\eps^2\frac{\theta^3}{|\lambda|} (|\lambda||\hat{\mathfrak{g}_{\ext}}(\lambda, \theta)|^2 + |\lambda||\hat{\mathfrak{g}_{\ext}}(\lambda, 0)|^2) \\
\stackrel{\mathclap{\eqref{Mimi:trace:estimate}}}{\leq} & \, C_\eps^2 \frac{\lambda^3 \theta^3}{|\lambda|^4} \Big( \frac{1}{|\lambda \theta|}\|\lambda \hat{\mathfrak{g}_{\ext}}\|_{L^2(0,\theta)}^2+ \|\partial_{\varphi} \hat{\mathfrak{g}_{\ext}}\|_{L^2(0,\theta)}\|\lambda \hat{\mathfrak{g}_{\ext}}\|_{L^2(0,\theta)} \Big).\qedhere
\end{align*}
\end{proof}

\subsection{Regularity of \eqref{SS:Mellin} and \eqref{fourth:order:equ} with $\mathfrak{g}=0$}\label{sec:93} As for the Helmholtz projection, we introduce a Fourier-Mellin representation formula for solutions of \eqref{fourth:order:equ} in the case $ \hat{\mathfrak{g}} = 0$, i.e., for the system
\begin{subequations}\label{fourth:order:equ:g=0}
\begin{align}
\partial_{\varphi}^4 \hat{u}_{\varphi} + 2(\lambda^2+1)\partial_{\varphi}^2 \hat{u}_{\varphi} +(\lambda^2-1)^2 \hat{u}_{\varphi}  &=   (\lambda^2-1)\hat{f}_{\varphi}(\lambda-2)- (\lambda+1) \partial_\varphi \hat{f}_r(\lambda-2)  && \text{ in }  \mathcal{S}, \label{fourth:order:equ:g=0a} \\
\hat{u}_{\varphi} =  \partial_{\varphi}^2 \hat{u}_{\varphi}   & = 0 && \text{ on } \{0, \theta\}.\label{fourth:order:equ:g=0b}
\end{align}
\end{subequations}
To simplify notation, let us denote the source term by
\begin{equation}
\label{ftoffrak}
\hat{\mathfrak{f}}(\lambda, \ph) = (\lambda^2-1)\hat{f}_{\varphi}(\lambda-2, \ph)- (\lambda+1) \partial_\varphi \hat{f}_r(\lambda-2, \ph).
\end{equation}

Recall the orthonormal systems $ \{ \ecos_k \}_{k \in \mathbb{N}}$ and $ \{ \esin_k \}_{k \in \mathbb{N}}$  from \eqref{cos:sin:base}.
The Fourier-Mellin representation formula for solutions of \eqref{fourth:order:equ:g=0} reads as follows.
\begin{lem} \label{lem:Fourier_Rep:STOKES}
Let $ \vf \in C^{\infty}_{\mathrm{c}}(\overline{\Omega}\setminus \{0\})$. The solution $ \hat{u}_{\varphi} $ of problem \eqref{fourth:order:equ:g=0} can be written as
\begin{equation}\label{eq:Rep:Stokes}
    \hat{u}_{\varphi}(\lambda,\ph)= \sum_{k=1}^{\infty}\frac{\hat{\mathfrak{f}}_k(\lambda)}{\big(\lambda^2-\left(\frac{k\pi}{\theta}\right)^2-1\big)^2}\esin_k(\ph)\qquad  \text{ for } \ph\in[0,\theta]\text{ and }\Re\lambda\in\left(-\tfrac{\pi}{\theta}+1,\tfrac{\pi}{\theta}-1\right),
\end{equation}
where $ \mathfrak{f} $ is associated with $ \vf $ by \eqref{ftoffrak}. Moreover, the function $ \hat{u}_r = - \frac{\partial_{\ph}\hat{u}_{\ph}}{\lambda+1} $ can be written  as
\begin{equation*} 
    \hat{u}_{r}(\lambda,\ph)= \sum_{k=1}^{\infty}\frac{\hat{\mathfrak{h}}_k(\lambda)}{\big(\lambda^2-\left(\frac{k\pi}{\theta}\right)^2-1\big)^2}\ecos_k(\ph)\qquad \text{ for }\ph\in[0,\theta] \text{ and }\Re\lambda\in\left(-\tfrac{\pi}{\theta}+1,\tfrac{\pi}{\theta}-1\right),
\end{equation*}
where
$ \hat{\mathfrak{h}}(\lambda, \ph) = - (\lambda-1)\partial_{\ph} \hat{f}_{\varphi}(\lambda-2, \ph)+ \partial^2_\varphi \hat{f}_r(\lambda-2, \ph)$.
\end{lem}

We deduce the following estimates.
\begin{lem}
\label{lem:bound:uph:stokes:CL}
Let $ \vf \in C^{\infty}_{\mathrm{c}}(\overline{\Omega}\setminus \{0\})$,  $ \theta \in (0, \pi) $, $\eps \in (0, 1- \tfrac{\theta}{\pi} ) $ and $ \Re \lambda = \alpha \in I_{\eps} \setminus \mathbb{Z}  $. Then the function $ \hat{u}_{\varphi} $ defined in \eqref{eq:Rep:Stokes} satisfies for $ M \geq 0 $ the bound
\begin{equation}
\label{MimiM:1}
\sum_{j+\ell = M+2 } |\lambda|^{2j}\|  \partial_{\ph}^{\ell}\hat{u}_{\varphi}(\lambda,\cdot) \|_{L^2(0,\theta)}^2 \leq  C_{\eps,M}^2  \frac{ 1 }{\min\{1, |\alpha| \}^4}\sum_{j+\ell = M } \frac{|\lambda-2|^{2j}}{\min\{1, |\alpha-2| \}^{2j}}\| \partial_{\ph}^{\ell} \hat{\bf f}(\lambda-2,\cdot) \|_{L^2(0,\theta)}^2,
\end{equation}
where $ C $ is independent of $ \theta $.
Moreover, the function $ \hat{u}_r = - \frac{\partial_{\ph}\hat{u}_{\ph}}{\lambda+1} $ satisfies for $ M \geq 0 $
\begin{equation*}
\sum_{j+\ell = M+2 } |\lambda|^{2j}\|  \partial_{\ph}^{\ell}\hat{u}_{r}(\lambda,\cdot) \|_{L^2(0,\theta)}^2 \leq  C_{\eps,M}^2  \frac{ 1 }{\min\{1, |\alpha| \}^4}\sum_{j+\ell = M } \frac{|\lambda-2|^{2j}}{\min\{1, |\alpha-2| \}^{2j}}\| \partial_{\ph}^{\ell} \hat{\bf f}(\lambda-2,\cdot) \|_{L^2(0,\theta)}^2,
\end{equation*}
for any $ \Re \lambda \in I_{\eps} $ and $ C $ independent of $ \theta $. 
\end{lem}

\begin{proof}
For $ \ell \leq  2 $, Bessel's identity \eqref{eq:Bessel_Identity} implies
\begin{align*}
\| \partial_{\ph}^{\ell} \hat{u}_{\varphi}(\lambda,\cdot) \|_{L^2(0,\theta)}^2 = & \, \sum_{k = 1}^{\infty}\frac{\big(\frac{k\pi}{\theta}\big)^{2\ell} |\hat{\mathfrak{f}}_k(\lambda)|^2}{\big|\lambda^2-\left(\frac{k\pi}{\theta}\right)^2-1\big|^4}  .
\end{align*}
First, we derive a lower-bound for the denominator. Rewrite, using $ \lambda = \alpha + i t $, with $ \alpha$, $ t \in \mathbb{R} $,
\begin{align*}
\Big|\lambda^2-\left(\tfrac{k\pi}{\theta}\right)^2-1\Big|^2 = & \,  \left|\alpha^2-t^2 - \left(\tfrac{k\pi}{\theta}\right)^2-1 + 2 i\alpha t \right|^2 \\
= & \, \left(t^2 + \left(\tfrac{k\pi}{\theta}\right)^2+1- \alpha^2\right)^2 + 4 \alpha^2 t^2.
\end{align*}
For $ k \geq 1 $, we  have
\begin{align*}
\left(\tfrac{k\pi}{\theta}\right)^2+1- \alpha^2 \geq (\eps-\eps^2)\left(\tfrac{\pi}{\theta}\right)^2 \geq (\eps-\eps^2) \alpha^2 = \frac{\alpha^2}{\sqrt{C_{\eps}}},
\end{align*}
which implies
\begin{equation}
\label{632}
\big(t^2 + \left(\tfrac{k\pi}{\theta}\right)^2+1- \alpha^2\big)^2 + 4 \alpha^2 t^2 \geq \Big(t^2 + \frac{\alpha^2}{\sqrt{C_{\eps}}} \Big)^2 + 4 \alpha^2 t^2 \geq \frac{1}{C_{\eps}}(t^2 + \alpha^2 )^2 = \frac{|\lambda|^4}{C_{\eps}}
\end{equation}
and
\begin{equation}
\label{63211}
\big(t^2 + \left(\tfrac{k\pi}{\theta}\right)^2+1- \alpha^2\big)^2 + 4 \alpha^2 t^2 \geq \Big(t^2 + \frac{1}{\sqrt{C_{\eps}}}\left(\tfrac{k\pi}{\theta}\right)^2 \Big)^2 + 4 \alpha^2 t^2 \geq \frac{1}{C_{\eps}}\left(\tfrac{k\pi}{\theta}\right)^4.
\end{equation}
Estimates \eqref{632} and \eqref{63211} imply
\begin{equation*}
\big|\lambda^2-\left(\tfrac{k\pi}{\theta}\right)^2-1\big| \geq \frac{|\lambda|^2}{\sqrt{C_{\eps}}} \quad \text{ and } \quad \big|\lambda^2-\left(\tfrac{k\pi}{\theta}\right)^2-1\big| \geq \frac{1}{\sqrt{C_{\eps}}}\left(\tfrac{k\pi}{\theta}\right)^2.
\end{equation*}
For $ \ell \leq 2 $
\begin{align}
\| \partial_{\ph}^{\ell} \hat{u}_{\varphi}(\lambda,\cdot) \|_{L^2(0,\theta)}^2 = & \, \sum_{k = 1}^{\infty}\frac{\big(\frac{k\pi}{\theta}\big)^{2\ell} |\hat{\mathfrak{f}}_k(\lambda)|^2}{\big|\lambda^2-\left(\frac{k\pi}{\theta}\right)^2-1\big|^4} \nonumber  \\
\leq & \, \frac{2}{\min\{1,|\alpha|\}^2}\sum_{k = 1}^{\infty}\frac{\big(\frac{k \pi}{\theta}\big)^{2\ell} |\lambda|^4|\widehat{ f_{\ph,k}}|^2+ \big(\frac{k \pi}{\theta}\big)^{2\ell+2} |\lambda|^2|\widehat{ f_{r,k}}|^2}{\big|\lambda^2-\left(\frac{k\pi}{\theta}\right)^2-1\big|^{4}} \nonumber \\
\leq & \, C_{\eps} \frac{2}{\min\{1,|\alpha|\}^2}\frac{1}{|\lambda|^{4-2\ell}}\| {\bf f}(\lambda-2,\cdot) \|_{L^2(0,\theta)}^2. \label{MimiM:2}
\end{align}
This proves \eqref{MimiM:1} for $ M = 0 $. For $ M = 1 $, \eqref{MimiM:2} implies
\begin{equation*}
\sum_{j+\ell = 3, \ell \leq 2 } |\lambda|^{2j}\|  \partial_{\ph}^{\ell}\hat{u}_{\varphi}(\lambda,\cdot) \|_{L^2(0,\theta)}^2 \leq  C_{\eps}^2  \frac{ 1 }{\min\{1, |\alpha| \}^4} \frac{ 1 }{\min\{1, |\alpha-2| \}^2} |\lambda-2|^2 \|  \hat{\bf f}(\lambda-2,\cdot) \|_{L^2(0,\theta)}^2.
\end{equation*}
It remains to estimate
\begin{align*}
\| \partial_{\ph}^3 \hat{u}_{\varphi}\|_{L^2(0,\theta)}^2\quad \stackrel{\mathclap{\eqref{fourth:order:equ:g=0b}}}{=} &\quad \, - \int_0^{\theta} \overline{\partial_{\ph}^2 \hat{u}_{\varphi}} \partial_{\ph}^4 \hat{u}_{\varphi} \, \dd \varphi \\
\stackrel{\mathclap{\eqref{fourth:order:equ:g=0}}}{=} & \quad \, 2(\lambda^2+1) \| \partial_{\ph}^2 \hat{u}_{\varphi}\|_{L^2(0,\theta)}^2 -(\lambda^2-1)^2 \| \partial_{\ph} \hat{u}_{\varphi}\|_{L^2(0,\theta)}^2 \\
& \, - \int_0^{\theta}\overline{\partial_{\ph}^2 \hat{u}_{\varphi}}(\lambda^2-1)\hat{f}_{\varphi}(\lambda-2) \dd \varphi + \int_0^{\theta}\overline{\partial_{\ph}^3 \hat{u}_{\varphi} }(\lambda+1)  \hat{f}_r(\lambda-2) \dd \varphi \\
\leq & \, C_{\eps}^2  \frac{ 1 }{\min\{1, |\alpha| \}^4} \frac{ 1 }{\min\{1, |\alpha-2| \}^2} |\lambda-2|^2 \|  \hat{\bf f}(\lambda-2,\cdot) \|_{L^2(0,\theta)}^2  + \tfrac{1}{2} \| \partial_{\ph}^3 \hat{u}_{\varphi}\|_{L^2(0,\theta)}^2.
\end{align*}
After absorbing the last term on the left-hand side, we deduce \eqref{MimiM:1} for $ M = 1 $.
For $ M \geq 2 $, we argue by induction. Suppose that \eqref{MimiM:1} holds for $ M = n \geq 1  $, then we show that \eqref{MimiM:1} holds for $ M = n + 1  $. First of all, notice that \eqref{MimiM:1} with $ M = n  $ implies
\begin{align*}
\sum_{j+\ell = n+3, \ell \leq n +2} |\lambda|^{2j}\| & \partial_{\ph}^{\ell}\hat{u}_{\varphi}(\lambda,\cdot) \|_{L^2(0,\theta)}^2\\
& \leq  C_{\eps}^2  \frac{ 1 }{\min\{1, |\alpha| \}^4} \sum_{j+\ell = n + 1, \ell \leq n }\frac{ |\lambda-2|^{2j} }{\min\{1, |\alpha-2| \}^{2j}}  \| \partial_{\varphi}^{\ell} \hat{\bf f}(\lambda-2,\cdot) \|_{L^2(0,\theta)}^2.
\end{align*}
It remains to estimate $ \|  \partial_{\ph}^{n+3}\hat{u}_{\varphi}(\lambda,\cdot) \|_{L^2(0,\theta)}^2  $. Using equation  \eqref{fourth:order:equ:g=0a}, we obtain
\begin{align*}
\partial_{\varphi}^{n+3} \hat{u}_{\varphi} = & \,  - 2(\lambda^2+1)\partial_{\varphi}^{n+1} \hat{u}_{\varphi} -(\lambda^2-1)^2 \partial_{\varphi}^{n-1}\hat{u}_{\varphi}  \\ & \, +   (\lambda^2-1)\partial_{\varphi}^{n-1}\hat{f}_{\varphi}(\lambda-2)- (\lambda+1) \partial_{\varphi}^{n} \hat{f}_r(\lambda-2)
\end{align*}
and thus
\begin{equation*}
 \|  \partial_{\ph}^{n+3}\hat{u}_{\varphi}(\lambda,\cdot) \|_{L^2(0,\theta)}^2  \leq  C_{\eps}^2  \frac{ 1 }{\min\{1, |\alpha| \}^4} \sum_{j+\ell = n + 1, \ell \leq n }\frac{ |\lambda-2|^{2j} }{\min\{1, |\alpha-2| \}^{2j}}  \| \partial_{\varphi}^{\ell} \hat{\bf f}(\lambda-2,\cdot) \|_{L^2(0,\theta)}^2.
\end{equation*}

To show the bound of $\|\partial^{\ell}_{\varphi} \hat{u}_r\|_{L^2(0,\theta)} $ it is enough to mimic the proof of the estimates for $\|\partial^{\ell}_{\varphi} \hat{u}_{\varphi}\|_{L^2(0,\theta)} $.
\end{proof}

Using the above lemma we prove an estimate for the solution of \eqref{SS:Mellin}.
\begin{cor}
\label{cor:est:source:term} Let $ \theta \in (0, \pi) $, $\eps \in (0, 1- \tfrac{\theta}{\pi} ) $, $\alpha\in\RR\setminus\ZZ$  and $M\in\NN$ be such that $ M+ \alpha+1 \in I_{\eps}$.
Let $ \vu^s$ the solution of system \eqref{SS:Mellin} with $ \mathfrak{g} = 0 $. Then we have the estimate
\begin{equation*}
\llbracket \vu^s \rrbracket_{M+2,\alpha} \leq \frac{C_{\eps,M}}{\min\{1, |\alpha+M+1|\}^{2}}\frac{1}{\min\{1, |\alpha+M-1|\}^{M}} \llbracket \vf \rrbracket_{M,\alpha}.
\end{equation*}
\end{cor}

\begin{proof}
Using Lemma \ref{lem:bound:uph:stokes:CL}, we deduce for $ M \geq 0 $ that 
\begin{align*}
\llbracket \vu^s \rrbracket_{M+2,\alpha}^2 = & \, \sum_{j+\ell= M+2} \int_0^{\theta} \int_{\Re \lambda = M + \alpha+1} |\lambda|^{2j}|\partial_{\ph}^{\ell} \hat{\vu^s}|^2 \, \dd \Im \lambda \dd\theta \\
=  & \,   \int_{\Re \lambda = M + \alpha+1} \sum_{j+\ell= M+2} |\lambda|^{2j}\|\partial_{\ph}^{\ell} \hat{\vu^s}\|^2_{L^2(0,\theta)} \, \dd \Im \lambda \\
\leq & \,  C_{\eps,M}^2  \frac{ 1 }{\min\{1, |\alpha+M-1| \}^4}\sum_{j+\ell = M} \int_{\Re \lambda = M + \alpha+1} \frac{|\lambda-2|^{2j}}{\min\{1, |\alpha-2| \}^{2j}}\| \partial_{\ph}^{\ell} \hat{\bf f}(\lambda-2,\cdot) \|_{L^2(0,\theta)}^2 \\
\leq & \, \frac{C_{\eps,M}^2}{\min\{1, |\alpha+M+1|\}^{4}}\frac{1}{\min\{1, |\alpha+M-1|\}^{2M}} \llbracket \vf \rrbracket_{M,\alpha}^2.\qedhere
\end{align*}
\end{proof}

\subsection{The proof of Proposition \ref{my:step} }
\label{sec:2}
We prove Proposition \ref{my:step} from Section \ref{sec:reg_problem} concerning higher regularity.

\begin{proof}[Proof of Proposition  \ref{my:step}]
Define $ \hat{u}_{\varphi} $ as in Corollary \ref{cor:repr} and $ \hat{u}_{r} $ as in \eqref{u:r:espression}. Motivated by  \eqref{SS:Mellin_1} we define
\begin{equation}\label{def:phat}
\hat{p}(\lambda, \varphi) = \frac{1}{\lambda} \big[ \big((\lambda+1)^2 + \partial_{\varphi}^2\big)\hat{u}_r(\lambda+1,\varphi) -2 \partial_{\varphi} \hat{u}_{\varphi}(\lambda+1,\varphi) -\hat{u}_r(\lambda+1,\varphi) + \hat{f}_r(\lambda-1, \varphi) \big].
\end{equation}
Note that by definition $ \hat{u}_{\varphi} $, $ \hat{u}_{r} $ and $\hat{p}$ are candidate solutions to \eqref{SS:Mellin} rather than being the Mellin transform of some $\uph, \ur$ and $p$ solving a problem in polar coordinates. In fact, below $\uph, \ur$ and $p$ will be recovered as the inverse Mellin transform of $ \hat{u}_{\varphi} $, $ \hat{u}_{r} $ and $\hat{p}$.

Under the hypothesis that  $ \hat{u}_{\varphi} $ and  $ \hat{u}_{r} $ satisfy the regularity estimate \eqref{est:skopel}, we first show that \eqref{SS:Mellin} is satisfied. Since \eqref{SS:Mellin_1}, \eqref{SS:Mellin_3}, \eqref{SS:Mellin_4} and \eqref{SS:Mellin_5} hold by construction, it remains to show that  $ \hat{u}_{\varphi} $, $ \hat{u}_{r} $ and $ \hat{p} $ solve \eqref{SS:Mellin_2}. Multiply \eqref{SS:Mellin_2} by $ \phi \in C^{\infty}_{\mathrm{c}}((0,\theta)) $ and integrate over $ (0,\theta) $. Integration by parts on the term involving the pressure gives
\begin{equation*}
\int_0^{\theta} (\partial_{\varphi} \hat{p} )\phi \, \dd \varphi = - \int_0^{\theta} \hat{p}  \partial_{\varphi} \phi \, \dd \varphi.
\end{equation*}
Then using the definition of $ \hat{p} $, integration by parts and using the definition of $\hat{u}_{r} $, we learn that  $ \hat{u}_{\varphi} $, $ \hat{u}_{r} $ and $ \hat{p} $ satisfy  \eqref{SS:Mellin_2} if and only if $ \hat{u}_{\varphi} $ is a weak solution to \eqref{fourth:order:equ}. The latter condition is satisfied by the definition of $ \hat{u}_{\varphi} $.% which holds true by definition of $ \hat{u}_{\varphi} $.

To recover $ \vu $ and $ p $ it is enough to notice that under the regularity estimate \eqref{est:skopel}, we can invert the Mellin transform for any fixed $ M $. We now show that for $ \widetilde{M} \in \NN $ such that $ \widetilde{M} +1 +\alpha \in I_{\eps} $,  if  $ \vf \in \mathcal{H}^{\tilde{M}}_{\alpha} $ and $ \mathfrak{g} \in  \prescript{1}{}{\mathcal{B}}^{\widetilde{M}+1}_{\alpha+1} $   we recover a unique solution. We know that   $ \hat{\vu} \in \hat{H}^{M+2}_{\alpha} $ for any $ 0 \leq M \leq \tilde{M} $ and therefore by Lemma \ref{isom:space} there exists $ \vu_M \in \prescript{M+2}{}{\mathcal{H}}^{M+2}_{\alpha} $ such that $ \hat{\vu}_M = \hat{\vu}|_{\{\lambda = M + \alpha+1 +i s \} \times (0,\theta) } $ for any $ M \in [0,\tilde{M}] $. It remains to verify that $ \vu_{0} = \dots = \vu_{\tilde{M}  } $.  Recall that
 \begin{equation*}
 \vu_M(r, \varphi) = \frac{1}{\sqrt{2\pi}}\int_{\Re\lambda=M +\alpha+1  }r^{\lambda} \widehat{\vu}(\lambda, \varphi) \dd \Im\lambda.
\end{equation*}
and that $ \hat{\vu} $ is defined via the integral representation in Corollary \ref{form:sol} and \eqref{u:r:espression}. By using  that $ G $, $ \partial_{\varphi'} G$ and $ \partial^2_{\varphi'} G$ are holomorphic on the strip $\Sigma:= \{ \lambda = \alpha + i s:  (|\alpha| +1)\theta < \pi, s\in\RR \}$ (Lemma \ref{prop:4}) and employing a standard density argument, we can move the line of integration. In particular, for any $ M \in [0, \tilde{M}] $, we have
\begin{align*}
 \vu_M (r, \varphi)  = & \,  \frac{1}{\sqrt{2\pi}}\int_{\Re\lambda=M +\alpha+1  }r^{\lambda} \widehat{\vu}(\lambda, \varphi) \dd \Im\lambda \\
= & \,  \frac{1}{\sqrt{2\pi}}\int_{\Re\lambda= \alpha +1  }r^{\lambda} \widehat{\vu}(\lambda, \varphi) \dd \Im\lambda = \vu_0 (r, \varphi) =: \vu(r, \varphi) .
\end{align*}
To recover $p$ from \eqref{def:phat}, we again use the properties of $ \hat{\vu} $ and the Green's function $G$ to see that we can move the line of integration except across zero due to the presence of the singular term $ 1/\lambda $. By the residue theorem this singularity corresponds to a constant.
More precisely, if $ \alpha > 0 $ or $  \tilde{M} + \alpha <  0 $, then we can define
\begin{equation*}
 p(r,\ph) :=  \,  \frac{1}{\sqrt{2\pi}}\int_{\Re\lambda= \ell + \alpha }r^{\lambda} \hat{p}(\lambda, \varphi) \dd \Im\lambda  \qquad  \text{ with }\ell\in[0,\tilde{M}].
 \end{equation*}
In this way we find $ p \in \prescript{1}{}{\mathcal{H}}^{\tilde{M}+1}_{\alpha} $.
If $ \alpha < 0 $ and $  \tilde{M} + \alpha >  0 $, there exists a natural number $ \ell_* \in [0, \tilde{M}-1]$ such that $ \alpha + \ell_* < 0 <  \alpha + \ell_* +1  $. For $ \ell \in [0, \ell_*] $ and $ k \in [\ell_*+1, \tilde{M}]$, we define
\begin{align*}
 p(r,\ph) := & \,  \frac{1}{\sqrt{2\pi}}\int_{\Re\lambda= \ell + \alpha  }r^{\lambda} \hat{p}(\lambda, \varphi) \dd \Im\lambda   \\
 = & \,  \frac{1}{\sqrt{2\pi}} \int_{\Re\lambda=k +\alpha }r^{\lambda} \hat{p}(\lambda, \varphi) \dd \Im\lambda  - C_{\text{res}}.
\end{align*}
We have that $  p \in \prescript{1}{}{\mathcal{H}}^{\ell_*+1}_{\alpha} $, while $  p + C_{\text{res}} \in \prescript{\ell_*+ 2}{}{\mathcal{H}}^{\tilde{M}+1}_{\alpha} $. After noticing that $ \zeta C_{\text{res}}  \in \prescript{1}{}{\mathcal{H}}^{\ell_*+1}_{\alpha}  $ and $ (1-\zeta) C_{\text{res}}  \in \prescript{\ell_*+ 2}{}{\mathcal{H}}^{\tilde{M}+1}_{\alpha} $, we decompose $p=\zeta p_0 +p_1$, where
$$ p_1 := p + \zeta  C_{\text{res}} \quad \text{ and } \quad  p_0: = - C_{\text{res}}.$$
With this choice $ p_1 \in \prescript{1}{}{\mathcal{H}}^{\tilde{M}+1}_{\alpha}$, in fact $ p_1 = p + \zeta  C_{\text{res}} $ is the sum of $ p $ and $ \zeta  C_{\text{res}} $ that are elements of $ \prescript{1}{}{\mathcal{H}}^{\ell_*+1}_{\alpha} $. At the same time,  $ p_1 = p + C_{\text{res}} - (1- \zeta) C_{\text{res}} $ and both $  p + C_{\text{res}} $ and  $ (1- \zeta) C_{\text{res}} $ are elements of $ \prescript{\ell_*+ 2}{}{\mathcal{H}}^{\tilde{M}+1}_{\alpha}  $.

We continue with the regularity of $ \hat{u}_{\varphi} $ and  $ \hat{u}_{r} $. By the linearity of the equation  $ \vu = \vu^b + \vu^s $, where $ \vu^b $ satisfies \eqref{SS:Mellin} with source term $ \vf = 0 $  while $ \vu^s $ satisfies \eqref{SS:Mellin} with zero boundary conditions $ \mathfrak{g} = 0 $. Then  using this decomposition, Corollary \ref{cor:est:bound:term} and  \ref{cor:est:source:term}, we deduce
\begin{align*}
\llbracket \vu \rrbracket_{M+2,\alpha}^2  \leq & \,  \llbracket \vu^b \rrbracket_{M+2,\alpha}^2 + \llbracket \vu^s \rrbracket_{M+2,\alpha}^2 \\
\leq & \, \frac{C_{\eps,M}}{\min\{1, |\alpha+M+1|\}^{M+2}}[ \mathfrak{g} ]_{M+\frac{1}{2},\alpha+1} \\
& \, +\frac{C_{\eps,M}}{\min\{1, |\alpha+M+1|\}^{2}}\frac{1}{\min\{1, |\alpha+M-1|\}^{M}} \llbracket \vf \rrbracket_{M,\alpha}.
\end{align*}

We finish the proof by showing that if $ \mathfrak{g} = \pm (g - r u_{r})  $ with $ \vu $ from Theorem \ref{Floris:theo}, then the solution defined via the Green's function coincides with $ \vu $. Note that $ \vu $ satisfies \eqref{sys:1}, in particular $ \mathfrak{g} = \pm (g - r u_{r}) = \pm (g + \partial_{\varphi}u_r)  $. Using this last expression in the estimates for the solution defined via the Green's functions, we deduce that the solution given by the Green's function is in $ \prescript{1}{}{\mathcal{H}}^2_{\alpha}$ as the solution $ \vu $ from Theorem \ref{Floris:theo}. Uniqueness of  $ \prescript{1}{}{\mathcal{H}}^2_{\alpha}$ solutions for the system \eqref{SsS:merlin} implies that $ \vu $ coincides with the solution given by the Green's function representation  with $ \mathfrak{g} = \pm (g - r u_{r}) $.
\end{proof}

\section{Proof of Proposition \ref{pro:exi:pol}}\label{sec:proof_210}

In this section we prove Proposition \ref{pro:exi:pol} from Section \ref{sec:pol_problem} concerning the polynomial problem.
\begin{proof}[Proof of Proposition \ref{pro:exi:pol}] Insert in \eqref{sto:sys:pol}
 the polynomial expansions
\begin{equation*}
\PPu^n(r,\ph) = \sum_{j= 0}^{n} \vu^{(j)}(\ph) \; r^j,  \quad  \PPp^n(r,\ph) = \sum_{j= 0}^{n-1} p^{(j)}(\ph) \; r^j \quad \text{ and } \quad \PPf^n(r,\ph) = \sum_{j= 0}^{n-2} \vc{f}^{(j)}(\ph) \; r^j,
\end{equation*}
to obtain for $ j \in \mathbb{N} $
\begin{subequations}\label{SS:Pol}
\begin{alignat}{5}
(j^2 + \partial_{\varphi}^2)u^{(j)}_r - 2 \partial_{\varphi} u^{(j)}_{\varphi} -u^{(j)}_r - (j-1)  p^{(j-1)} = \, & - f^{(j-2)}_r \quad && \text{ in } (0,\theta),   \label{SS:Pol1}  \\
(j^2 + \partial_{\varphi}^2)u^{(j)}_{\varphi} + 2 \partial_{\varphi} u^{(j)}_r -u^{(j)}_{\varphi} - \partial_{\varphi} p^{(j-1)} = & \, -f^{(j-2)}_{\varphi} \quad && \text{ in } (0,\theta),    \label{SS:Pol2} \\
(j +1) u^{(j)}_r + \partial_{\varphi} u^{(j)}_{\varphi} = & \, 0 \quad && \text{ in } (0,\theta), \label{SS:Pol3} \\
u^{(j)}_{\varphi} = & \, 0  \quad && \text{ on } \{0,\theta\},  \label{SS:Pol4}  \\
\partial_{\varphi} u^{(j)}_{r} = & \, \mp u^{(j-1)}_r  \pm \delta_{n,j-1} u^{(n)}_r \quad && \text{ on }  \{0,\theta\}, \label{SS:Pol5}
\end{alignat}
\end{subequations}
where $ \delta_{n,j-1} = 1 $ if $ j = n+1 $ and $ 0 $ else.
Note that this system is the same as \eqref{SS:Mellin} with $ \lambda \in \mathbb{C}$ replaced by $ j \in \mathbb{N}$. Hence, for $ j \leq n $, as in Corollary \ref{cor:repr} and in Lemma \ref{lem:Fourier_Rep:STOKES}, we get the solution formula
\begin{align}\label{form:sol:2}
u^{(j)}_{\varphi}(\varphi) = & \, \sum_{k=1}^{\infty}\frac{\hat{\mathfrak{f}}_k^{(j)}}{\big(j^2-\left(\frac{k\pi}{\theta}\right)^2-1\big)^2}\esin_k(\ph)   + (j + 1 )(1-\delta_{n,j-1}) u_{r}^{(j-1)}(\theta)  \partial_{\varphi'}G(j, \varphi, \theta)   \\ &  \, -  (j + 1 )(1-\delta_{n,j-1}) u_{r}^{(j-1)}(0) \partial_{\varphi'}G(j, \varphi, 0),\nonumber
\end{align}
where $ G $ is defined as in Lemma \ref{Lemma:green:fun} (replacing $\lambda$ by $j$) and (cf. \eqref{ftoffrak})
\begin{equation*}
  \hat{\mathfrak{f}}^{(j)}= (j^2-1)\hat{f}_\ph^{(j-2)}(\ph)-(j+1)\dph\hat{f}_r^{(j-2)}(\ph).
\end{equation*}
For $ j = 0 $, we have $ \vf^{(-2)} =  0 $ and $ \vu^{(-1)} = 0 $ by assumption. By \eqref{SS:Pol3} and \eqref{SS:Pol4} we obtain $\vu^{(0)}=0$ and by \eqref{SS:Pol1} we also obtain $ p^{(-1)} = 0 $. For $ j = 1 $, we have
$ \vf^{(-1)} = 0 $ and $ \vu^{(0)} = 0 $. By  \eqref{form:sol:2}, we deduce $ \vu^{(1)} = 0 $. Rewriting \eqref{SS:Pol2} gives
\begin{equation*}
\partial_{\varphi} p^{(0)} =  (1 + \partial_{\varphi}^2)u^{(1)}_{\varphi} + 2 \partial_{\varphi} u^{(1)}_r -u^{(1)}_{\varphi} + f^{(-1)}_{\varphi} = 0.
\end{equation*}
We deduce that $ p^{(0)} $ is a constant. 

Notice that for $ j \geq 2  $, the integral of \eqref{SS:Pol1} over $ (0, \theta) $ rewrites in the form  %\todo{check?}
\begin{align}
\label{pressure:const}
\int_0^{\theta} p^{(j-1)}  \dd\varphi = \int_{0}^{\theta}  \frac{f_{r}^{(j-2)}}{j-1} \dd\varphi + \frac{1}{j-1}\left[\mp u^{(j-1)}_r  \pm \delta_{n,j-1} u^{(n)}_r  \right]_0^{\theta},
\end{align}
%
%\begin{align}
%\label{pressure:const}
%\int_0^{\theta} p^{(j-1)}  \dd\varphi = \int_{0}^{\theta} \frac{(j+1)^2 }{j-1}u_{r}^{(j)} + \frac{f_{r}^{(j-2)}}{j-1} + \frac{1}{j-1}\left[\mp u^{(j-1)}_r  \pm \delta_{n,j-1} u^{(n)}_r  \right]_0^{\theta},
%\end{align}
%
after integration by parts and using \eqref{SS:Pol3} and \eqref{SS:Pol5}.

For $ 2 \leq j \leq n $  we argue by induction. Suppose that we have already found $ \vu^{(j-1)} $ and $ p^{(j-2)} $, then formula \eqref{form:sol:2} defines $ u^{(j)}_{\varphi}$. By Lemmata \ref{lem:Bound:Green_Rep:STOKES} and \ref{lem:bound:uph:stokes:CL}, we have $ u^{(j)}_{\varphi} \in H^{M+2}(0,\theta) $ with the desired bound. With \eqref{SS:Pol3} we define $ u_{r}^{(j)} = \partial_{\varphi}u_{\varphi}^{(j)}/(j+1) \in H^{M+1}(0,\theta)$. From \eqref{SS:Pol2} we deduce the pressure $ p^{(j-1)} \in H^{M+1}(0,\theta) $ up to a constant. Equation \eqref{pressure:const} defines uniquely such a constant. We deduce from \eqref{SS:Pol1} that $ u_r^{(j)} \in H^{M+2}(0,\theta)$.

For $ j = n+1 $ the right-hand side of \eqref{SS:Pol} is identically zero since by assumption $ \vf^{(n-1)} = 0 $ and $  \mp u^{(n)}_r  \pm u^{(n)}_r  = 0 $. Therefore, the unique solution is $ \vu^{(n+1)} = 0$, $ p^{(n)} = 0 $.
Similarly, for $  j \geq n+2 $ , the right-hand side of \eqref{SS:Pol} is identically zero because $ \vf^{(j)} = 0 $ and $  u^{(j-1)}_r   = 0 $ by induction.

Finally, the estimates on $ \PPu^n $ and $ \PPp^n $ are a consequence of \eqref{form:sol:2} and Lemmata \ref{lem:Bound:Green_Rep:STOKES} and \ref{lem:bound:uph:stokes:CL}.
\end{proof}

\appendix

\section{Vector identities and polar coordinates}\label{app:polar_vec}

\subsection{Vector identities}\label{app:vec}
For sufficiently smooth $\vu,\vv:\RR^2\supset \Om\to\RR^2$ in Cartesian coordinates $(x_1,x_2)$ we recall the notation
\begin{align*}
    \grad \vu:\grad\vv=\sum_{i,j=1}^2\partial_{x_i}u_j\partial_{x_i}v_j,\quad
    \vu\otimes\vv\in\RR^{2\times 2}\text{ with }(\vu\otimes\vv)_{ij}:=u_iv_j \text{ for }1\leq i,j\leq2.
\end{align*}
Moreover, recall that in two dimensions the curl is defined as $\om_{\vu}:=\curl\vu=\partial_{x_1} u_2-\partial_{x_2} u_1$.
\begin{lem}\label{lem:app_curl} For a sufficiently smooth vector field $\vu:\RR^2\supset \Om\to\RR^2$ and a sufficiently smooth scalar field $\phi$ we have the following properties:
\begin{enumerate}[label=(\roman*)]
\begin{minipage}{0.4\linewidth}
\item $\curl\grad\phi=0$,
\item $\div\curl\vu = 0$,
\end{minipage}
\begin{minipage}{0.4\linewidth}
\item $\div\del\vu=\del\div\vu$,
\item $\curl\del\vu=\del\curl\vu$ if $\div\vu=0$.
\end{minipage}
\end{enumerate}
Moreover,  we have 
\begin{align*}
    \sum_{j=1}^2\div(v_j\grad u_j)=\vv \cdot\del\vu+\grad\vv:\grad\vu,\quad
    \grad(\phi \vu)=\phi\grad\vu+\vu\otimes\grad\phi.
\end{align*}
\end{lem}
For a vector $\vc{a}=(a_1,a_2)$ the rotated vector is $\vc{a}^{\perp}:=(-a_2,a_1)$. It holds that $\vc{a}^{\perp}\cdot\vc{a}=0$ and $\vc{a}^{\perp}\cdot\vc{b}=-\vc{a}\cdot\vc{b}^{\perp}$. For the rotated gradient $\grad^{\perp}$ we have the properties
\begin{equation*}
  \grad^{\perp}\cdot \vu=\curl \vu,\qquad \del\vu = \grad^{\perp}\curl\vu\text{ if }\div\vu=0.
\end{equation*}

\subsection{Polar coordinates}\label{app:polar} In polar coordinates $(r,\ph)$ the unit vectors at $(r,\ph)$ are given by $\ver=(\cos\ph,\sin\ph)$ and $\veph=(-\sin\ph,\cos\ph)$ and we write $\vu(r,\ph) = \ur(r,\ph)\ver + \uph(r,\ph)\veph$ as $\vu=(u_r,\uph)$.
The gradient and Laplace operator are in polar coordinates given by
\begin{equation*}\label{eqapp:laplace_polar}
    \nabla = (\dr)\ver+\left(r^{-1}\dph\right)\veph\quad \text{ and }\quad \del = \dr^2+r^{-1} \dr+r^{-2}\dph^2=r^{-2}\big((r\dr)^2+\dph^2\big).
\end{equation*}
Therefore, for a sufficiently smooth vector field $\vu:\RR^2\supset \Om\to\RR^2$ and a sufficiently smooth scalar field $\phi$ we have
\begin{align}
\operatorname{div} \vu =\grad\cdot\vu =&\;r^{-1}\big((r\dr+1)\ur+\dph\uph\big),\label{eqapp:div_polar}\\
\curl\vu=\grad^{\perp}\cdot\vu=&\;r^{-1}\big((r\dr+1)\uph-\dph\ur\big),\label{eqapp:curl_polar}\\
    \grad \vu =& \;r^{-1}\begin{pmatrix}
    r\dr\ur & \dph\ur-\uph\\
    r\dr\uph & \dph\uph+\ur
    \end{pmatrix},\label{eqapp:grad_polar}\\
    \del \vu =& \;r^{-2}\begin{pmatrix}\big((r\dr)^2+\dph^2\big)\ur-2\dph\uph-\ur\\
    \big((r\dr)^2+\dph^2\big)\uph+2\dph\ur-\uph
    \end{pmatrix}\label{eqapp:del_polar},\\
    \grad\otimes\grad\phi=&\begin{pmatrix}\dr^2\phi&r^{-1}\dph\dr\phi-r^{-2}\dph\phi\\
    r^{-1}\dph\dr\phi-r^{-2}\dph\phi&r^{-2}\dph^2\phi+r^{-1}\dr\phi\end{pmatrix}.\label{eqapp:del_hessian}
\end{align}
Finally, we also have the commutation relations
\begin{equation}\label{eqapp:com_rel}
    r^{\gamma} (r\dr) = (r\dr-\gamma)r^{\gamma}\quad \text{ and }\quad r\dr r^{\gamma} = r^{\gamma}(r\dr+\gamma) \quad \text{ for }\gamma\in\RR.
\end{equation}

\section{Some results on weighted Sobolev spaces}

\label{app:Hardy:traces}

\subsection{Proof of the claim in Remark \ref{rem:1}}

As an application of Hardy's inequality we prove the claim in Remark \ref{rem:1}.  Recall that $\prescript{k}{}{\mathcal{H}}^{k}_{\alpha} $ is the closure of $C_{\mathrm{c}}^{\infty}(\overline{\Om}\setminus\{0\})$ with respect to $\llbracket \cdot\rrbracket_{k,\alpha}$ as defined in Section \ref{subsec:mainresults}.

\begin{lem}
\label{lemma:norm}
Let $ k \in \mathbb{N} $  and $ \alpha \in \mathbb{R}  $ such that $ \alpha + k -1 \neq 0 $ and let
\begin{equation*}
  \mathfrak{H}^{k}_{\alpha} :=  \left\{ \vu: \llbracket \vu \rrbracket_{k-\ell,\alpha+\ell}  <  \infty   \text{ for all } 0 \leq \ell \leq k            \right\},
\end{equation*}
endowed with the norm  $ \| \vu \|_{\mathfrak{H}^k_{\alpha}}^2 =   \llbracket \vu \rrbracket_{0,\alpha+k}^2 + \dots + \llbracket \vu \rrbracket_{k,\alpha}^2   $. Then the inclusion
 $\prescript{k}{}{\mathcal{H}}^{k}_{\alpha} \longrightarrow \mathfrak{H}^{k}_{\alpha}$
is a linear and continuous bijection.
\end{lem}

\begin{proof}
The inclusion being linear and continuous is a direct consequence of Hardy's inequality. To show that it is surjective it suffices to prove that for any $ \vu \in \mathfrak{H}^{k}_{\alpha} $ there exists a sequence $ \vu_{n} \in C^{\infty}_{\mathrm{c}}(\overline{\Omega} \setminus \{0\}) $
such that
$
\llbracket \vu_n - \vu \rrbracket_{k,\alpha} \longrightarrow 0.
$
Define the cut-off function
\begin{equation*}
    \eta\in C_{\mathrm{c}}^{\infty}(\RR)\quad \text{ satisfying }\quad \eta\big|_{[-1,1]}=1\text{  and  }\eta\big|_{\RR\setminus(-2,2)}=0.
\end{equation*}
Furthermore, define for $n \in\NN$
\begin{equation}
\label{def:cut:off}
    \eta_n(r) :=\eta\Big(\tfrac{\log(r)}{n}\Big),\qquad r>0.
\end{equation}
Then  supp$(\eta_{n}) = [e^{-2n}, e^{2n}]$ and $\eta_n(r)\longrightarrow 1$ pointwise as $n\longrightarrow\infty$ for all $r>0$. Moreover,
\begin{equation*}
    (r\dr)^j\eta_n(r)=n^{-j} \eta^{(j)}\left(\tfrac{\log r}{n}\right).
\end{equation*}
For $ \vu \in \mathfrak{H}^{k}_{\alpha} $ it holds that
$
\eta_n \vu \longrightarrow \vu $ in $ \mathfrak{H}^{k}_{\alpha}
$
and for any $n\in\NN$ the vector field $ \eta_n \vu  $ is compactly supported on $\overline{\Omega} \setminus \{0\}$. Therefore, by mollification there exists  $ \vu_{n} \in C^{\infty}_{\mathrm{c}}(\overline{\Omega} \setminus \{0\}) $ such that $ \llbracket \vu_n - \eta_n \vu \rrbracket_{k,\alpha} \leq \| \eta_n \vu -\vu \|_{\mathfrak{H}^k_{\alpha}}$. This sequence $ \vu_n $ satisfies the desired properties.
\end{proof}

\subsection{Trace theorems} Recall from Section \ref{sec:reg_problem} that for $ s \in\RR $ and $ \alpha \in \mathbb{R} $ such that $ s+\alpha -\frac{1}{2} \neq 0 $ the space $ \tilde{H}^s_{\alpha} $ is the closure of $ C^{\infty}_{\mathrm{c}}(\partial \Omega') $ with respect to the norm 
\begin{equation*}
\| g \|_{\tilde{H}^s_{\alpha}}^2 = \sum_{\varphi \in \{0, \theta \}} \int_{\Re \lambda = s+\alpha -\frac{1}{2}}  |\lambda|^{2s} | \hat{g}(\lambda, \varphi )|^2 \, \dd \Im \lambda.
\end{equation*}

\begin{prop}[Trace operators and extensions]
\label{prop:trace}
For $ k \in\NN $ and $ \alpha \in \mathbb{R} \setminus \mathbb{Z} $, there exists a linear continuous trace operator
\begin{equation*}
T:\prescript{k}{}{\mathcal{H}}^{k}_{\alpha} \longrightarrow \prod_{\ell= 0}^{k-1} \tilde{H}^{k- \ell -\frac{1}{2}}_{\alpha}
\end{equation*}
such that $ Tu = (\partial_{\varphi}^{\ell} u |_{\partial \Omega'})_{\ell = 0}^{k-1} $ for any $ u \in C^{\infty}_{\mathrm{c}}(\overline{\Omega} \setminus \{0\}) $. Moreover, $ T $ admits a  linear and continuous right inverse, that we call extension 
\begin{equation*}
L :  \prod_{\ell= 0}^{k-1} \tilde{H}^{k- \ell -\frac{1}{2}}_{\alpha} \longrightarrow \prescript{k}{}{\mathcal{H}}^{k}_{\alpha}.
\end{equation*}

\end{prop}

\begin{proof}

The existence of $ T $ is a direct consequence of H\"older's inequality. In fact for $ \vartheta \in \{0, \theta \}$ and $ \ell \in \{0, 1, \dots, k-1  \} $ it holds
\begin{align*}
|\lambda||\partial_{\varphi}^{\ell} \hat{u}(\cdot, \vartheta)|^2 \leq & \,  \tfrac{|\lambda|}{\theta} \|\partial_{\varphi}^{\ell} \hat{u} \|_{L^2(0,\theta)}^2 + \| \lambda \partial_{\varphi}^{\ell} \hat{u} \|_{L^2(0,\theta)}^2+  \|\partial_{\varphi}^{\ell+1} \hat{u} \|_{L^2(0,\theta)}^2  \\ \leq & \, (\tfrac{1}{ \theta\alpha} + 1)\| \lambda \partial_{\varphi}^{\ell} \hat{u} \|_{L^2(0,\theta)}^2+  \|\partial_{\varphi}^{\ell+1} \hat{u} \|_{L^2(0,\theta)}^2.
\end{align*}
This implies that
\begin{equation*}
\|\partial_{\varphi}^{\ell}u\|_{\tilde{H}^{k-\ell-\frac{1}{2}}_{\alpha}} \leq C_{\alpha,\theta} \| u \|_{\prescript{k}{}{\mathcal{H}}^{k}_{\alpha}}.
\end{equation*}

The existence of the right inverse is a consequence of the following explicit formula.
For $ (u_1,\dots, u_{k-1})  \in  \prod_{\ell= 0}^{k-1} \tilde{H}^{k- \ell -\frac{1}{2}}_{\alpha} $, define the extension $ U $ in Mellin variables as
\begin{align}
\hat{U} (\lambda, \varphi) = & \,  \sum_{\ell = 0 }^{k - 1}  \frac{\varphi^{\ell}}{\ell!} \hat{u}_{\ell}(\lambda, 0) h\big(\tfrac{\varphi(1+ |\lambda | )}{\theta}\big) \nonumber \\
& +  \sum_{\ell = 0 }^{k - 1}  (-1)^{\ell} \frac{(\theta - \varphi)^{\ell}}{\ell!} \hat{u}_{\ell}(\lambda, \theta) h\big(\tfrac{(\theta-\varphi)(1+ |\lambda | )}{\theta}\big), \label{inverse:trace}
\end{align}
where $ h \in C^{\infty}_{\mathrm{c}}([0,1)) $ satisfies $ 0 \leq h \leq 1 $ and $ h = 1 $ in an open neighbourhood of $ 0 $.
Upon noting that
$$ \int_{0}^{\theta}  h\big(\tfrac{\varphi(1+ |\lambda | )}{\theta}\big)\dd \varphi   \lesssim \frac{\theta}{|\lambda|}, $$
it is straightforward to see that \eqref{inverse:trace} indeed defines a linear and continuous right inverse $L$ to the trace operator $ T $.
%Notice that the formula does not give a compactly supported function but it is easy to show that such a function is in the space.
\end{proof}

\subsection{Density results}
Recall that
\begin{align*}
  \TT &:= \big\{\vc{v}\in C_{\mathrm{c}}^{\infty}(\overline{\Om}\setminus\{0\}):  \div\vv=0\text{ in }\Omega\text{ and } v_{\ph}=0\text{ on }\dOm'\big\}.
\end{align*}

\begin{prop}
\label{prop:ind:spaces}
Let $ \alpha \in \mathbb{R} \setminus \mathbb{Z} $, then the inclusions
\begin{align*}
I_1:  &\; \overline{\TT}^{\prescript{2}{}{\mathcal{H}}^{2}_{\alpha}}   \longrightarrow  \big\{ \vu \in \prescript{2}{}{\mathcal{H}}^{2}_{\alpha}: \div \vu =0 \text{ in }\Om \text{ and } u_{\varphi} = 0 \text{ on }\dOm'\big\}\quad \text{ and }\\
I_2:  &\; \overline{\{ \vv \in \TT :   \vv = 0 \text{ on } \partial \Omega' \}}^{\prescript{2}{}{\mathcal{H}}^{2}_{\alpha}}   \longrightarrow  \big\{ \vu \in \prescript{2}{}{\mathcal{H}}^{2}_{\alpha}: \div \vu =0 \text{ in }\Om \text{ and } \vu = 0 \text{ on }\dOm'\big\}
\end{align*}
are linear and bijective.
\end{prop}

\begin{proof}
We prove the statement for $I_1$ as the proof for $I_2$ is similar. Let $ \vv \in \prescript{2}{}{\mathcal{H}}^{2}_{\alpha} $ satisfy $  \div \vv =0 $ in $ \Om $ and $ v_{\varphi} = 0 $ on $ \dOm'  $. Then it suffices to show that there exists a sequence $ (\vv_k)_{k\geq 1} \in \TT$ such that
$ \vv_k \longrightarrow \vv $ in $ \prescript{2}{}{\mathcal{H}}^{2}_{\alpha}$ as $k\longrightarrow\infty$.
The stream function is given by
$$ \psi(r,\varphi) = \int_{0}^{\varphi} r v_{r}(r, \tilde{\ph}) \dd\tilde{\ph} $$
and satisfies $r\dr \psi = r\uph$ and $\dph \psi = -r\ur$. Moreover, we have $\psi=0$ on $\dOm'$. Indeed, by definition it is clear that $\psi(r,0)=0$. Using that $ v_{\ph}(r,\theta) = 0 $ gives  $ r \partial_r \psi ( r, \theta ) = 0 $, i.e., $ \psi(r, \theta) $ is constant. By applying the trace operator $ T $ to $ \psi $ we obtain $ \psi |_{\partial \Omega'}  \in \tilde{H}^{\frac{3}{2}}_{\alpha} $, which implies that the constant is zero.
In addition, we have $ \psi \in \prescript{3}{}{\mathcal{H}}^{3}_{\alpha} $ since by Lemma  \ref{lemma:norm}
\begin{align*}
  \|\psi\|^2_{\prescript{3}{}{\mathcal{H}}^{3}_{\alpha}}\sim \|\psi\|^2_{\mathfrak{H}^3_{\alpha} }& = \sum_{0 \leq j +  \ell \leq 3 }\int_{0}^{\theta}\int_0^{\infty} r^{-2\alpha - 6}|(r \partial_r)^j\partial_{\varphi}^{\ell} \psi|^2 r \, \dd r \dd\varphi\\& \lesssim \sum_{0 \leq j +  \ell \leq 2}\int_{0}^{\theta}\int_0^{\infty} r^{-2\alpha - 4}|(r \partial_r)^j\partial_{\varphi}^{\ell} \vv|^2 r  \dd r \dd\varphi \lesssim \llbracket \vv \rrbracket_{2, \alpha }^2.
\end{align*}

By definition of $ \prescript{3}{}{\mathcal{H}}^{3}_{\alpha} $, there exists a sequence $ (\psi_{k})_{k \geq 1} $ of smooth compactly supported functions such that $ \psi_{k} \longrightarrow \psi $ in $  \prescript{3}{}{\mathcal{H}}^{3}_{\alpha}  $ as $k\longrightarrow \infty$. With the extension operator $L$ from Proposition \ref{prop:trace}, we define $ \tilde{\vv}_k \in C^{\infty}(\overline{\Om}) $ such that  $ \div \tilde{\vv}_k = 0$, $(\tilde{\vv}_k)_\ph=0$ on $\dOm'$ and %\todo{why the L operator?}
\begin{equation}
\label{def:app:dd:ff}
 \tilde{\vv}_k: = \nabla^{\perp} (\psi_k - L(\psi_k, 0,\dots,0 ))  \longrightarrow \vv \qquad \text{ in }  \prescript{2}{}{\mathcal{H}}^{2}_{\alpha} \text{ as }k\longrightarrow \infty.
\end{equation}
 However, $ \tilde{\vv}_k $ is not compactly supported in $ \overline{\Omega} \setminus \{0\}$. To solve this, consider the cut-off functions $ \eta_k $ as defined in \eqref{def:cut:off} and let $ g(k) $ be an increasing sequence such that
\begin{equation*}
\|\nabla^{\perp} \big( \eta_{g(k)} (\psi_k - L(\psi_k, 0,\dots,0 )) \big) -  \tilde{\vv}_k \|_{\prescript{2}{}{\mathcal{H}}^{2}_{\alpha}} \leq \| \tilde{\vv}_k - \vv  \|_{\prescript{2}{}{\mathcal{H}}^{2}_{\alpha}}\qquad \text{for all }k\geq 1.
\end{equation*}
Then, the sequence $(\vv_k)_{k\geq 1}$ defined by
\begin{equation*}
\vv_k := \nabla^{\perp} \big( \eta_{g(k)} (\psi_k - L(\psi_k, 0,\dots,0 )) \big)
\end{equation*}
has all the desired properties.
\end{proof}
\begin{prop}\label{prop:densH^2alpha}
For $ \alpha \in \mathbb{R} \setminus \mathbb{Z} $ and $\theta\in (0,\frac{\pi}{2})$ the inclusions
\begin{align*}
I_1 :  &\; \overline{\TT}^{\mathscr{H}^2_{\alpha}}   \longrightarrow  \big\{ \vu: \div \vu =0 \text{ in }\Om, u_{\varphi} = 0 \text{ on }\dOm'\text{ and }\| \vu \|_{\mathscr{H}^{2}_{\alpha} } < \infty \big\}\quad \text{ and }\\
I_2:&\; \overline{\TT}^{\mathfrak{X}^2_{\alpha,\theta}}   \longrightarrow  \big\{ \vu: \div \vu =0 \text{ in }\Om, u_{\varphi} = 0 \text{ on }\dOm'\text{ and }\| \vu \|_{\mathfrak{X}^{2}_{\alpha,\theta} } < \infty \big\}
\end{align*}
are linear and bijective.
\end{prop}

\begin{proof} We prove the statement for $I_1$ and the proof for $I_2$ is similar.
We follow the proof of Proposition \ref{prop:ind:spaces}. In particular, we have that the stream function satisfies $\psi\in \prescript{2}{}{\mathcal{H}}^{3}_{\alpha}$ and $ \partial_{\varphi} \psi|_{\partial \Omega'} \in \tilde{H}^0_{\alpha-1} $. This leads to the extra difficulty of finding a smooth, compactly supported  approximation of $ \psi $ in the norm $\prescript{2}{}{\mathcal{H}}^{3}_{\alpha}$ and  $ \partial_{\varphi} \psi|_{\partial \Omega'}$ in $ \tilde{H}^0_{\alpha-1} $. Let $ \eta_n $ again be the cut-off function as defined in \eqref {def:cut:off}. Note that $ \eta_{n} \psi $ converges to $ \psi $ in $\prescript{2}{}{\mathcal{H}}^{3}_{\alpha}$ and $ \partial_{\varphi}( \eta_{n} \psi) $ to $ \partial_{\varphi} \psi $  in $ \tilde{H}^0_{\alpha-1} $ as $n\longrightarrow \infty$.
Then  $ \eta_{n} \psi  $ is compactly supported away from 0, so we can find an approximating sequence $ \psi_n^k\in C^{\infty}_{\mathrm{c}}(\Omega_{2n}) $ such that $ \psi_n^k \longrightarrow \eta_n \psi $ in $ \prescript{3}{}{\mathcal{H}}^{3}_{\alpha} $ as $ k \longrightarrow \infty, $ where
$$ \Omega_{2n} = \{ (r, \varphi) \in \Omega:  r \in(e^{-4n}, e^{4n})  \}.$$
By the trace theorem (Proposition \ref{prop:trace}) and the fact that the weights $r^{\beta}$ for any $\beta\in\RR$ are equivalent on compact subsets of $ \Omega \setminus \{0\} $, it follows that $ \|f\|_{\prescript{2}{}{\mathcal{H}}^{2}_{\alpha}} +  \| \partial_{\varphi} f \|_{\tilde{H}^0_{\alpha-1}} \leq C_{2n }\|f\|_{\prescript{3}{}{\mathcal{H}}^{3}_{\alpha}} $ for any $ f \in C^{\infty}_{\mathrm{c}}(\Omega_n) $. Therefore, $ \psi_n^k $ and $ \partial_{\varphi} \psi_n^k $ converge, respectively, to $ \eta_n \psi $ and $ \partial_{\varphi}(\eta_n \psi ) $ in $ \prescript{2}{}{\mathcal{H}}^{2}_{\alpha} $ and $ \tilde{H}^0_{\alpha-1}$ as $ k \longrightarrow \infty$.  We can now conclude by defining the approximate velocity field as in \eqref{def:app:dd:ff} and proceed from there.
\end{proof}

\section{Auxiliary estimates}\label{app:Aux_est}
In this appendix we prove some auxiliary estimates which are required in Section \ref{sec:higher_reg}. Recall that we defined the interval
$$ I_{\eps} = \left[ -(1-\eps)\tfrac{\pi}{\theta} + 1,(1-\eps) \tfrac{\pi}{\theta}- 1 \right]. $$

\begin{lem}
Let $ \lambda \in\mathbb{C} $. If $\Re \lambda \in I_{\eps} $, $ |\Im \lambda| \theta < \tfrac{1}{4} $, $ |\lambda -1 | \geq \tfrac{1}{4}$, $ |\lambda +1 | \geq \tfrac{1}{4}$ and $ \varphi \in (0, \theta ) $, then we have the estimates
\begin{equation}
\label{key:estimate:luce:pulsata:-1}
\left| \frac{\sin((\lambda+1)\varphi)}{4 \lambda \sin((\lambda+1)\theta)} \pm \frac{\sin((\lambda-1)\varphi)}{4 \lambda \sin((\lambda-1)\theta)}\right| \leq C_{\eps}\frac{1}{|\lambda|},
\end{equation}
\begin{equation}
\label{key:estimate:luce:pulsata}
\left| \frac{\sin((\lambda+1)\varphi)}{4\lambda \sin((\lambda+1)\theta)} - \frac{\sin((\lambda-1)\varphi)}{4 \lambda \sin((\lambda-1)\theta)}\right| \leq C_{\eps} \frac{|\lambda| \theta}{|\lambda|^2},
\end{equation}
\begin{equation}
\label{key:estimate:luce:pulsata:cos}
\left| \frac{(\lambda+1)\cos((\lambda+1)\varphi)}{4\lambda \sin((\lambda+1)\theta)} - \frac{(\lambda-1)\cos((\lambda-1)\varphi)}{4 \lambda \sin((\lambda-1)\theta)}\right| \leq C_{\eps} \theta.
\end{equation}

\end{lem}

\begin{proof}
 Note that in any compact subset of $ \mathbb{C} $ we have $|\sin z| \leq C |z| $, so \eqref{key:estimate:luce:pulsata:-1} follows easily.

For \eqref{key:estimate:luce:pulsata} note that
\begin{align}
\frac{\sin((\lambda+1)\varphi)}{4\lambda \sin((\lambda+1)\theta)} & - \frac{\sin((\lambda-1)\varphi)}{4 \lambda \sin((\lambda-1)\theta)} \label{riscrittura:102}\\
= & \, \frac{\sin((\lambda + 1)\varphi)  \sin((\lambda-1)\theta)-\sin((\lambda - 1)\varphi) \sin((\lambda+1)\theta) }{4\lambda \sin((\lambda+1)\theta)  \sin((\lambda-1)\theta)}. \nonumber
\end{align}
To estimate the numerator, we rewrite using the Taylor expansion theorem with Lagrange remainders:
\begin{equation*}
\sin(x) = x - \int_0^x(x-s) \sin(s) \dd s.
\end{equation*}
Using this formula, we deduce 
\begin{align}
\sin &((\lambda + 1)\varphi)  \sin((\lambda-1)\theta)-\sin((\lambda - 1)\varphi) \sin((\lambda+1)\theta) \nonumber \\
= & \, -(\lambda +1)\varphi \int_0^{(\lambda-1)\theta} ((\lambda-1)\theta - s)\sin(s)  \dd s+(\lambda -1)\varphi \int_0^{(\lambda+1)\theta} ((\lambda+1)\theta - s)\sin(s)  \dd s  \label{monkey:1}\\
& \, -(\lambda -1)\theta \int_0^{(\lambda+1)\varphi} ((\lambda+1)\varphi - s)\sin(s)  \dd s+ (\lambda +1)\theta \int_0^{(\lambda-1)\varphi} ((\lambda-1)\varphi - s)\sin(s)  \dd s \label{monkey:2} \\
& \, + \int_0^{(\lambda+1)\varphi}((\lambda+1)\varphi-s)\sin(s) \dd s \int_0^{(\lambda-1)\theta}((\lambda-1)\theta-s)\sin(s) \dd s  \label{monkey:3} \\
& \,  - \int_0^{(\lambda-1)\varphi}((\lambda-1)\varphi-s)\sin(s) \dd s \int_0^{(\lambda+1)\theta}((\lambda+1)\theta-s)\sin(s) \dd s. \label{monkey:4}
\end{align}
We bound \eqref{monkey:1}, \eqref{monkey:2} and \eqref{monkey:3}+\eqref{monkey:4} separately.
We start by rewriting the first term
\begin{align*}
\eqref{monkey:1} = & \, (\lambda^2-1)\varphi \theta \int_{(\lambda-1)\theta}^{(\lambda +1)\theta} \sin(s)  \dd s - \lambda \varphi \int_{(\lambda+1)\theta}^{(\lambda-1)\theta} s \sin(s)  \dd s+ \varphi \int_{0}^{(\lambda-1)\theta} s \sin(s)  \dd s  \\ & \, + \varphi \int_{0}^{(\lambda+1)\theta} s \sin(s)  \dd s.
\end{align*}
We then deduce
\begin{equation}
|\eqref{monkey:1}| \leq  \, C (|\lambda^2-1| + |\lambda|+ |\lambda-1|^3+|\lambda+1|^3)\theta^4 
\leq  \, C \theta, \label{monkey:q1}
\end{equation}
where in the last inequality we have used that $ |\lambda| \theta \leq |\Re \lambda|\theta+|\Im \lambda|\theta \leq \pi +1 $ and $ \theta \leq 1 $.
The same estimates holds for \eqref{monkey:2}. We are left with
\begin{align*}
\eqref{monkey:3} & +\eqref{monkey:4} = \\
&   \,   (\lambda^2-1)\varphi \theta \Big( \int_0^{(\lambda+1)\varphi} \sin s  \dd s \int_0^{(\lambda-1)\theta} \sin s \dd s - \int_0^{(\lambda-1)\varphi}  \sin s  \dd s \int_0^{(\lambda+1)\theta}  \sin s  \dd s\Big) \\
&  - \, \lambda \theta \Big( \int_0^{(\lambda+1)\varphi} s \sin s  \dd s \int_0^{(\lambda-1)\theta} \sin s  \dd s - \int_0^{(\lambda-1)\varphi} s \sin s  \dd s \int_0^{(\lambda+1)\theta}  \sin s  \dd s\Big) \\
&  -  \,   \lambda\varphi \Big( \int_0^{(\lambda+1)\varphi} \sin s  \dd s \int_0^{(\lambda-1)\theta} s \sin s  \dd s - \int_0^{(\lambda-1)\varphi}  \sin s  \dd s \int_0^{(\lambda+1)\theta} s \sin s  \dd s\Big) \\
&   \, +  \theta \Big( \int_0^{(\lambda+1)\varphi} s \sin s  \dd s \int_0^{(\lambda-1)\theta} \sin s  \dd s + \int_0^{(\lambda-1)\varphi} s \sin s  \dd s \int_0^{(\lambda+1)\theta}  \sin s  \dd s\Big) \\
&   \, + \varphi \Big( \int_0^{(\lambda+1)\varphi} \sin s  \dd s \int_0^{(\lambda-1)\theta} s \sin s  \dd s + \int_0^{(\lambda-1)\varphi}  \sin s  \dd s \int_0^{(\lambda+1)\theta} s \sin s  \dd s\Big) \\
& \, + \Big( \int_0^{(\lambda+1)\varphi} s \sin s  \dd s \int_0^{(\lambda-1)\theta} s \sin s  \dd s - \int_0^{(\lambda-1)\varphi} s \sin s  \dd s \int_0^{(\lambda+1)\theta} s \sin s  \dd s\Big)
\end{align*}
Note
\begin{align*}
\Big|  \int_0^{(\lambda+1)\varphi}  \sin s  \dd s \int_0^{(\lambda-1)\theta}  \sin s  \dd s - \int_0^{(\lambda-1)\varphi}  \sin s  \dd s \int_0^{(\lambda+1)\theta}  \sin s  \dd s  \Big| \leq C|\lambda - 1 |^2\theta^4
\end{align*}
and
\begin{align*}
\Big|  \int_0^{(\lambda+1)\varphi} \sin s  \dd s \int_0^{(\lambda-1)\theta} s \sin s  \dd s + \int_0^{(\lambda-1)\varphi} \sin s  \dd s \int_0^{(\lambda+1)\theta} s \sin s  \dd s  \Big| &  \\ \leq C(|\lambda + 1 |+|\lambda-1|)&|\lambda^2-1|^2\theta^5,
\end{align*}
which implies
\begin{align}
|\eqref{monkey:3}+\eqref{monkey:4}| \leq & \, C (|\lambda^2-1||\lambda-1|^2+|\lambda||\lambda-1|^3+ (|\lambda + 1 |+|\lambda-1|)|\lambda^2-1|^2+|\lambda-1|^3 )\theta^6  \nonumber \\
\leq & \, C \theta, \label{monkey:q3}
\end{align}
where in the last inequality we have used that $ |\lambda| \theta \leq |\alpha|\theta+|\Im \lambda|\theta \leq \pi +1 $ and $ \theta \leq 1 $.
Inequality \eqref{monkey:q1} implies the result \eqref{monkey:q3}.

By using \eqref{monkey:q1} and \eqref{monkey:q3} to estimate the numerator of \eqref{riscrittura:102} and the fact that $$ |\lambda \sin((\lambda-1)\theta) \sin((\lambda+1)\theta)| \geq c_{\eps}^2 |\lambda||\lambda^2-1|\theta^2, $$ we deduce \eqref{key:estimate:luce:pulsata}.
A similar strategy proves \eqref{key:estimate:luce:pulsata:cos}.
\end{proof}

\bibliographystyle{plain}
\bibliography{Well-posednessStokesonwedgewithNavierslip}

\end{document}